\documentclass[CM,Submssn,SecEq]{stability-1}  

\firstpage{}\volume{}\volumeyear{}\copyrightyear{}\doiyear{}\doi{}\received{XXX}\revised{XXX}

\title[Stability of the conical K\"ahler-Ricci flows on Fano manifolds]{Stability of the conical K\"ahler-Ricci flows on Fano manifolds}

\author{Jiawei}{Liu}{F.~Lastname}{Magdeburg}
\author{Xi}{Zhang}{F.~Lastname}{Hefei}

\contact{Institut f\"ur Analysis und Numerik, Otto-von-Guericke-Universit\"at Magdeburg, universit\"atsplatz 2, 39106, Magdeburg, Germany}{jiawei.liu@ovgu.de}
\contact{Key Laboratory of Wu Wen-Tsun Mathematics, Chinese Academy of Sciences, and School of Mathematical Sciences,
University of Science and Technology of China, JinZhai Road, 230026, Hefei, China}{mathzx@ustc.edu.cn}

\researchsupported{The first author is supported by the Special Priority Program SPP 2026 ``Geometry at Infinity" from the German Research Foundation (DFG), and the second author is supported by NSF in China No.11625106, 11571332 and 11721101.}

\theoremstyle{plain}
 \newtheorem{thm}{Theorem}[section]
\newtheorem{lem}[thm]{Lemma}

\newtheorem{rem}[thm]{Remark}
\newtheorem{pro}[thm]{Proposition}
\newtheorem{defi}[thm]{Definition}

\begin{document}

\begin{abstract}
In this paper, we study the stability of the conical K\"ahler-Ricci flows on Fano manifolds. That is, if there exists a conical K\"ahler-Einstein metric with cone angle $2\pi\beta$ along the divisor, then for any $\beta'$ sufficiently close to $\beta$, the corresponding conical K\"ahler-Ricci flow converges to a conical K\"ahler-Einstein metric with cone angle $2\pi\beta'$ along the divisor. Here, we only use the condition that the Log Mabuchi energy is bounded from below. This is a weaker condition than the properness that we have adopted to study the convergence in \cite{JWLXZ,JWLXZ1}. As corollaries, we give parabolic proofs of Donaldson's openness theorem \cite{SD2} and his existence conjecture \cite{SD2009} for the conical K\"ahler-Einstein metrics with positive Ricci curvatures. 
\end{abstract}

%%\tableofcontents %% Just for papers exceeding 50 pages.

\section{Introduction}\label{sec:intro}

Since the conical K\"ahler-Einstein metrics play an important role in the solution of the Yau-Tian-Donaldson's conjecture which has been proved by Chen-Donaldson-Sun \cite{CDS1,CDS2,CDS3} and Tian \cite{T1}, the existence and geometry of the conical K\"ahler-Einstein metrics have been widely concerned. The conical K\"ahler-Einstein metrics have been studied by Berman \cite{RB}, Brendle \cite{SB}, Campana-Guenancia-P$\breve{a}$un \cite{CGP}, Donaldson \cite{SD2}, Guenancia-P$\breve{a}$un \cite{GP1}, Guo-Song \cite{GuoSong1, GuoSong2}, Jeffres \cite{TJEF}, Jeffres-Mazzeo-Rubinstein \cite{JMR}, Li-Sun \cite{LS}, Mazzeo \cite{MAZZ}, Song-Wang \cite{SW}, Tian-Wang \cite{TianWang} and Yao \cite{Yao1} etc. For more details, readers can refer to Rubinstein's article \cite{RUBIN}.

The conical K\"ahler-Ricci flows were introduced to attack the existence of the conical K\"ahler-Einstein metrics. These flows were first proposed in Jeffres-Mazzeo-Rubinstein's paper (see Section $2.5$ in \cite{JMR}), then Song-Wang (conjecture $5.2$ in \cite{SW}) made a conjecture on the relation between the convergence of these flows and the greatest Ricci lower bounds of the manifolds. Then the existence, regularity and limit behavior of the conical K\"ahler-Ricci flows have been studied by Chen-Wang \cite{CW,CW1}, Edwards \cite{GEDWA,GEDWA1}, Liu-Zhang \cite{JWLCJZ}, Liu-Zhang \cite{JWLXZ,JWLXZ1}, Nomura \cite{Nomura}, Shen \cite{LMSH1,LMSH2},  Wang \cite{YQW} and Zhang \cite{YSZ,YSZ1} etc.

Let $M$ be a Fano manifold with complex dimension $n$, $\omega_0\in c_1(M)$ be a smooth K\"ahler metric and $D\in|-\lambda K_{M}|$ ($0<\lambda\in \mathbb{Q}$) be a smooth divisor. Assume that the K\"ahler current $\hat{\omega}\in c_{1}(M)$ admits $L^{p}$-density with respect to $\omega_{0}^{n}$ for some $p>1$ and satisfies $\int_{M}\hat{\omega}^{n}=\int_{M}\omega_{0}^{n}$. Let $\gamma\in(0,1)$ and $\mu_\gamma=1-(1-\gamma)\lambda$. The conical K\"ahler-Ricci flows take the following form.
$$(CKRF_{\mu_\gamma}):\ \ \ \ \left\{
\begin{aligned}
 &\ \frac{\partial \omega_{\gamma}(t)}{\partial t}=-Ric(\omega_{\gamma}(t))+\mu_{\gamma}\omega_{\gamma}(t)+(1-\gamma)[D],\\
  \\
 &\ \omega_{\gamma}(t)|_{t=0}=\hat{\omega}\\
\end{aligned}
\right.$$
where $[D]$ is the current of integration along $D$. In \cite{JWLXZ1}, we proved the existence of these flows by using smooth approximation (see also \cite{JWLXZ,YQW} for the stronger initial metrics).  Let $c_{1,\gamma}(M)=c_1(M)-(1-\gamma)[D]$ be the twisted first Chern class. Then $c_{1,\gamma}(M)=\mu_\gamma[\omega_0]$ in the Fano case. When $\mu_\gamma$ is negative or zero, Chen-Wang \cite{CW1} proved that the corresponding conical K\"ahler-Ricci flow converges to a conical K\"ahler-Einstein metric with Ricci curvature $\mu_\gamma$ and cone angle $2\pi\gamma$ along $D$. When $\mu_\gamma>0$ is sufficiently small and $\lambda\geqslant1$, Li-Sun (see section $2.3$ in \cite{LS}, when $\lambda=1$, see also Berman's work \cite{RB} and Jeffres-Mazzeo-Rubinstein's work \cite{JMR}) proved that the Log Mabuchi energy $\mathcal{M}_{\mu_\gamma}$ is proper by using its definition and the property that the Log $\alpha$-invariant is positive. Then the convergence of the conical K\"ahler-Ricci flows $(CKRF_{\mu_\gamma})$ follows from the arguments in \cite{JWLXZ1}. In other $\mu_\gamma>0$ cases, there are obstacles. In \cite{JWLXZ1}, under the assumptions that there exists a conical K\"ahler-Einstein metric with Ricci curvature $\mu_\gamma$ and cone angle $2\pi\gamma$ along $D$ and no nontrivial holomorphic vector fields tangent to $D$, we deduced their convergence.

If there exists a conical K\"ahler-Einstein metric $\omega_{\varphi_\beta}$ with cone angle $2\pi\beta$ along $D$, then $\omega_{\varphi_\beta}$ obtains the minimum of the Log Mabuchi energy $\mathcal{M}_{\mu_\beta}$ (Corollary $2.10$ in \cite{LS}), and thus this energy is bounded from blow. Furthermore, when $\lambda\geqslant1$, Li-Sun (Corollary $1.7$ in \cite{LS})  proved that $\mathcal{M}_{\mu_\beta}$ is proper by using Berman's properness theorem (Theorem $1.5 $ in \cite{RB}) and Donaldson's openness theorem (Theorem $2$ in \cite{SD2}). For $\lambda>0$, Tian-Zhu (Theorem $0.1$ in \cite{GTXHZ05}) proved this property by assuming in addition that there is no nontrivial holomorphic vector fields tangent to $D$. Darvas-Rubinstein solved Tian's properness conjectures and gave more general properness theorems (Theorem $2.12$ in \cite{TDYR}). Under the above assumptions, by using the properness of $\mathcal{M}_{\mu_\beta}$ and the uniform Perelman's estimates, we \cite{JWLXZ,JWLXZ1} proved that the conical K\"ahler-Ricci flows $(CKRF_{\mu_\beta})$ with $\mu_\beta>0$ converge to $\omega_{\varphi_\beta}$. It is worth noting that the properness of $\mathcal{M}_{\mu_\beta}$ is a stronger condition than that $\mathcal{M}_{\mu_\beta}$ is bounded from below.

In this paper, we study the stability of the conical K\"ahler-Ricci flows by only using the weaker condition that the Log Mabuchi energy is bounded form below. We hope that this method can play a positive role in researching the relation between the limit behavior of the (conical) K\"ahler-Ricci flows and the stability of the manifolds, and studying the existence of the K\"ahler-Einstein metrics with cone angle zero by using parabolic method. The first problem is related to a parabolic type of Yau-Tian-Donaldson's conjecture, and the latter one is related to a Tian's conjecture \cite{GT94} that the complete Tian-Yau K\"ahler-Einstein metric on the complement of $D$ should be the limit of the conical K\"ahler-Einstein metrics as the cone angles tend to zero.
\begin{thm} \label{01} Assume that $\lambda >0$ and there is no nontrivial holomorphic vector fields on $M$ tangent to $D$. If $\mu_\beta$ $(0<\beta<1)$  is positive, we further assume that there is a conical K\"ahler-Einstein metric with Ricci curvature $\mu_\beta$ and cone angle $2\pi\beta$ along $D$. Then for any $\beta'$ sufficiently close to $\beta$, the conical K\"ahler-Ricci flow $(CKRF_{\mu_{\beta'}})$ converges to a conical K\"ahler-Einstein metric with Ricci curvature $\mu_{\beta'}$ and cone angle $2\pi\beta'$ along $D$ in $C_{loc}^{\infty}$-topology outside $D$ and globally in $C^{\alpha,\beta'}$-sense for any $\alpha\in(0,\min\{1,\frac{1}{\beta'}-1\})$.\end{thm}
\begin{rem}\label{02}
When $\lambda\geqslant1$, the assumption that there is no nontrivial holomorphic vector fields tangent to $D$ can be removed according to Theorem $1.5$ of Berman \cite{RB} (see also Corollary $2.21$ in \cite{LS}) or Theorem $2.8$ of Song-Wang \cite{SW}.
\end{rem}
As in \cite{JWLXZ1}, we still use the twisted K\"ahler-Ricci flows
$$(TKRF_{\mu_\gamma,\varepsilon}):\ \ \ \ \left\{
\begin{aligned}
&\ \frac{\partial \omega_{\gamma,\varepsilon}(t)}{\partial t}=-Ric(\omega_{\gamma,\varepsilon}(t))+\mu_\gamma\omega_{\gamma,\varepsilon}(t)+(1-\gamma)\theta_{\varepsilon}\\
  \\
&\ \omega_{\gamma,\varepsilon}(t)|_{t=0}=\hat{\omega}\\
\end{aligned}
\right.$$
to study the conical K\"ahler-Ricci flows $(CKRF_{ \mu_\gamma})$, where $\theta_{\varepsilon}=\lambda\omega_{0}+\sqrt{-1}\partial\overline{\partial}\log(\varepsilon^{2}+|s|_{h}^{2})$ are smooth closed positive $(1,1)$-forms, $s$ is the definition section of $D$ and $h$ is a smooth Hermitian metric on $-\lambda K_{M}$ with curvature $\lambda\omega_{0}$. We assume that $|s|_h^2<\frac{1}{2}$ by rescaling $h$.

Let $u_{\gamma,\varepsilon}(t)$ and $u_{\gamma}(t)$ be the twisted Ricci potentials of $\omega_{\gamma,\varepsilon}(t)$ and $\omega_{\gamma}(t)$ with normalization $\frac{1}{V}\int_M e^{-u_{\gamma,\varepsilon}(t)}dV_{\gamma,\varepsilon}(t)=1$ and $\frac{1}{V}\int_M e^{-u_{\gamma}(t)}dV_{\gamma}(t)=1$ respectively. We denote $A_{\mu_\gamma,\varepsilon}(t)$ and $A_{\mu_\gamma}(t)$ be the functionals
\begin{equation*}\begin{split}\label{008}
A_{\mu_\gamma,\varepsilon}(t)&=\frac{1}{V}\int_M u_{\gamma,\varepsilon}(t)e^{-u_{\gamma,\varepsilon}(t)}dV_{\gamma,\varepsilon}(t),\\
A_{\mu_\gamma}(t)&=\frac{1}{V}\int_M u_{\gamma}(t)e^{-u_{\gamma}(t)}dV_{\gamma}(t).
\end{split}
\end{equation*}
When $\mu_\gamma=0$, the functional $A_{\mu_\gamma}(t)$ is well-defined (Theorem \ref{2019021901}) and converges to zero as $t$ tend to $\infty$ (Theorem \ref{20190002}). When $\mu_\gamma>0$, from the uniform Perelman's estimates (Theorem \ref{1.8.2}), the functional $A_{\mu_\gamma}(t)$ is well-defined. Furthermore, $A_{\mu_\gamma,\varepsilon}(t)$ is nondecreasing along $(TKRF_{\mu_\gamma,\varepsilon})$ (Theorem \ref{16}), and $A_{\mu_\gamma}(t)$ converges to $0$ as $t$ tend to $\infty$ when the Log Mabuchi energy $\mathcal{M}_{\mu_\gamma}$ is bounded from below (Theorem \ref{17}).
\begin{rem}
\label{01001001}
In this paper, when considering the convergence for a sequence of twisted K\"ahler-Ricci flows as $t$ tend to $\infty$, we use the functionals $A_{\mu_\gamma,\varepsilon}(t)$ and $A_{\mu_\gamma}(t)$. For a single (twisted) K\"ahler-Ricci flow, after obtaining the uniform $C^\infty$-estimates, we usually use the energy $\|\nabla u(t)\|_{L^2(M)}$ (see Liu-Wang \cite{LW}, Phong-Sturm et al.\cite{PSST,PSSW,PS} and Tian-Zhu \cite{GTXHZ07}) or the (twisted) Perelman's entropy (see Collins-Sz$\acute{e}$kelyhidi \cite{TC} and Tian-Zhu et al.\cite{GTXHZ13,TZZZ}) to prove the convergence. But in our case, we need consider the convergence for a specific choice sequence of twisted K\"ahler-Ricci flows. We do not use $\|\nabla u_{\gamma,\varepsilon}(t)\|_{L^2(M)}$ or the twisted Perelman's entropy because we can not prove that these functionals converge to $0$ uniformly (that is, the convergence depends on $\mu_\gamma$ and $\varepsilon$). In order to overcome these difficulties, we use the functionals $A_{\mu_\gamma,\varepsilon}(t)$ and $A_{\mu_\gamma}(t)$, especially the convergence of $A_{\mu_\gamma}(t)$ and the monotonicity of $A_{\mu_\gamma,\varepsilon}(t)$, to get the convergence we need. Although the twisted Perelman's entropies are also monotonous along the twisted K\"ahler-Ricci flow (see Liu \cite{JWL} and Collins-Sz$\acute{e}$kelyhidi \cite{TC}), we still do not use these entropies because there is no details about them in this conical case. On the manifolds with isolated conical singularity, Dai-Wang \cite{DWL1,DWL2} and Ozuch \cite{Ozuch} developed Perelman's entropies, and recently Kr\"oncke-Vertman \cite{KKBV} improved the regularities for the minimizers of these entropies and studied their monotonicity along the Ricci de Turck flow.
\end{rem}
There are three cases in Theorem \ref{01}, that is, $\mu_\beta$ is negative, zero and positive. The last two cases are difficult.

{\it Case $1$}. If $\mu_\beta$ is negative, $\mu_{\gamma}$ is also negative when $\gamma$ sufficiently close to $\beta$. Then we can prove Theorem \ref{01} without any assumptions because the uniform estimates along the twisted K\"ahler-Ricci flows can be deduced directly (see also Chen-Wang's work \cite{CW1}).

{\it Case $2$}. If $\mu_\beta$ is zero (that is, $\lambda=\frac{1}{1-\beta}>1$ ), then $\mu_{\gamma}$ is a negative or sufficiently small positive number when $\gamma$ close to $\beta$. The case of $\mu_{\gamma}<0$ can be solved as in {\it Case $1$}. There are difficulties when $\mu_{\gamma}>0$, because no direct uniform Perelman's estimates (independent of $\mu_{\gamma}$) can be used (these estimates depend on $\frac{1}{\mu_\gamma}$ from the proof of Proposition $5.5$ in \cite{JWLXZ}). For overcoming these difficulties, we need some new arguments.

First, we prove a uniform lower bound independent of $\mu_{\gamma}$ for the twisted scalar curvature $R(\omega_{\mu_\gamma,\varepsilon}(t))-(1-\gamma)tr_{\omega_{\mu_\gamma,\varepsilon}(t)}\theta_\varepsilon$ along $(TKRF_{\mu_\gamma,\varepsilon})$ when $t\geqslant1$ (Lemma \ref{1902}). This uniform lower bound plays an important role in proving the uniform (Perelman's) estimates (independent of $\mu_{\gamma}$) along $(TKRF_{\mu_{\gamma},\varepsilon})$ for $\gamma$ sufficiently close to $\beta$.

Next, without any assumptions, we prove that there exists uniform constant $D_\beta$ such that along the twisted K\"ahler-Ricci flows $(TKRF_{0,\varepsilon})$,
\begin{equation}\label{20190223}
\|u_{\beta,\varepsilon}(t)\|_{C^0(M)}+osc_{M}\varphi_{\beta,\varepsilon}(t)+\sup\limits_M tr_{\omega_\varepsilon^\beta}\omega_{\beta,\varepsilon}(t) \leqslant D_\beta
\end{equation}
for any $\varepsilon\in(0,\frac{1}{2}]$ and $t\geqslant1$ (Theorem \ref{2019021901}). Here we denote $\varphi_{\gamma,\varepsilon}(t)$ is the metric potential of $\omega_{\gamma,\varepsilon}(t)$ with background metric $\omega_{0}$, and $\omega_\varepsilon^\gamma$ is smooth K\"ahler metric given in equation $(\ref{TKRF9})$. For a single (twisted) K\"ahler-Ricci flow in the case of $\mu_\beta=0$, the bound on the oscillation of metric potential follows from Yau's $C^0$-estimates for the elliptic complex Monge-Amp\`ere equation. The main tools are the Sobolev inequality, the Poincar\'e inequality and the Moser iteration (see Lemma $3$ in \cite{HDC} or Theorem $4.6$ in \cite{JSBW}). Since we consider a sequence of equations and do not have uniform Poincar\'e inequality, we here need modified the original proof when we adopt this method. We remark that this bound can also be deduced from Ko{\l}odeizj's $L^p$-estimates \cite{K000,K}. Then by using estimates $(\ref{20190223})$ and the smooth approximation \cite{JWLXZ,JWLXZ1}, we prove that the conical K\"ahler-Ricci flow $(CKRF_{0})$ converges to a Ricci flat conical K\"ahler-Einstein metric with cone angle $2\pi\beta$ along $D$ (Theorem \ref{20190002}).

At last, by using the continuity of $(TKRF_{\mu_\gamma,\varepsilon})$ with respect to $\mu_\gamma$ and $\varepsilon$ (Remark \ref{201908001}), we deduce the uniform estimates (independent of $\mu_{\gamma}$) along $(TKRF_{\mu_{\gamma},\varepsilon})$ for $\gamma$ sufficiently close to $\beta$. In this process, we mainly use the uniform lower bound of the twisted scalar curvature $R(\omega_{\mu_\gamma,\varepsilon}(t))-(1-\gamma)tr_{\omega_{\mu_\gamma,\varepsilon}(t)}\theta_\varepsilon$, the convergence of $A_{0}(t)$, the monotonicity of $A_{\mu_\gamma,\varepsilon}(t)$ and Dinew's uniqueness theorem (Theorem $1.2$ in \cite{SDINEW}, see also Berndtsson's work \cite{BBERN}). We prove the following lemma.
\begin{lem} \label{1901} Assume that $\lambda>1$ and $\beta=1-\frac{1}{\lambda}$ (that is, $\mu_\beta=0$). Then there exists a constant $\delta(\lambda)>0$ such that
\begin{equation}\label{1903}
\|u_{\gamma,\varepsilon}(t)\|_{C^0(M)}+osc_{M}\varphi_{\gamma,\varepsilon}(t)+\sup\limits_M tr_{\omega_\varepsilon^\gamma}\omega_{\gamma,\varepsilon}(t) \leqslant L_{\beta}+1
\end{equation}
for any $\varepsilon\in(0,\delta(\lambda))$, $\gamma\in(\beta,\beta+\delta(\lambda))$ and $t\in[\frac{1}{\delta(\lambda)},+\infty)$, where $L_{\beta}$ comes from $(\ref{20190304})$.\end{lem}
In \cite{JWLXZ,JWLXZ1}, we prove the uniform Perelman's estimates along the twisted K\"ahler-Ricci flows $(TKRF_{\mu_{\gamma},\varepsilon})$ when $\mu_\gamma>0$. These estimates are important in studying the convergence of the conical K\"ahler-Ricci flows  $(CKRF_{\mu_{\gamma}})$. From the proof of Proposition $5.5$ in \cite{JWLXZ}, we know that these estimates depend on $\frac{1}{\mu_\gamma}$. So they are not uniform for $\mu_\gamma\in(0,1)$. In Lemma \ref{1901}, we have obtained the uniform bound for $\|u_{\gamma,\varepsilon}(t)\|_{C^0(M)}$ when $\mu_\gamma\in(0,\eta(\lambda))$ with $\lambda>1$ and the uniform lower bound for the twisted scalar curvature $R(\omega_{\mu_\gamma,\varepsilon}(t))-(1-\gamma)tr_{\omega_{\mu_\gamma,\varepsilon}(t)}\theta_\varepsilon$, then by using Lemma \ref{1.8.601}, we get the uniform Perelman's estimates independent of $\mu_\gamma$ along $(TKRF_{\mu_{\gamma},\varepsilon})$ when $\lambda>1$.
\begin{thm}\label{20190223001} Let $\omega_{\gamma,\varepsilon}(t)$ be a solution of the twisted K\"ahler Ricci flow $(TKRF_{\mu_\gamma,\varepsilon})$. When $\lambda>1$, there exists uniform constant $C$ such that
\begin{equation}\begin{split}\label{p0}
|R(\omega_{\gamma,\varepsilon}(t))-(1-\gamma)tr_{\omega_{\gamma,\varepsilon}(t)}\theta_{\varepsilon}|&\leqslant C,\\
\|u_{\gamma,\varepsilon}(t)\|_{C^{1}(\omega_{\gamma,\varepsilon}(t))}&\leqslant C,\\
diam(M,\omega_{\gamma,\varepsilon}(t))&\leqslant C
\end{split}
\end{equation}
hold for any $\mu_\gamma\in(0,1]$, $t\geqslant 1$ and $\varepsilon\in(0,\delta(\lambda))$, where $\delta(\lambda)$ is the constant in Lemma \ref{1901} and $diam(M,\omega_{\gamma,\varepsilon}(t))$ is the diameter of the manifold with respect to $\omega_{\gamma,\varepsilon}(t)$.\end{thm}
\begin{rem}\label{20190223002}
When $\lambda=1$ (resp. $0<\lambda<1$), we can not get the uniform Perelman's estimates independent of $\mu_\gamma\in(0,1]$ (resp. $\mu_\gamma\in(1-\lambda,1]$) by this method. Because there is no uniform estimates as $(\ref{20190223})$ when $\mu_\gamma=0$ (resp. $\mu_\gamma=1-\lambda$), that is, $\gamma=0$. 
\end{rem}
By using Lemma \ref{1901}, the convergence of $(CKRF_{\mu_{\gamma}})$ with $\gamma\in(\beta,\beta+\delta(\lambda))$ follows from the arguments in \cite{JWLXZ,JWLXZ1}. It should be noted that we still don't need any assumptions in this case. As a corollary of the first two cases, we get the following existence theorem for the conical K\"ahler-Einstein metrics by using conical K\"ahler-Ricci flows. This result implies the existence conjecture for the conical K\"ahler-Einstein metrics with positive Ricci curvatures (when $\lambda=1$, they correspond to the cone angles), which is given by Donaldson \cite{SD2009} and proved by Berman, Jeffres-Mazzeo-Rubinstein and Li-Sun (see Theorem $1.5$ in \cite{RB}, Corollary $1$ in \cite{JMR} and Theorem $1.1$ in \cite{LS}).
\begin{thm}\label{20190221}  Assume $\lambda>1$. There exist conical K\"ahler-Einstein metrics with Ricci curvatures $\mu_\gamma$ and cone angles $2\pi\gamma$ along $D$ when $\gamma\in(0,1-\frac{1}{\lambda}+\delta(\lambda))$ for some $\delta(\lambda)>0$.
\end{thm}
{\it Case $3$}. If $\mu_\beta$ is positive, then $\mu_{\gamma}$ is still positive when $\gamma$ sufficiently close to $\beta$. In this case, we can not get the uniform estimates along $(TKRF_{\mu_{\beta},\varepsilon})$ by no assumptions as in {\it Case $2$}. We need the assumptions in Theorem \ref{01}. Let $\omega_{\varphi_{\beta}}=\omega_0+\sqrt{-1}\partial\bar{\partial}\varphi_\beta$ be a conical K\"ahler-Einstein metric with Ricci curvature $\mu_\beta$ and cone angle $2\pi \beta$ along $D$. We approximate $\varphi_\beta$ with a decreasing sequence of smooth $\omega_0$-psh functions $\varphi_\varepsilon$. Now we introduce the twisted K\"ahler-Ricci flows
$$(TKRF^{\beta}_{\mu_\gamma,\varepsilon}):\ \ \ \ \left\{
\begin{aligned}
&\ \frac{\partial \omega^{\beta}_{\gamma,\varepsilon}(t)}{\partial t}=-Ric(\omega^{\beta}_{\gamma,\varepsilon}(t))+\mu_\gamma\omega^{\beta}_{\gamma,\varepsilon}(t)+(\mu_\beta-\mu_\gamma)\omega_{\varphi_{\varepsilon}}+(1-\beta)\theta_{\varepsilon}\\
  \\
 &\ \omega^{\beta}_{\gamma,\varepsilon}(t)|_{t=0}=\hat{\omega}\\
\end{aligned}
\right.$$
with $\gamma\in [1-\frac{1}{\lambda},\beta]$ (that is, $\mu_\gamma\in[0,\mu_\beta]$). By the arguments in section $3$ of \cite{JWLXZ1}, when $\varepsilon$ tend to $0$, the limit flows of $(TKRF^{\beta}_{\mu_\gamma,\varepsilon})$ are the twisted conical K\"ahler-Ricci flows
$$(CKRF^{\beta}_{\mu_\gamma}):\ \ \ \ \left\{
\begin{aligned}
&\ \frac{\partial \omega^{\beta}_{\gamma}(t)}{\partial t}=-Ric(\omega^{\beta}_{\gamma}(t))+\mu_\gamma\omega^{\beta}_{\gamma}(t)+(\mu_\beta-\mu_\gamma)\omega_{\varphi_\beta}+(1-\beta)[D].\\
  \\
 &\ \omega^{\beta}_{\gamma}(t)|_{t=0}=\hat{\omega}\\
\end{aligned}
\right.$$
In the following arguments, we restrict $\gamma\in[1-\frac{1}{\lambda},\beta]$ in the flows $(TKRF^{\beta}_{\mu_\gamma,\varepsilon})$ and $(CKRF^{\beta}_{\mu_\gamma})$. In fact, $\omega^{\beta}_{\beta,\varepsilon}(t)=\omega_{\beta,\varepsilon}(t)$ and $\omega^{\beta}_{\beta}(t)=\omega_{\beta}(t)$ by the uniqueness results (see Proposition $2.7$ and Theorem $3.7$ in \cite{JWLXZ1}). Let $u^{\beta}_{\gamma,\varepsilon}(t)$ and $u^{\beta}_{\gamma}(t)$ be the twisted Ricci potentials of $\omega^{\beta}_{\gamma,\varepsilon}(t)$ and $\omega^{\beta}_{\gamma}(t)$ with normalization $\frac{1}{V}\int_M e^{-u^{\beta}_{\gamma,\varepsilon}(t)}dV^\beta_{\gamma,\varepsilon}(t)=1$ and $\frac{1}{V}\int_M e^{-u^{\beta}_{\gamma}(t)}dV^\beta_{\gamma}(t)=1$ respectively. We denote functionals
\begin{equation*}\begin{split}\label{008}
A^\beta_{\mu_\gamma,\varepsilon}(t)&=\frac{1}{V}\int_M u^{\beta}_{\gamma,\varepsilon}(t)e^{-u^{\beta}_{\gamma,\varepsilon}(t)}dV^\beta_{\gamma,\varepsilon}(t),\\
A^\beta_{\mu_\gamma}(t)&=\frac{1}{V}\int_M u^{\beta}_{\gamma}(t)e^{-u^{\beta}_{\gamma}(t)}dV^\beta_{\gamma}(t).
\end{split}
\end{equation*}
When $\mu_\gamma=0$, the functional $A^\beta_{\mu_\gamma}(t)$ is well-defined (Theorem \ref{20190221001}) and converges to zero as $t$ tend to $\infty$ (Theorem \ref{20190221002}). When $\mu_\gamma\in(0,\mu_\beta]$, from the uniform Perelman's estimates (Theorem \ref{20190314001}), the functional $A^\beta_{\mu_\gamma}(t)$ is well-defined. Furthermore, when $\mu_\gamma\in(0,\mu_\beta]$, $A^\beta_{\mu_\gamma,\varepsilon}(t)$ is nondecreasing along $(TKRF^\beta_{\mu_\gamma,\varepsilon})$ (Theorem \ref{16}), and $A^\beta_{\mu_\gamma}(t)$ converges to $0$ as $t$ tend to $\infty$ when the Log Mabuchi energy $\mathcal{M}^\beta_{\mu_\gamma}$ is bounded from below (Theorem \ref{46}).

First, by similar arguments as in {\it Case $2$}, we get the uniform estimates along $(TKRF^{\beta}_{0,\varepsilon})$ (Theorem \ref{20190221001}) and then obtain the uniform estimates along $(TKRF^{\beta}_{\mu_\gamma,\varepsilon})$ when $\mu_\gamma$ sufficiently close to $0$ (Lemma \ref{20190221003}). We let $\psi^\beta_{\gamma,\varepsilon}(t)$ be the metric potential (with specific choice of initial data, see $(\ref{TKRF5})$) of $\omega^\beta_{\gamma,\varepsilon}(t)$ with background metric $\omega_{0}$. Fix a $\mu_{\beta_0}>0$ obtained in Lemma \ref{20190221003}, there exists constant $C_{\beta_0}$ independent of $\varepsilon\in(0,\delta)$ such that
\begin{equation}\label{05}
\|\psi^{\beta}_{\beta_0,\varepsilon}(t)\|_{C^0(M)}\leqslant C_{\beta_0}\ \ \ \ \ \ \ \ \ \ \ for\ t\in [0,+\infty).
\end{equation}
At the same time, there exists uniform constant $\hat{C}_{\beta_0}$ such that
\begin{equation}\label{05001}
\|\dot{\psi}^{\beta}_{\gamma,\varepsilon}(t)\|_{C^0(M)}\leqslant \hat{C}_{\beta_0}
\end{equation}
for any $\varepsilon>0$, $\gamma\in[\beta_0,\beta]$ and $\ t\in [1,+\infty)$ (Proposition \ref{1.8.5.3}).

Then, by using the uniform estimates $(\ref{05})$ and $(\ref{05001})$, the continuity of $(TKRF^\beta_{\mu_\gamma,\varepsilon})$ with respect to $\mu_\gamma$ and $\varepsilon$ (Remark \ref{43}), the convergence of $A^\beta_{\mu_\gamma}(t)$, the monotonicity of $A^\beta_{\mu_\gamma,\varepsilon}(t)$ and Berndtsson's uniqueness theorem \cite{BBERN} for the conical K\"ahler-Einstein metrics with bounded potentials, we deduce the uniform estimates (independent of $\gamma$) of $\psi^\beta_{\gamma,\varepsilon}(t)$ for $\gamma\in(\beta_0,\beta)$. That is, we obtain the following lemma.
\begin{lem} \label{0080} Under the same assumptions as in Theorem \ref{01}. For above $\mu_{\beta_0}>0$, there exists a constant $\delta_{\beta_0}>0$ depending on $\beta_0$ such that
\begin{equation}\label{006}
\|\psi^{\beta}_{\gamma,\varepsilon}(t)\|_{C^0(M)}\leqslant\max(\|\varphi_\beta\|_{C^0(M)}+\frac{\hat{C}_{\beta_0}+|\xi_\beta|}{\mu_{\beta_0}},C_{\beta_0})+1
\end{equation}
for any $\varepsilon\in(0,\delta_{\beta_0})$, $\gamma\in[\beta_0,\beta)$ and $t\in[\frac{1}{\delta_{\beta_0}},+\infty)$, where $C_{\beta_0}$, $\hat{C}_{\beta_0}$ and $\xi_\beta$ are the constants in $(\ref{05})$, $(\ref{05001})$ and $(\ref{2019022301})$ respectively.\end{lem}
Therefore, there exists a uniform constant $B_\beta$ such that for any $\varepsilon\in(0,\delta_{\beta_0})$,
\begin{equation}\label{007}
\|\psi_{\beta,\varepsilon}(t)\|_{C^0(M)}\leqslant B_\beta\ \ \ \ \ \ \ \ \ \ \ \ \ on\ \ \ [0,+\infty).
\end{equation}
Then by using the arguments in \cite{JWLXZ1},  we get the convergence of the conical K\"ahler-Ricci flows $(CKRF_{\mu_\beta})$. Instead of using the properness of $\mathcal{M}_{\mu_\beta}$ in \cite{JWLXZ,JWLXZ1}, here we only use the weaker condition that $\mathcal{M}_{\mu_\beta}$ is bounded from below.
\begin{thm}\label{732}Assume that $\lambda >0$ and there is no nontrivial holomorphic fields on $M$ tangent to $D$. If $\mu_\beta$ $(0<\beta<1)$  is positive, we further assume that there is a conical K\"ahler-Einstein metric with Ricci curvature $\mu_\beta$ and cone angle $2\pi\beta$ along $D$. Then the conical K\"ahler-Ricci flow $(CKRF_{\mu_\beta})$ converges to a conical K\"ahler-Einstein metric with Ricci curvature $\mu_\beta$ and cone angle $2\pi\beta$ along $D$ in $C_{loc}^{\infty}$-topology outside divisor $D$ and globally in $C^{\alpha,\beta}$-sense for any $\alpha\in(0,\min\{1,\frac{1}{\beta}-1\})$.
\end{thm}
We denote $\psi_{\gamma,\varepsilon}(t)$ is the metric potential (with specific choice of initial data, see $(\ref{TKRF6})$) of $\omega_{\gamma,\varepsilon}(t)$ with background metric $\omega_{0}$. At last, by similar arguments as above, we prove
\begin{lem} \label{009} Under the same assumptions as in Theorem \ref{01}. There exists a constant $\delta>0$ such that
\begin{equation}\label{11}
\|\psi_{\gamma,\varepsilon}(t)\|_{C^0(M)}\leqslant\max(\|\varphi_\beta\|_{C^0(M)}+\frac{C_\beta+|\xi_\beta|}{\mu_\beta},B_{\beta})+1
\end{equation}
for any $\varepsilon\in(0,\delta)$, $\gamma\in(\beta-\delta,\beta+\delta)$ and $t\in[\frac{1}{\delta},+\infty)$, where $B_{\beta}$ is the constant in $(\ref{007})$, $\xi_\beta$ is the constant in $(\ref{2019022301})$ and $C_\beta$ is the constant in Proposition \ref{1.8.5.3.11}.
\end{lem}
Then the convergence of the conical K\"ahler-Ricci flows $(CKRF_{\mu_{\gamma}})$ with $\gamma\in(\beta-\delta,\beta+\delta)$ follows from the arguments in \cite{JWLXZ,JWLXZ1}.

It is generally known that Donaldson's celebrated openness theorem played an important role in solving the Yau-Tian-Donaldson's conjecture. It was first proposed by Donaldson \cite{SD2}. Then Yao \cite{Yao2} (see a remark in \cite{JWLXZ}) and Tian-Zhu \cite{GTXHZ05} gave different proofs by using elliptic methods. Here, as a corollary of Theorem \ref{01}, we give a parabolic proof of Donaldson's openness theorem.
\begin{thm}\label{12}({\bf Donaldson's openness theorem})\ \ \ Assume that $\lambda >0$ and there is no nontrivial holomorphic vector fields on $M$ tangent to $D$. If there is a conical K\"ahler-Einstein metric with Ricci curvature $\mu_\beta$ and cone angle $2\pi\beta$ along $D$, then there exist conical K\"ahler-Einstein metrics with Ricci curvature $\mu_{\beta'}$ and cone angles $2\pi\beta'$ along $D$ for $\beta'$ sufficiently close to $\beta$.
\end{thm}
We give a further remark of our project which is related to the parabolic version of the Yau-Tian-Donaldson's conjecture.
\begin{rem}\label{20190405}
When $\lambda>1$. We denote $\mu_\beta=0$ (that is, $\beta=1-\frac{1}{\lambda}$) and define a set $\Omega_\lambda$ as follows
\begin{equation*}
\begin{split}
\Omega_\lambda=\bigg\{s\in(\beta,1]\bigg|&\ (CKRF_{\mu_\gamma})\ converges\ to\ a\ conical\ K\ddot{a}hler-Einstein\ metric\ with\\
&\ Ricci\ curvature\ \mu_\gamma\ and\ cone\ angle\ 2\pi\gamma\ along\ D\ for\ \gamma\in(\beta,s] \bigg\} .
\end{split}
\end{equation*}
From Theorem \ref{01}, $\Omega_\lambda$ contains $(\beta,\beta_0)$ with some $\beta_0\in(\beta,1)$, so it is non-empty. Furthermore, $\Omega_\lambda$ is open in $(\beta,1)$. If we can combine the closeness of $\Omega_\lambda$ with the stability of the manifolds, then $1\in\Omega_\lambda$ and thus the smooth K\"ahler-Ricci flow converges to a K\"ahler-Einstein metric. 

When $\lambda=1$, we denote $\kappa>0$ sufficiently small and $\nu_\gamma=1-(1-\gamma)(1+\kappa)$. We consider the twisted conical K\"ahler-Ricci flows
$$(TCKRF_{\nu_\gamma}):\ \ \ \ \left\{
\begin{aligned}
 &\ \frac{\partial \omega_{\gamma}(t)}{\partial t}=-Ric(\omega_{\gamma}(t))+\nu_{\gamma}\omega_{\gamma}(t)+(1-\gamma)\kappa\omega_0+(1-\gamma)[D].\\
  \\
 &\ \omega_{\gamma}(t)|_{t=0}=\hat{\omega}\\
\end{aligned}
\right.$$
Assume that $\nu_{\hat{\beta}}=0$ (that is, $\hat{\beta}=1-\frac{1}{1+\kappa}>0$). We define set $\Omega_\kappa$ as follows
\begin{equation*}
\begin{split}
\Omega_\kappa=\bigg\{s\in(\hat{\beta},1]\bigg|&\ (TCKRF_{\nu_\gamma})\ converges\ to\ a\ twisted\ conical\ K\ddot{a}hler-Einstein\ metric\\
&\ \ with\ twisted\ form\ (1-\gamma)\kappa\omega_0\ ,\ twisted\ Ricci\ curvature\ \nu_\gamma\ and\ cone\\
&\ \  angle\ 2\pi\gamma\ along\ D\ for\ \gamma\in(\hat{\beta},s] \bigg\} .
\end{split}
\end{equation*}
By similar arguments as Theorem \ref{01}, we know that $\Omega_\kappa$ is open in $(\hat{\beta},1)$ and contains $(\hat{\beta},\beta_0)$ for some $\beta_0\in(\hat{\beta},1)$. If $\Omega_\kappa$ is closed under some stability assumption on the manifold, then $1\in\Omega_\kappa$ which means that the smooth K\"ahler-Ricci flow converges to a K\"ahler-Einstein metric.

The above problem can be seen as a parabolic version of the Yau-Tian-Donaldson's conjecture. In our subsequent work, we will focus on the relation between the closeness of $\Omega_\lambda$ (resp. $\Omega_\kappa$) and the K-stability of the manifold. 
\end{rem}
This paper is organized as follows. In section $2$, we recall some results of the twisted K\"ahler-Ricci flows and conical K\"ahler-Ricci flows in the literature. In section $3$, we study the {\it Case 2} of Theorem \ref{01}. We first give a uniform lower bound independent of $\mu_\gamma$ for the twisted scalar curvatures $R(\omega_{\mu_\gamma,\varepsilon}(t))-(1-\gamma)tr_{\omega_{\mu_\gamma,\varepsilon}(t)}\theta_\varepsilon$. Then after we get some uniform estimates along the twisted K\"ahler-Ricci flows $(TKRF_{0,\varepsilon})$ and $(TKRF_{\mu_\gamma,\varepsilon})$, we prove Lemma \ref{1901}. At last, we prove Lemma \ref{0080} and Lemma \ref{009} after we get some uniform regularities for the twisted K\"ahler-Ricci flows $(TKRF^\beta_{\mu_\gamma,\varepsilon})$ in section $5$.

\medskip

{\bf Acknowledge:} The first author would like to thank Professors Jiayu Li, Miles Simon and Xiaohua Zhu for their useful discussions, constant help and encouragement. He also would like to thank Professor Xiangwen Zhang, Doctors Chao Li and Xishen Jin for their several useful comments. Part of this work was carried out while the first author's visit to the University of Newcastle and the University of Adelaide in Australia. He is grateful to Professors James McCoy and Thomas Leistner for their invitations, and the universities for their hospitality.

\section{Preliminaries}\label{sec:2}
In this section, we give some known results about the twisted K\"ahler-Ricci flows and the (twisted) conical K\"ahler-Ricci flows. Let $M$ be a compact K\"ahler manifold of complex dimension $n$ and $D$ be a smooth divisor. By saying that a closed positive $(1,1)$-current $\omega$ with locally bounded potentials is conical K\"ahler metric with cone angle $2\pi \beta $ ($0<\beta<1$) along $D$, we mean that $\omega$ is smooth K\"ahler metric on $M\setminus D$. And near each point $p\in D$, there exists local holomorphic coordinate $(z_{1}, \cdots, z_{n})$ in a neighborhood $U$ of $p$ such that $D=\{z_{n}=0\}$ and $\omega$ is asymptotically equivalent to the model conical metric
\begin{equation}\label{2016081901}\sqrt{-1} \sum_{j=1}^{n-1} dz_{j}\wedge d\overline{z}_{j}+\sqrt{-1} |z_{n}|^{2\beta -2} dz_{n}\wedge d\overline{z}_{n} \ \ \ on\ \ U.\end{equation}
\begin{defi}
Let $\omega_{0}$ be a smooth K\"ahler metric and $D\subset M$ be a smooth divisor which satisfies $c_1(M)=\mu[\omega_0]+(1-\beta)[D]$ with $\mu\in \mathbb{R}$. We call $\omega $ a conical K\"ahler-Einstein metric with Ricci curvature $\mu$ and cone angle $2\pi \beta $ along $D$ if it is a conical K\"ahler metric with cone angle $2\pi \beta $ along $D$ and satisfies
\begin{equation}\label{2015.1.21.1}
Ric (\omega )=\mu \omega + (1-\beta ) [D]\ \ on\ M.
\end{equation}
\end{defi}
Equation (\ref{2015.1.21.1}) is classical outside $D$ and it holds in the sense of currents on $M$.
There are other definitions of metrics with conical singularities (see \cite{SD2,JMR}, etc.). But for conical K\"ahler-Einstein metrics, these definitions turn out to be equivalent (see Theorem $2$ in \cite{JMR}). 

In this paper, by saying that a conical K\"ahler-Einstein metric with cone angle $2\pi \beta $ along $D$ we also mean that its Ricci curvature is $\mu_\beta$. 

We write $\hat{\omega}=\omega_0+\sqrt{-1}\partial\bar{\partial}\varphi_0$. Let $\omega_\gamma=\omega_0+\frac{k}{\gamma^2}\sqrt{-1}\partial\bar{\partial}|s|_h^{2\gamma}$ with $\gamma\in(0,1)$ be the conical K\"ahler metric which is given by Donaldson \cite{SD2}. Now, in the Fano case, we list some results proved in \cite{JWLXZ1}.
\begin{thm}\label{1444}(Theorem $2.8$ in \cite{JWLXZ1})
There exists a unique long-time solution $\omega_{\gamma,\varepsilon}(t)$ to the twisted K\"ahler-Ricci flow $(TKRF_{\mu_\gamma,\varepsilon})$ in the following sense.
\begin{itemize}
\item $\omega_{\gamma,\varepsilon}(t)$ satisfies the twisted K\"ahler-Ricci flow $(TKRF_{\mu_\gamma,\varepsilon})$ on $(0,\infty)\times M$;
\item  There exists a metric potential $\varphi_{\gamma,\varepsilon}(t)\in C^0\big([0,\infty)\times M\big)\cap C^\infty\big((0,\infty)\times M\big)$ such that $\omega_{\gamma,\varepsilon}(t)=\omega_{0}+\sqrt{-1}\partial\bar{\partial}\varphi_{\gamma,\varepsilon}(t)$ and $\lim\limits_{t\rightarrow0^+}\|\varphi_{\gamma,\varepsilon}(t)-\varphi_{0}\|_{L^{\infty}(M)}=0$.
\end{itemize}
\end{thm}
\begin{thm}\label{04.5}(Theorem $1.2$ in \cite{JWLXZ1}) There exists a unique long-time solution $\omega_\gamma(t)$ to the conical K\"ahler-Ricci flow $(CKRF_{\mu_\gamma})$ in the following sense.
\begin{itemize}
  \item  For any $[\delta, T]$ ($0<\delta<T<\infty$), there exists constant $C$ such that
   \begin{equation*}
  C^{-1}\omega_{\gamma}\leqslant\omega_\gamma(t)\leqslant C\omega_\gamma\ \ \ \ \ \ \ \ \ on\ \ [\delta, T]\times (M\setminus D);
  \end{equation*}
  \item  On $(0,\infty)\times(M\setminus D)$, $\omega_\gamma(t)$ satisfies the smooth K\"ahler-Ricci flow;
  \item On $(0,\infty)\times M$, $\omega_\gamma(t)$ satisfies equation $(CKRF_{\mu_\gamma})$ in the sense of currents;
  \item There exists a metric potential $\varphi_\gamma(t)\in C^{0}\big([0,\infty)\times M\big)\cap C^{\infty}\big((0,\infty)\times (M\setminus D)\big)$ such that $\omega_\gamma(t)=\omega_{0}+\sqrt{-1}\partial\bar{\partial}\varphi_\gamma(t)$ and $\lim\limits_{t\rightarrow0^{+}}\|\varphi_\gamma(t)-\varphi_{0}\|_{L^{\infty}(M)}=0$;
  \item On $[\delta, T]$, there exist constants $\alpha\in(0,1)$ and $C^{\ast}$ such that $\varphi_\gamma(t)$ is $C^{\alpha}$ on $M$ with respect to $\omega_{0}$ and $\| \frac{\partial\varphi_\gamma(t)}{\partial t}\|_{L^{\infty}(M\setminus D)}\leqslant C^{\ast}$.
  \end{itemize}
\end{thm}
\begin{rem}\label{14}  In Theorem \ref{04.5}, by saying that $\omega_\gamma(t)$ satisfies equation $(CKRF_{\mu_\gamma})$ in the sense of currents on $M_{\infty}:=(0,\infty)\times M$, we mean that for any smooth $(n-1,n-1)$-form $\eta(t)$ with compact support in $(0,\infty)\times M$, we have
  \begin{equation*}
  \int_{M_{\infty}}\frac{\partial \omega_\gamma(t)}{\partial t}\wedge \eta(t,x)dt=\int_{M_{\infty}}(-Ric(\omega_\gamma(t))+\mu_\gamma\omega_\gamma(t)+(1-\gamma)[D])\wedge \eta(t,x)dt,
  \end{equation*}
  where the integral on the left side can be written as
  \begin{equation*}
  \int_{M_{\infty}}\frac{\partial \omega_\gamma(t)}{\partial t}\wedge \eta(t,x)dt=-\int_{M_{\infty}} \omega_\gamma(t)\wedge\frac{\partial\eta(t,x)}{\partial t}dt.
  \end{equation*}
\end{rem}
By using the similar arguments as that in the sections $2$ and $3$ of \cite{JWLXZ1}, we have the following results.
\begin{thm}\label{0044.5}
There exists a unique long-time solution $\omega^\beta_{\gamma,\varepsilon}(t)$ to the twisted K\"ahler-Ricci flow $(TKRF^{\beta}_{\mu_\gamma,\varepsilon})$ in the following sense.
\begin{itemize}
\item $\omega^\beta_{\gamma,\varepsilon}(t)$ satisfies the twisted K\"ahler-Ricci flow $(TKRF^{\beta}_{\mu_\gamma,\varepsilon})$ on $(0,\infty)\times M$;
\item  There exists a metric potential $\varphi^\beta_{\gamma,\varepsilon}(t)\in C^0\big([0,\infty)\times M\big)\cap C^\infty\big((0,\infty)\times M\big)$ such that $\omega^\beta_{\gamma,\varepsilon}(t)=\omega_{0}+\sqrt{-1}\partial\bar{\partial}\varphi^\beta_{\gamma,\varepsilon}(t)$ and $\lim\limits_{t\rightarrow0^+}\|\varphi^\beta_{\gamma,\varepsilon}(t)-\varphi_{0}\|_{L^{\infty}(M)}=0$.
\end{itemize}
\end{thm}
\begin{thm}\label{04.5001} There exists a unique long-time solution $\omega^\beta_\gamma(t)$ to the conical K\"ahler-Ricci flow $(CKRF^\beta_{\mu_\gamma})$ in the following sense.
\begin{itemize}
  \item  For any $[\delta, T]$ ($0<\delta<T<\infty$), there exists constant $C$ such that
   \begin{equation*}
  C^{-1}\omega_{\beta}\leqslant\omega^\beta_\gamma(t)\leqslant C\omega_\beta\ \ \ \ \ \ \ \ \ on\ \ [\delta, T]\times (M\setminus D);
  \end{equation*}
  \item  On $(0,\infty)\times(M\setminus D)$, $\omega^\beta_\gamma(t)$ satisfies the smooth twisted K\"ahler-Ricci flow;
  \item On $(0,\infty)\times M$, $\omega^\beta_\gamma(t)$ satisfies equation $(CKRF^\beta_{\mu_\gamma})$ in the sense of currents;
  \item There exists a metric potential $\varphi^\beta_\gamma(t)\in C^{0}\big([0,\infty)\times M\big)\cap C^{\infty}\big((0,\infty)\times (M\setminus D)\big)$ such that $\omega^\beta_\gamma(t)=\omega_{0}+\sqrt{-1}\partial\bar{\partial}\varphi^\beta_\gamma(t)$ and $\lim\limits_{t\rightarrow0^{+}}\|\varphi^\beta_\gamma(t)-\varphi_{0}\|_{L^{\infty}(M)}=0$;
  \item On $[\delta, T]$, there exist constants $\alpha\in(0,1)$ and $C^{\ast}$ such that $\varphi^\beta_\gamma(t)$ is $C^{\alpha}$ on $M$ with respect to $\omega_{0}$ and $\| \frac{\partial\varphi^\beta_\gamma(t)}{\partial t}\|_{L^{\infty}(M\setminus D)}\leqslant C^{\ast}$.
  \end{itemize}
\end{thm}
\begin{rem} By Liu-Zhang's results (Theorem $3.8$ in \cite{JWLCJZ}, see also Theorem $3.10$ in \cite{JWLXZ1}). The solutions $\omega_{\gamma,\varepsilon}(t)$ and $\omega^\beta_{\gamma,\varepsilon}(t)$ are $C^{\alpha,\beta}$ for any $\alpha\in(0,\min\{1,\frac{1}{\beta}-1\})$ when $t>0$.
\end{rem}
From the proof of Proposition $5.9$ in \cite{JWLXZ}, we have the following uniform Sobolev inequalities when the cone angles away from $0$.
\begin{thm}\label{Sobolev}
Assume that $n\geqslant2$. For any $\beta_0\in(0,1)$, there exists uniform constant $C$ such that
\begin{equation}\label{1.5.6}
(\int_{M}v^{\frac{2n}{n-1}}dV^\gamma_{\varepsilon})^{\frac{n-1}{n}}\leqslant C(\int_{M}|dv|^{2}_{\omega^\gamma_{\varepsilon}}dV^\gamma_{\varepsilon}+\int_{M}|v|^{2}dV^\gamma_{\varepsilon})
\end{equation}
hold for any smooth functions $v$ on $M$, $\gamma\in[\beta_0,1)$ and $\varepsilon>0$, where $dV_\varepsilon^\gamma=\frac{(\omega_\varepsilon^\gamma)^n}{n!}$.
\end{thm}
\begin{thm}\label{Sobolev1}
Assume that $n=1$. Then for any $\beta_0\in(0,1)$, there exists uniform constant $C$ such that
\begin{equation}\label{1.5.6001}
(\int_{M}v^{2}dV^\gamma_{\varepsilon})^{\frac{1}{2}}\leqslant C(\int_{M}|dv|_{\omega^\gamma_{\varepsilon}}dV^\gamma_{\varepsilon}+\int_{M}|v|dV^\gamma_{\varepsilon})
\end{equation}
holds for any smooth functions $v$ on $M$, $\gamma\in[\beta_0,1)$ and $\varepsilon>0$.
\end{thm}
In \cite{JWLXZ,JWLXZ1}, we proved the uniform Perelman's estimates for $\mu_\gamma>0$, which play an important role in studying the convergence of the conical K\"ahler-Ricci flows. These estimates mainly depend on the uniform lower bounded of the twisted scalar curvature $R(\omega_{\gamma,\varepsilon}(t))-(1-\gamma)tr_{\omega_{\gamma,\varepsilon}(t)}\theta_{\varepsilon}$, the uniform Sobolev inequality $\mathcal{C}_{S}(M,\omega_{\gamma,\varepsilon}(1))$ and $\frac{1}{\mu_\gamma}$. Hence when $\mu_\gamma$ away from $0$, we have the following uniform Perelman's estimates by using the arguments in \cite{JWLXZ}.
\begin{thm}\label{1.8.2} Let $\omega_{\gamma,\varepsilon}(t)$ be a solution of the twisted K\"ahler Ricci flow $(TKRF_{\mu_\gamma,\varepsilon})$. Then for any $0<\mu_{\gamma_1}<\mu_{\gamma_2}\leqslant1$, there exists uniform constant $C$, such that
\begin{equation}\begin{split}\label{p0}
|R(\omega_{\gamma,\varepsilon}(t))-(1-\gamma)tr_{\omega_{\gamma,\varepsilon}(t)}\theta_{\varepsilon}|&\leqslant C,\\
\|u_{\gamma,\varepsilon}(t)\|_{C^{1}(\omega_{\gamma,\varepsilon}(t))}&\leqslant C,\\
diam(M,\omega_{\gamma,\varepsilon}(t))&\leqslant C
\end{split}
\end{equation}
hold for any $\mu_\gamma\in[\mu_{\gamma_1},\mu_{\gamma_2}]$, $t\geqslant 1$ and $\varepsilon>0$.\end{thm}
\begin{rem}\label{15} The uniform Perelman's estimates for $(TKRF^\beta_{\mu_\gamma,\varepsilon})$ can also be established by using the similar arguments as that in \cite{JWLXZ}, but there exist a few details need to be verified in the proof, so we present them in section $4$.
\end{rem}
In the proof of uniform Perelman's estimates, we prove that the functional $A_{\mu_\gamma,\varepsilon}(t)$ with $\mu_\gamma>0$ is nondecreasing by using the uniform Poincar\'e inequality along along the twisted K\"ahler-Ricci flow $(TKRF_{\mu_\gamma,\varepsilon})$. Since $(\mu_\beta-\mu_\gamma)\omega_{\varphi_{\varepsilon}}+(1-\beta)\theta_{\varepsilon}$ are positive closed $(1,1)$-forms when $\mu_\gamma\in(0,\mu_\beta]$, there also exists uniform Poincar\'e inequality along $(TKRF^\beta_{\mu_\gamma,\varepsilon})$. So $A^\beta_{\mu_\gamma,\varepsilon}(t)$ is nondecreasing along $(TKRF^\beta_{\mu_\gamma,\varepsilon})$ when $\mu_\gamma\in(0,\mu_\beta]$.
\begin{thm}\label{16}
When $\mu_\gamma\in(0,1]$, the functional $A_{\mu_\gamma,\varepsilon}(t)$ is nondecreasing along the twisted K\"ahler-Ricci flow $(TKRF_{\mu_\gamma,\varepsilon})$.

When $\mu_\gamma\in(0,\mu_\beta]$, the functional $A^\beta_{\mu_\gamma,\varepsilon}(t)$ is nondecreasing along the twisted K\"ahler-Ricci flow $(TKRF^\beta_{\mu_\gamma,\varepsilon})$.
\end{thm}
Now we recall Aubin's functionals. Let $\phi_t$ be a path with $\phi_0=c$ and $\phi_1=\phi$, then
\begin{equation*}\label{3.22.35}
\begin{split}
I_{\omega_0}(\phi)&=\frac{1}{V}\int_M\phi(dV_{0}-dV_{\phi}), \\
J_{\omega_0}(\phi)&=\frac{1}{V}\int_0^1\int_M\dot{\phi}_t(dV_{0}-dV_{\phi_{t}})dt,
\end{split}
\end{equation*}
where $dV_{0}=\frac{\omega_{0}^{n}}{n!}$ and $dV_{\phi}=\frac{\omega_{\phi}^{n}}{n!}$. They satisfy $0\leqslant\frac{1}{n}J_{\omega_0}\leqslant\frac{1}{n+1}I_{\omega_0}\leqslant J_{\omega_0}$. The twisted Mabuchi energy is defined as
\begin{equation}\label{9090989}
\begin{split}
\mathcal{M}_{k,\theta}(\phi)&=-k(I_{\omega_{0}}(\phi)-J_{\omega_{0}}(\phi))-\frac{1}{V}\int_{M}u_{\omega_{0}}(dV_{0}-dV_{\phi})\\
&\ +\frac{1}{V}\int_{M}\log\frac{\omega_{\phi}^{n}}{\omega_{0}^{n}}dV_{\phi},
\end{split}
\end{equation}
where $u_{\omega_{0}}$ satisfies $-Ric(\omega_{0})+k\omega_{0}+\theta=\sqrt{-1}\partial\bar{\partial} u_{\omega_{0}}$ and $\frac{1}{V}\int_{M}e^{-u_{\omega_{0}}}dV_{0}=1$. We denote $\mathcal{M}^\beta_{\mu_\gamma,\varepsilon}$ and $\mathcal{M}_{\mu_\gamma,\varepsilon}$ are the twisted Mabuchi energies with twisted forms $(\mu_\beta-\mu_\gamma)\omega_{\varphi_\varepsilon}+(1-\beta)\theta_\varepsilon$ and $(1-\gamma)\theta_\varepsilon$ respectively. Then
\begin{equation*}\label{90909900001}
\begin{split}
\mathcal{M}^\beta_{\mu_\gamma,\varepsilon}(\phi)&=-\mu_\gamma(I_{\omega_{0}}(\phi)-J_{\omega_{0}}(\phi))+\frac{1}{V}\int_{M}\log\frac{\omega_{\phi}^{n}}{\omega_{0}^{n}}dV_{\phi}\\
&\ -\frac{1}{V}\int_{M}\Big(F_0+(\mu_\beta-\mu_\gamma)\varphi_\varepsilon+(1-\beta)\log(\varepsilon^2+|s|^2_h)\Big)(dV_{0}-dV_{\phi}),\\
\mathcal{M}_{\mu_\gamma,\varepsilon}(\phi)&=-\mu_\gamma(I_{\omega_{0}}(\phi)-J_{\omega_{0}}(\phi))+\frac{1}{V}\int_{M}\log\frac{\omega_{\phi}^{n}}{\omega_{0}^{n}}dV_{\phi}\\
&\ -\frac{1}{V}\int_{M}\Big(F_0+(1-\gamma)\log(\varepsilon^2+|s|^2_h)\Big)(dV_{0}-dV_{\phi}).
\end{split}
\end{equation*}
The Log Mabuchi energy is denoted by
\begin{equation}\label{9090990}
\begin{split}
\mathcal{M}_{\mu_\gamma}(\phi)&=-\mu_\gamma(I_{\omega_{0}}(\phi)-J_{\omega_{0}}(\phi))+\frac{1}{V}\int_{M}\log\frac{\omega_{\phi}^{n}}{\omega_{0}^{n}}dV_{\phi}\\
&\ -\frac{1}{V}\int_{M}\Big(F_0+(1-\gamma)\log|s|^2_h\Big)(dV_{0}-dV_{\phi}),
\end{split}
\end{equation}
and the twisted Log Mabuchi energy is defined as
\begin{equation}\label{9090991}
\begin{split}
\mathcal{M}^\beta_{\mu_\gamma}(\phi)&=-\mu_\gamma(I_{\omega_{0}}(\phi)-J_{\omega_{0}}(\phi))+\frac{1}{V}\int_{M}\log\frac{\omega_{\phi}^{n}}{\omega_{0}^{n}}dV_{\phi}\\
&\ -\frac{1}{V}\int_{M}\Big(F_0+(\mu_\beta-\mu_\gamma)\varphi_\beta+(1-\beta)\log|s|^2_h\Big)(dV_{0}-dV_{\phi}).
\end{split}
\end{equation}

\medskip

By using Berndtsson's convexity theorem \cite{BBERN}, Li-Sun (Corollary $2.10$ in \cite{LS}) proved that the conical K\"ahler-Einstein metric is the minimum point of the Log Mabuchi energy.
\begin{thm}\label{201902131} (Corollary $2.10$ in \cite{LS}) If there exists a conical K\"ahler-Einstein metric $\omega_{\varphi_\beta}$ with cone angle $2\pi\beta$ along $D$, then $\omega_{\varphi_\beta}$ obtains the minimum of $\mathcal{M}_{\mu_\beta}$.
\end{thm}
Combining the uniform Perelman's estimates, the Perelman's non-collapsing theorem and the uniform Poincar\'e inequality along $(TKRF_{\mu_\gamma,\varepsilon})$, we get the following estimates and convergence for $A_{\mu_\gamma}(t)$ by using the arguments in Lemma $4.3$ of \cite{JWL} or Lemma $3$ of \cite{PSSW}.
\begin{thm}\label{17}
For $\mu_\gamma\in(0,1]$, there exists constant $C$ such that for $t\geqslant1$,
\begin{equation}
0\leqslant-A_{\mu_\gamma}(t)\leqslant C\|\nabla u_{\gamma}(t)\|^{\frac{1}{n+1}}_{L^2(M)}\|\nabla u_{\gamma}(t)\|_{C^{0}(M)}^{\frac{n}{n+1}}.
\end{equation}
Furthermore, if $\mathcal{M}_{\mu_\gamma}$ is bounded form below, then $A_{\mu_\gamma}(t)$ converge to $0$ as $t\rightarrow+\infty$.
\end{thm}
\begin{rem}\label{15} We can also get a similar result for $A^\beta_{\mu_\gamma}(t)$ after we prove the uniform Perelman's estimates along the twisted K\"ahler-Ricci flow $(TKRF^\beta_{\mu_\gamma,\varepsilon})$ in section $4$.
\end{rem}

\section{Proof of the case $\mu_\beta=0$\label{sec:3}}

Assume that $\lambda>1$. Then $\mu_\beta=0$ when $\beta=1-\frac{1}{\lambda}\in(0,1)$. In this section, we study the convergence of the conical K\"ahler-Ricci flow $(CKRF_{\mu_\gamma})$ when $\mu_\gamma$ is positive and sufficiently small. For $\mu_\gamma\in(0, 1]$ (that is, $\gamma\in(\beta,1]$), there is no direct uniform Perelman's estimates independent of $\mu_\gamma$ along the twisted K\"ahler-Ricci flows $(TKRF_{\mu_\gamma,\varepsilon})$ and hence we do not have the uniform estimates for $\dot{\psi}_{\gamma,\varepsilon}(t)$, so we need some new discussions. In this section, we always assume that $\mu_\beta=0$ and $\mu_\gamma\in[\mu_\beta,1]$ (that is, $\gamma\in[\beta,1]$) unless otherwise specified, and normalize $\varphi_0$ by
\begin{equation}\label{nom0}
\inf\limits_M\varphi_0\geqslant 1.
\end{equation}

First,  by improving the proof of Proposition $5.1$ in \cite{JWLXZ}, we get the uniform lower bound of the twisted scalar curvature $R(\omega_{\gamma,\varepsilon}(t))-(1-\gamma)tr_{\omega_{\gamma,\varepsilon}(t)}\theta_\varepsilon$ when $t\geqslant 1$.
\begin{lem}\label{1902} $R(\omega_{\gamma,\varepsilon}(t))-(1-\gamma)tr_{\omega_{\gamma,\varepsilon}(t)}\theta_{\varepsilon}$ are uniformly bounded from below by $-4n$ along the twisted K\"ahler-Ricci flows $(TKRF_{\mu_\gamma,\varepsilon})$, that is, for any $\mu_\gamma\in[0,1]$, $\varepsilon>0$ and $t\geqslant 1$, we have
\begin{equation}\label{3.22.9}R(\omega_{\gamma,\varepsilon}(t))-(1-\gamma)tr_{\omega_{\gamma,\varepsilon}(t)}\theta_{\varepsilon}\geqslant -4n.
\end{equation}
\end{lem}
\begin{proof}  First, we derive the evolution equation of $(t-\frac{1}{2})^2\Big(R(\omega_{\gamma,\varepsilon}(t))-(1-\gamma)tr_{\omega_{\gamma,\varepsilon}(t)}\theta_{\varepsilon}\Big)$.
\begin{equation*}
\begin{split}
&\ \ \ \ \ (\frac{\partial}{\partial t}-\Delta_{\omega_{\gamma,\varepsilon}(t)})\Big((t-\frac{1}{2})^2\big(R(\omega_{\gamma,\varepsilon}(t))-(1-\gamma)tr_{\omega_{\gamma,\varepsilon}(t)}\theta_{\varepsilon}\big)\Big)\\
&=(t-\frac{1}{2})^2|Ric(\omega_{\gamma,\varepsilon}(t))-(1-\gamma)\theta_{\varepsilon}|^{2}_{\omega_{\varepsilon}(t)}-\mu_\gamma (t-\frac{1}{2})^2\Big(R(\omega_{\gamma,\varepsilon}(t))-(1-\gamma)tr_{\omega_{\gamma,\varepsilon}(t)}\theta_{\varepsilon}\Big)\\
&\ \ \ \ \ +2(t-\frac{1}{2})\Big(R(\omega_{\gamma,\varepsilon}(t))-(1-\gamma)tr_{\omega_{\gamma,\varepsilon}(t)}\theta_{\varepsilon}\Big).
\end{split}
\end{equation*}
Let $(t_0,x_0)$ be the minimum point of $(t-\frac{1}{2})^2\Big(R(\omega_{\gamma,\varepsilon}(t))-(1-\gamma)tr_{\omega_{\gamma,\varepsilon}(t)}\theta_{\varepsilon}\Big)$ on $[\frac{1}{2},1]\times M$.

\medskip

$Case\ 1$, $t_0=\frac{1}{2}$, then we have $(t-\frac{1}{2})^2\Big(R(\omega_{\gamma,\varepsilon}(t))-(1-\gamma)tr_{\omega_{\gamma,\varepsilon}(t)}\theta_{\varepsilon}\Big)\geqslant 0$.

\medskip

$Case\ 2$, $t_0\neq\frac{1}{2}$, without loss of generality, we can assume $R(\omega_{\gamma,\varepsilon}(t))-(1-\gamma)tr_{\omega_{\gamma,\varepsilon}(t)}\theta_{\varepsilon}\leqslant0$ at $(t_0,x_0)$. By inequality
\begin{equation*}|Ric(\omega_{\gamma,\varepsilon}(t))-(1-\gamma)\theta_{\varepsilon}|^{2}_{\omega_{\gamma,\varepsilon}(t)}\geqslant \frac{\Big(R(\omega_{\gamma,\varepsilon}(t))-(1-\gamma)tr_{\omega_{\gamma,\varepsilon}(t)}\theta_{\varepsilon}\Big)^2}{n},
\end{equation*}
we have
\begin{equation*}
\begin{split}0&\geqslant  (t_0-\frac{1}{2})^2\frac{\Big(R(\omega_{\gamma,\varepsilon}(t_0))-(1-\gamma)tr_{\omega_{\gamma,\varepsilon}(t_0)}\theta_{\varepsilon}\Big)^2}{n}+2(t_0-\frac{1}{2})\Big(R(\omega_{\gamma,\varepsilon}(t_0))-(1-\gamma)tr_{\omega_{\gamma,\varepsilon}(t_0)}\theta_{\varepsilon}\Big)\\
&=\Big((t_0-\frac{1}{2})\frac{R(\omega_{\gamma,\varepsilon}(t_0))-(1-\gamma)tr_{\omega_{\gamma,\varepsilon}(t_0)}\theta_{\varepsilon}}{\sqrt{n}}+\sqrt{n}\Big)^2-n.
\end{split}
\end{equation*}
Then $(t_0-\frac{1}{2})^{2}\Big(R(\omega_{\gamma,\varepsilon}(t_0))-(1-\gamma)tr_{\omega_{\gamma,\varepsilon}(t_0)}\theta_{\varepsilon}\Big)\geqslant -2n(t_0-\frac{1}{2})\geqslant -n$. Hence,
\begin{equation}(t-\frac{1}{2})^{2}\Big(R(\omega_{\gamma,\varepsilon}(t))-(1-\gamma)tr_{\omega_{\gamma,\varepsilon}(t)}\theta_{\varepsilon}\Big)\geqslant -n
\end{equation}
 for any $\mu_\gamma\in[0,1]$, $\varepsilon>0$ and $t\in[\frac{1}{2},1]$. In particular, when $t=1$, we have
\begin{equation}
R(\omega_{\gamma,\varepsilon}(1))-(1-\gamma)tr_{\omega_{\gamma,\varepsilon}(1)}\theta_{\varepsilon}\geqslant -4n.
\end{equation}
Then we consider the lower bound of $R(\omega_{\gamma,\varepsilon}(t))-(1-\gamma)tr_{\omega_{\gamma,\varepsilon}(t)}\theta_{\varepsilon}$ on $[1,\infty)\times M$.
\begin{equation*}
\begin{split}
&\ \ \ (\frac{\partial}{\partial t}-\Delta_{\omega_{\gamma,\varepsilon}(t)})\big(R(\omega_{\gamma,\varepsilon}(t))-(1-\gamma)tr_{\omega_{\gamma,\varepsilon}(t)}\theta_{\varepsilon}\big)\\
&=|Ric(\omega_{\gamma,\varepsilon}(t))-(1-\gamma)\theta_{\varepsilon}|^{2}_{\omega_{\gamma,\varepsilon}(t)}-\mu_\gamma \Big(R(\omega_{\gamma,\varepsilon}(t))-(1-\gamma)tr_{\omega_{\gamma,\varepsilon}(t)}\theta_{\varepsilon}\Big)\\
&\geqslant -\mu_\gamma \Big(R(\omega_{\gamma,\varepsilon}(t))-(1-\gamma)tr_{\omega_{\gamma,\varepsilon}(t)}\theta_{\varepsilon}\Big),
\end{split}
\end{equation*}
By maximum principle, on $[1,\infty)\times M$, we have
\begin{equation}
R(\omega_{\gamma,\varepsilon}(t))-(1-\gamma)tr_{\omega_{\gamma,\varepsilon}(t)}\theta_{\varepsilon}\geqslant -4ne^{-\mu_\gamma(t-1)}\geqslant -4n
\end{equation}
for any $\mu_\gamma\in[0,1]$, $t\geqslant 1$ and $\varepsilon>0$. We complete the proof this lemma.
\end{proof}

We write the flows $(TKRF_{\mu_\gamma,\varepsilon})$ as parabolic Monge-Amp\`ere equations
\begin{equation}\label{TKRF7}
\left \{\begin{split}
&\  \frac{\partial \varphi_{\gamma,\varepsilon}(t)}{\partial t}=\log\frac{(\omega_{0}+\sqrt{-1}\partial\bar{\partial}\varphi_{\gamma,\varepsilon}(t))^{n}}{\omega_{0}^{n}}+\mu_\gamma\varphi_{\gamma,\varepsilon}(t)\\
&\ \ \ \ \ \ \ \ \ \ \ \ \ \ \ \ \ \ \ \ +F_{0}+\log(\varepsilon^{2}+|s|_{h}^{2})^{1-\gamma},\\
&\  \varphi_{\gamma,\varepsilon}(0)=\varphi_{0}
\end{split}
\right.
\end{equation}
where $F_{0}$ satisfies $-Ric(\omega_{0})+\omega_{0}=\sqrt{-1}\partial\overline{\partial}F_{0}$ and $\frac{1}{V}\int_{M}e^{-F_{0}}dV_{0}=1$. Let $\chi_\gamma$ with $\gamma\in(0,1)$ be the function $\chi_\gamma(\varepsilon^{2}+|s|_h^2)=\frac{1}{\gamma}\int_{0}^{|s|_h^2}\frac{(\varepsilon^{2}+r)^{\gamma}-\varepsilon^{2\gamma}}{r}dr$ which is given by Campana-Guenancia-P$\breve{a}$un \cite{CGP}. Denote $F_{\gamma,\varepsilon}=F_{0}+\log\Big(\frac{(\omega^\gamma_{\varepsilon})^{n}}{\omega_{0}^{n}}\cdot(\varepsilon^{2}+|s|_{h}^{2})^{1-\gamma}\Big)$, $\phi_{\gamma,\varepsilon}(t)=\varphi_{\gamma,\varepsilon}(t)-k\chi_\gamma$ and $\omega^\gamma_{\varepsilon}=\omega_{0}+\sqrt{-1}k\partial\overline{\partial}\chi_\gamma$. Equation (\ref{TKRF7}) can be written as
\begin{equation}\label{TKRF9}
\left \{\begin{split}
&\  \frac{\partial \phi_{\gamma,\varepsilon}(t)}{\partial t}=\log\frac{(\omega^\gamma_{\varepsilon}+\sqrt{-1}\partial\bar{\partial}\phi_{\gamma,\varepsilon}(t))^{n}}{(\omega^\gamma_{\varepsilon})^{n}}+F_{\gamma,\varepsilon}+\mu_\gamma(\phi_{\gamma,\varepsilon}(t)+k\chi_\gamma).\\
&\  \\
&\  \phi_{\gamma,\varepsilon}(0)=\varphi_{0}-k\chi_\gamma:=\phi_{\gamma,\varepsilon}
\end{split}
\right.
\end{equation}
We first prove the uniform $C^0$-estimates (independent of $\mu_\gamma$) for $\phi_{\gamma,\varepsilon}$.
\begin{lem}\label{201901}
For any $T>0$, there exists constant $C$ depending only on $\|\varphi_0\|_{L^\infty(M)}$, $\beta$, $n$, $\omega_{0}$ and $T$ such that for any $t\in[0,T]$, $\varepsilon>0$ and $\mu_\gamma\in[0,1]$ (that is, $\gamma\in[\beta,1]$),
\begin{equation}
\|\phi_{\gamma,\varepsilon}(t)\|_{C^0(M)}\leqslant C.
\end{equation}
\end{lem}
\begin{proof}  Since $\beta=1-\frac{1}{\lambda}>0$ and $\gamma\in[\beta,1]$, from Campana-Guenancia-P$\breve{a}$un's results (see $(15)$ and $(25)$ in \cite{CGP}), there exist uniform constants $A_0$ and $A_1$, such that for any $\varepsilon>0$ and $\gamma\in[\beta,1]$,
\begin{equation}
\|F_{\gamma,\varepsilon}+k\mu_\gamma\chi_\gamma\|_{C^0(M)}\leqslant A_0\ \ and\ \ \|\phi_{\gamma,\varepsilon}\|_{C^0(M)}\leqslant A_1.
\end{equation}
Then we have
\begin{equation*}
\frac{\partial e^{-\mu_\gamma t}\phi_{\gamma,\varepsilon}(t)}{\partial t}\leqslant e^{-\mu_\gamma t}\log\frac{(e^{-\mu_\gamma t}\omega^\gamma_{\varepsilon}+\sqrt{-1}\partial\bar{\partial}e^{-\mu_\gamma t}\phi_{\gamma,\varepsilon}(t))^{n}}{(e^{-\mu_\gamma t}\omega^\gamma_{\varepsilon})^{n}}+e^{-\mu_\gamma t}A_0,
\end{equation*}
which is equivalent to
\begin{equation*}
\begin{split}
&\ \ \frac{\partial }{\partial t}\big(e^{-\mu_\gamma t}\phi_{\gamma,\varepsilon}(t)-A_0\int_0^t e^{-\mu_\gamma s}ds\big)\\
&\leqslant e^{-\mu_\gamma t}\log\frac{\big(e^{-\mu_\gamma t}\omega^\gamma_{\varepsilon}+\sqrt{-1}\partial\bar{\partial}(e^{-\mu_\gamma t}\phi_{\gamma,\varepsilon,}(t)-A_0\int_0^t e^{-\mu_\gamma s}ds)\big)^{n}}{(e^{-\mu_\gamma t}\omega^\gamma_{\varepsilon})^{n}}.
\end{split}
\end{equation*}
For any $\delta>0$, we denote $\tilde{\phi}_{\gamma,\varepsilon}(t)=e^{-\mu_\gamma t}\phi_{\gamma,\varepsilon}(t)-A_0\int_0^t e^{-\mu_\gamma s}ds-\delta t$. Let $(t_0,x_0)$ be the maximum point of $\tilde{\phi}_{\gamma,\varepsilon}(t)$ on $[0,T]\times M$. If $t_0>0$, by maximum principle, we have
\begin{equation*}
\begin{split}
0&\leqslant\frac{\partial }{\partial t}\big(e^{-\mu_\gamma t}\phi_{\gamma,\varepsilon}(t)-A_0\int_0^t e^{-\mu_\gamma s}ds-\delta t\big)\Big|_{(t_0,x_0)}\\
&\leqslant e^{-\mu_\gamma t}\log\frac{\big(e^{-\mu_\gamma t}\omega^\gamma_{\varepsilon}+\sqrt{-1}\partial\bar{\partial}\tilde{\phi}_{\gamma,\varepsilon}(t)\big)^{n}}{(e^{-\mu_\gamma t}\omega^\gamma_{\varepsilon})^{n}}\Big|_{(t_0,x_0)}-\delta\\
&\leqslant-\delta,
\end{split}
\end{equation*}
which is impossible. Hence $t_0=0$, then
\begin{equation*}
\phi_{\gamma,\varepsilon}(t)\leqslant e^{\mu_\gamma t}\sup\limits_{M}\phi_{\gamma,\varepsilon}(0)+\delta Te^{\mu_\gamma T}+A_0T e^{\mu_\gamma T}.
\end{equation*}
Let $\delta\rightarrow0$, we obtain
\begin{equation}\label{202}
\phi_{\gamma,\varepsilon}(t)\leqslant A_1e^{T}+A_0T e^{T}.
\end{equation}
Hence there exists a constant $C$ depending only on $\|\varphi_0\|_{L^\infty(M)}$, $\beta$, $n$, $\omega_{0}$ amd $T$ such that $ \phi_{\gamma,\varepsilon}(t)\leqslant C$ for any $t\in[0,T]$, $\varepsilon>0$ and $\mu_\gamma\in[0,1]$. By similar arguments we can get the uniform lower bound for $\phi_{\gamma,\varepsilon}(t)$.
\end{proof}

Then by using the arguments in section $2$ of \cite{JWLXZ1} (see also Song-Tian's work \cite{JSGT}), we have the following lemmas.
\begin{lem}\label{201902} For any $T>0$, there exists a constant $C$ depending only on $\|\varphi_0\|_{L^\infty(M)}$, $n$, $\beta$, $\omega_{0}$ and $T$, such that for any $t\in(0,T]$, $\varepsilon>0$ and $\mu_\gamma\in[0,1]$,
\begin{equation}\label{201903}
\frac{t^n}{C}\leqslant\frac{(\omega^\gamma_{\varepsilon}+\sqrt{-1}\partial\bar{\partial}\phi_{\gamma,\varepsilon}(t))^{n}}{(\omega^\gamma_{\varepsilon})^n}\leqslant e^{\frac{C}{t}}.
\end{equation}
\end{lem}
\begin{lem}\label{201904} For any $T>0$, there exists a constant $C$ depending only on $\|\varphi_0\|_{L^\infty(M)}$, $n$, $\beta$, $\omega_{0}$ and $T$ such that for any $t\in(0,T]$, $\varepsilon>0$ and $\mu_\gamma\in[0,1]$,
\begin{equation}\label{201905}
e^{-\frac{C}{t}}\omega^\gamma_{\varepsilon}\leqslant\omega_{\gamma,\varepsilon}(t)\leqslant e^{\frac{C}{t}}\omega^\gamma_{\varepsilon}.
\end{equation}
\end{lem}
Hence away from time $0$, on any compact subset in $M\setminus D$, $\omega_{\gamma,\varepsilon}$ are uniformly equivalent to $\omega_0$. Then Evans-Krylov-Safonov's estimates (see \cite{NVK}) imply the high order regularities.
\begin{lem}\label{201906} For any $0<\delta<T<\infty$, $k\in\mathbb{N}^{+}$ and $B_r(p)\subset\subset M\setminus D$, there exist constants $C_{\delta,T,k,p,r}$, such that for any $\varepsilon>0$ and $\mu_\gamma\in[0,1]$,
\begin{equation}\label{201907}
\|\varphi_{\gamma,\varepsilon}(t)\|_{C^{k}\big([\delta,T]\times B_r(p)\big)}\leqslant C_{\delta,T,k,p,r},
\end{equation}
where constants $C_{\delta,T,k,p,r}$ depend on $\|\varphi_0\|_{L^\infty(M)}$, $n$, $\beta$, $\delta$, $k$, $T$, $\omega_0$ and $dist_{\omega_{0}}(B_r(p),D)$.
\end{lem}
Next, we prove the uniform estimates along the twisted K\"ahler-Ricci flows $(TKRF_{0,\varepsilon})$ and the convergence of the conical K\"ahler-Ricci flows $(CKRF_{0})$. 
\begin{thm}\label{2019021901}
There exists a uniform constant $D_\beta$ such that
\begin{equation}\label{1904000}
\|u_{\beta,\varepsilon}(t)\|_{C^0(M)}+osc_{M}\varphi_{\beta,\varepsilon}(t)+\sup\limits_M tr_{\omega_\varepsilon^\beta}\omega_{\beta,\varepsilon}(t) \leqslant D_{\beta}
\end{equation}
for any $\varepsilon\in(0,\frac{1}{2}]$ and $t\geqslant 1$.
\end{thm}
\begin{proof}  From Lemma \ref{201902}, $\|\dot{\varphi}_{\beta,\varepsilon}(1)\|_{C^0(M)}$ are uniformly bounded for any $\varepsilon>0$. Then the maximum principle implies that $\|\dot{\varphi}_{\beta,\varepsilon}(t)\|_{C^0(M)}$ are uniformly bounded for any $\varepsilon>0$ and $t\geqslant 1$. Since $u_{\beta,\varepsilon}(t)=\dot{\varphi}_{\beta,\varepsilon}(t)+c_{\beta,\varepsilon}(t)$ and $\frac{1}{V}\int_M e^{-u_{\beta,\varepsilon}(t)}dV_{\beta,\varepsilon}(t)=1$, $c_{\beta,\varepsilon}(t)$ and $u_{\beta,\varepsilon}(t)$ are uniformly bounded. Denote $\tilde{\varphi}_{\beta,\varepsilon}(t)=\varphi_{\beta,\varepsilon}(t)-\frac{1}{V}\int_M\varphi_{\beta,\varepsilon}(t) dV_0$, then $\tilde{\varphi}_{\beta,\varepsilon}(t)$ satisfies
\begin{equation}\label{201903}
\frac{1}{V}\int_M \tilde{\varphi}_{\beta,\varepsilon}(t) dV_0=0.
\end{equation}
By using the Green's formula with respect to $\omega_0$ and $-\Delta_{\omega_0}
\tilde{\varphi}_{\beta,\varepsilon}(t)\leqslant n$, we have
\begin{equation}\label{201911}
\tilde{\varphi}_{\beta,\varepsilon}(t,x)=\frac{1}{V}\int_{M}\tilde{\varphi}_{\beta,\varepsilon}(t,y)dV_{0}-\frac{1}{V}\int_{M}\Delta_{g_0}
\tilde{\varphi}_{\beta,\varepsilon}(t,y)G_{0}(x,y)dV_{0}\leqslant C
\end{equation}
for any $t\geqslant 0$ and $\varepsilon>0$, where constant $C$ depends only on $n$ and $\omega_0$. Therefore, there exists a constant $A$ such that $A-\tilde{\varphi}_{\beta,\varepsilon}(t)\geqslant 1$. Since $\chi_\beta$ is positive, using integration by parts, we have
\begin{equation}
\begin{split}
\int_M(A-\tilde{\varphi}_{\beta,\varepsilon}(t))\Delta_{\omega_0}\chi_\beta dV_0=-\int_M\chi_\beta\Delta_{\omega_0}\tilde{\varphi}_{\beta,\varepsilon}(t)dV_0\leqslant C,
\end{split}
\end{equation}
where constant $C$ depends only on $\beta$, $n$ and $\omega_0$. On the other hand, from Campana-Guenancia-P$\breve{a}$un's work \cite{CGP}, there exist constants $\delta$ and $C$ such that
\begin{equation*}
\begin{split}
\Delta_{\omega_0}\chi_\beta&=\frac{tr_{\omega_0}\sqrt{-1}<D's,D's>_h}{(\varepsilon^2+|s|_h^2)^{1-\beta}}-\frac{1}{\beta}\big((\varepsilon^2+|s|_h^2)^\beta-\varepsilon^{2\beta}\big)n\\
&\geqslant \delta(\varepsilon^2+|s|_h^2)^{\beta-1}-C.
\end{split}
\end{equation*}
Since $A-\tilde{\varphi}_{\beta,\varepsilon}(t)$ is positive, we conclude that
\begin{equation}
\int_M(A-\tilde{\varphi}_{\beta,\varepsilon}(t))\Delta_{\omega_0}\chi_\beta dV_0\geqslant \delta\int_M(A-\tilde{\varphi}_{\beta,\varepsilon}(t))((\varepsilon^2+|s|_h^2)^{\beta-1}-C)dV_0.
\end{equation}
Using the normalization $(\ref{201903})$, we obtain
\begin{equation}
\int_M(A-\tilde{\varphi}_{\beta,\varepsilon}(t))(\varepsilon^2+|s|_h^2)^{\beta-1} dV_0\leqslant C.
\end{equation}
Hence there exists a uniform constant such that for any $\varepsilon>0$ and $t\geqslant 0$,
\begin{equation}\label{201912}
\int_M(A-\tilde{\varphi}_{\beta,\varepsilon}(t))dV^\beta_\varepsilon\leqslant C.
\end{equation}
Denote $\tilde{\phi}_{\beta,\varepsilon}(t)=\phi_{\beta,\varepsilon}(t)-\frac{1}{V}\int_M\phi_{\beta,\varepsilon}(t) dV_0$. Since $\chi_\beta$ is uniformly bounded, by $(\ref{201911})$ and $(\ref{201912})$, there exist constants $B$ and $C$ such that for any $\varepsilon>0$ and $t\geqslant 0$,
\begin{equation}\label{201913}
B-\tilde{\phi}_{\beta,\varepsilon}(t)\geqslant 1\ \ \ and\ \ \ V\leqslant\int_M(B-\tilde{\phi}_{\beta,\varepsilon}(t))dV^\beta_\varepsilon\leqslant C.
\end{equation}
Let $f_{\varepsilon t}=B-\tilde{\phi}_{\beta,\varepsilon}(t)\geqslant 1$. For any $\alpha\geqslant 0$ and $t\geqslant 1$,
\begin{equation*}
\begin{split}
C\int_Mf_{\varepsilon t}^{\alpha+1}(\omega^\beta_{\varepsilon})^n&\geqslant \int_Mf_{\varepsilon t}^{\alpha+1}e^{\dot{\phi}_{\beta,\varepsilon}-F_{\beta,\varepsilon}}(\omega^\beta_{\varepsilon})^n\geqslant \int_Mf_{\varepsilon t}^{\alpha+1}\omega^n_{\beta,\varepsilon}(t)\geqslant  \int_Mf_{\varepsilon t}^{\alpha+1}(\omega^n_{\beta,\varepsilon}(t)-(\omega_\varepsilon^\beta)^n)\\
&=\sum\limits_{m=0}^{n-1}\int_Mf_{\varepsilon t}^{\alpha+1}(\omega_{\beta,\varepsilon}(t)-(\omega_\varepsilon^\beta))\wedge\omega^m_{\beta,\varepsilon}(t)\wedge(\omega^{\beta}_{\varepsilon})^{n-m-1}\\
&=\sum\limits_{m=0}^{n-1}\int_Mf_{\varepsilon t}^{\alpha+1}\sqrt{-1}\partial\bar{\partial}\tilde{\phi}_{\beta,\varepsilon}(t)\wedge\omega^m_{\beta,\varepsilon}(t)\wedge(\omega^{\beta}_{\varepsilon})^{n-m-1}\\
&=\sum\limits_{m=0}^{n-1}\frac{(\alpha+1)}{(\frac{\alpha}{2}+1)^2}\int_M\sqrt{-1}\partial f_{\varepsilon t}^{\frac{\alpha}{2}+1}\wedge\bar{\partial}f_{\varepsilon t}^{\frac{\alpha}{2}+1}\wedge\omega^m_{\beta,\varepsilon}(t)\wedge(\omega^{\beta}_{\varepsilon})^{n-m-1}\\
&\geqslant \frac{(\alpha+1)}{(\frac{\alpha}{2}+1)^2}\int_M\sqrt{-1}\partial f_{\varepsilon t}^{\frac{\alpha}{2}+1}\wedge\bar{\partial}f_{\varepsilon t}^{\frac{\alpha}{2}+1}\wedge(\omega^{\beta}_{\varepsilon})^{n-1}.
\end{split}
\end{equation*}
Hence the $L^2$-norm of $\nabla f_{\varepsilon t}^{\frac{\alpha}{2}+1}$ with respect to $\omega_\varepsilon^\beta$ can be bounded as follows
\begin{equation}\label{yewan10}
\|\nabla f_{\varepsilon t}^{\frac{\alpha}{2}+1}\|_{L^2}^2\leqslant C \frac{n(\frac{\alpha}{2}+1)^2}{(\alpha+1)}\int_Mf_{\varepsilon t}^{\alpha+1} ~dV_{\varepsilon}^\beta.
\end{equation}
When $n\geqslant2$, we denote $\kappa=\frac{n}{n-1}$ and $p=\alpha+2$. By using the uniform Sobolev inequality $(\ref{1.5.6})$, we have
\begin{equation}\label{20190001}
\begin{split}
(\int_Mf_{\varepsilon t}^{p\kappa}dV^\beta_{\varepsilon})^{\frac{n-1}{n}}&\leqslant C\int_M(|\nabla f_{\varepsilon t}^{\frac{\alpha}{2}+1}|_{\omega^\beta_\varepsilon}^2+f_{\varepsilon t}^{\alpha+2})dV^\beta_{\varepsilon}\\
&\leqslant C\frac{(\frac{\alpha}{2}+1)^2}{(\alpha+1)}\int_Mf_{\varepsilon t}^{\alpha+1}dV^\beta_{\varepsilon}+C\int_Mf_{\varepsilon t}^{\alpha+2}dV^\beta_{\varepsilon}\\
&\leqslant Cp\int_Mf_{\varepsilon t}^p~dV^\beta_{\varepsilon}.
\end{split}
\end{equation}
Hence there exists uniform constant $C$ such that
\begin{equation}\|f_{\varepsilon t}\|_{L^{p\kappa}}\leqslant (Cp)^{\frac{1}{p}}\|f_{\varepsilon t}\|_{L^p}.\end{equation}
Let $p_i=p_{i-1}\kappa$ and $p_0=\frac{2n}{n-1}$. By Moser's iteration, we have
\begin{equation*}
\|f_{\varepsilon t}\|_{L^{p_i}}\leqslant (Cp_{i-1})^{\frac{1}{p_{i-1}}}\|f_{\varepsilon t}\|_{L^{p_{i-1}}}\leqslant\cdots\leqslant C^{\sum_{j=0}^{i-1}\frac{1}{p_j}}\prod_{j=0}^{i-1}p_j^{\frac{1}{p_j}}\|f_{\varepsilon t}\|_{L^{p_0}}.
\end{equation*}
After letting $i\rightarrow\infty$, we deduce
\begin{equation}\label{yewan13}\|f_{\varepsilon t}\|_{L^{\infty}}\leqslant C^{\sum_{j=0}^{\infty}\frac{1}{p_j}}\prod_{j=0}^{\infty}p_j^{\frac{1}{p_j}}\|f_{\varepsilon t}\|_{L^{\frac{2n}{n-1}}}\leqslant C\|f_{\varepsilon t}\|_{L^{\frac{2n}{n-1}}}.\end{equation}
On the other hand, taking $\alpha=0$ in $(\ref{20190001})$ and using H\"older inequality, we have
\begin{equation}\label{20190002}
\begin{split}
(\int_Mf_{\varepsilon t}^{\frac{2n}{n-1}}dV^\beta_{\varepsilon})^{\frac{n-1}{n}}&\leqslant C\int_Mf_{\varepsilon t}^2~dV^\beta_{\varepsilon}\\
&\leqslant C(\int_Mf_{\varepsilon t}~dV^\beta_{\varepsilon})^{\frac{2}{n+1}}(\int_Mf^{\frac{2n}{n-1}}_{\varepsilon t}~dV^\beta_{\varepsilon})^{\frac{n-1}{n+1}}.
\end{split}
\end{equation}
Combining $(\ref{201913})$ with $(\ref{yewan13})$, for $t\geqslant 1$, we have
\begin{equation}\label{AAA}osc_M\phi_{\beta,\varepsilon}(t)=osc_M\tilde{\phi}_{\beta,\varepsilon}(t)\leqslant B-\inf\limits_M\tilde{\phi}_{\beta,\varepsilon}(t)=\|f_{\varepsilon t}\|_{L^{\infty}}\leqslant C.
\end{equation}
Since $\chi_\beta$ is uniformly bounded, and $osc_M\varphi_{\beta,\varepsilon}(t)$ are uniformly bounded for $t\in[0,1]$,
\begin{equation}\label{20190309001}osc_M\varphi_{\beta,\varepsilon}(t)\leqslant C
\end{equation}
for any $\varepsilon>0$ and $t\geqslant 0$. 

Now we denote $\hat{\phi}_{\beta,\varepsilon}(t)=\phi_{\beta,\varepsilon}(t)-\frac{1}{V}\int_M\phi_{\beta,\varepsilon}(t) e^{-F_0+C_{\beta,\varepsilon}}\frac{dV_0}{(\varepsilon^2+|s_h^2|)^{1-\beta}}$, where $C_{\beta,\varepsilon}$ is the normalization constant such that $\frac{1}{V}\int_Me^{-F_0+C_{\beta,\varepsilon}}\frac{dV_0}{(\varepsilon^2+|s|_h^2)^{1-\beta}}=1$. This normalization implies that for any $\varepsilon\in(0,\frac{1}{2}]$,
\begin{equation}
e^{-C_{\beta,\varepsilon}}=\frac{1}{V}\int_Me^{-F_0}\frac{dV_0}{(\varepsilon^2+|s|_h^2)^{1-\beta}}\geqslant \frac{1}{V}\int_Me^{-F_0}dV_0=1,
\end{equation}
Hence $C_{\beta,\varepsilon}\leqslant0$. Since $\frac{1}{V}\int_M\hat{\phi}_{\beta,\varepsilon}(t)e^{-F_0+C_{\beta,\varepsilon}}\frac{dV_0}{(\varepsilon^2+|s|_h^2)^{1-\beta}}=0$, we have
\begin{equation}
\sup\limits_M\hat{\phi}_{\beta,\varepsilon}(t)\geqslant 0\ \ and\ \ \inf\limits_M\hat{\phi}_{\beta,\varepsilon}(t)\leqslant0.
\end{equation}
From $(\ref{AAA})$, we have $osc_M\hat{\phi}_{\beta,\varepsilon}(t)\leqslant C$. Hence we conclude that 
\begin{equation}\label{201903040001}
\|\hat{\phi}_{\beta,\varepsilon}(t)\|_{C^0(M)}\leqslant C
\end{equation}
for any $\varepsilon\in(0,\frac{1}{2}]$ and $t\geqslant 0$. On the other hand, from Lemma \ref{201904}, there exists uniform constant $C$ such that for any $\varepsilon\in(0,\frac{1}{2}]$,
\begin{equation}
\frac{1}{C}\omega^\beta_{\varepsilon}\leqslant\omega_{\beta,\varepsilon}(1)\leqslant C\omega^\beta_{\varepsilon}.
\end{equation}
By using the arguments in the proof of Proposition $3.1$ in \cite{JWLXZ}, for any $t>1$, there exist uniform constant $B$ and $C$ such that
\begin{equation*}
\begin{split}
tr_{\omega_\varepsilon^\beta}\omega_{\beta,\varepsilon}(t)&\leqslant exp\Big(C+B\phi_{\beta,\varepsilon}(t)-B\inf\limits_{[1,t]\times M}\phi_{\beta,\varepsilon}(s)\Big)\\
&\leqslant exp\Big(C+B\hat{\phi}_{\beta,\varepsilon}(t)+\frac{B}{V}\int_M\phi_{\beta,\varepsilon}(t) e^{-F_0+C_{\beta,\varepsilon}}\frac{dV_0}{(\varepsilon^2+|s|_h^2)^{1-\beta}}\\
&\ \ \ \ \ \ \ \ \ \ \ \ -B\inf\limits_{[1,t]\times M}\hat{\phi}_{\beta,\varepsilon}(s)-\inf\limits_{[1,t]}\frac{B}{V}\int_M\phi_{\beta,\varepsilon}(s) e^{-F_0+C_{\beta,\varepsilon}}\frac{dV_0}{(\varepsilon^2+|s|_h^2)^{1-\beta}}\Big)\\
&\leqslant C\ exp\Big(\frac{B}{V}\int_M\phi_{\beta,\varepsilon}(t) e^{-F_0+C_{\beta,\varepsilon}}\frac{dV_0}{(\varepsilon^2+|s|_h^2)^{1-\beta}}\\
&\ \ \ \ \ \ \ \ \ \ \ \ -\inf\limits_{[1,t]}\frac{B}{V}\int_M\phi_{\beta,\varepsilon}(s) e^{-F_0+C_{\beta,\varepsilon}}\frac{dV_0}{(\varepsilon^2+|s|_h^2)^{1-\beta}}\Big).
\end{split}
\end{equation*}
By Jensen's inequality, we have 
\begin{equation*}
\begin{split}
&\ \frac{d}{dt}\frac{1}{V}\int_M\phi_{\beta,\varepsilon}(t) e^{-F_0+C_{\beta,\varepsilon}}\frac{dV_0}{(\varepsilon^2+|s|_h^2)^{1-\beta}}\\
&=\frac{1}{V}\int_M\Big(\log\frac{\omega^n_{\beta,\varepsilon}(t)}{\omega^n_0} +F_0+(1-\beta)\log(\varepsilon^2+|s|_h^2)\Big)e^{-F_0+C_{\beta,\varepsilon}}\frac{dV_0}{(\varepsilon^2+|s|_h^2)^{1-\beta}}\\
&\leqslant \log \Big(\frac{1}{V}\int_M e^{C_{\beta,\varepsilon}}dV_{\beta,\varepsilon}(t)\Big)=C_{\beta,\varepsilon}\leqslant0.
\end{split}
\end{equation*}
Hecne $\inf\limits_{[1,t]}\frac{1}{V}\int_M\phi_{\beta,\varepsilon}(s) e^{-F_0+C_{\beta,\varepsilon}}\frac{dV_0}{(\varepsilon^2+|s|_h^2)^{1-\beta}}=\frac{1}{V}\int_M\phi_{\beta,\varepsilon}(t) e^{-F_0+C_{\beta,\varepsilon}}\frac{dV_0}{(\varepsilon^2+|s|_h^2)^{1-\beta}}$, and then there exists uniform constant $C$ such that
\begin{equation} \label{201903040002}
tr_{\omega_\varepsilon^\beta}\omega_{\beta,\varepsilon}(t)\leqslant C\ \ \ \ \ on\ \ \ \ [1,\infty)\times M
\end{equation}
for any $\varepsilon\in(0,\frac{1}{2}]$.

When $n=1$, we can not use the uniform Sobolev inequality $(\ref{1.5.6})$. The estimates for $osc_M\varphi_{\gamma,\varepsilon}(t)$ follows from Ko{\l}odziej's $L^p$-estimates \cite{K000,K} directly. The other proofs are similar as above.
\end{proof}
Since $\|\dot{\varphi}_{\beta,\varepsilon}(t)\|_{C^0(M)}$ are uniformly bounded, we have 
\begin{equation*}
 tr_{\omega_{\beta,\varepsilon}(t)}\omega_\varepsilon^\beta\leqslant\frac{1}{(n-1)!}( tr_{\omega_\varepsilon^\beta}\omega_{\beta,\varepsilon}(t))^{n-1}\frac{(\omega_\varepsilon^\beta)^n}{\omega_{\beta,\varepsilon}^{n}(t)}\leqslant C\exp(-\dot{\varphi}_{\beta,\varepsilon}(t)+F_{\beta,\varepsilon})\leqslant C
\end{equation*}
for any $t\geqslant 1$ and $\varepsilon\in(0,\frac{1}{2}]$. Combining this inequality with $(\ref{201903040002})$, we get 
\begin{equation}
\frac{1}{C}\omega^\beta_{\varepsilon}\leqslant\omega_{\beta,\varepsilon}(t)\leqslant C\omega^\beta_{\varepsilon} \ \ on \ \ [1,\infty)\times M,
\end{equation}
and then obtain the uniform $C^\infty_{loc}$-estimates in $M\setminus D$ along $(TKRF_{0,\varepsilon})$ for any $\varepsilon\in(0,\frac{1}{2}]$ and $t\geqslant 1$. At the same time, the twisted Mabuchi energy $\mathcal{M}_{0,\varepsilon}(\varphi_{\beta,\varepsilon}(t))$ are uniformly bounded from below. Then by using the smooth approximation and the arguments in section $7$ of \cite{JWLXZ}, we get the convergence of the conical K\"ahler-Ricci flows $(CKRF_{0})$. In \cite{CW1}, Chen-Wang obtained this result by considering the conical K\"ahler-Ricci flow directly. In fact, by using similar arguments as above, we get the following convergence theorem.
\begin{thm}\label{20190002}
Let $M$ be a K\"ahler manifold, $\omega_0$ be a smooth K\"ahler metric and $D$ be a smooth divisor. Assume that the K\"ahler current $\hat{\omega}\in c_{1}(M)$ admits $L^{p}$-density with respect to $\omega_{0}^{n}$ for some $p>1$ and satisfies $\int_{M}\hat{\omega}^{n}=\int_{M}\omega_{0}^{n}$. If $c_1(M)=(1-\gamma)[D]$ with $\gamma\in(0,1)$, then the conical K\"ahler-Ricci flow $(CKRF_{0})$ converges to a Ricci flat conical K\"ahler-Einstein metric with cone angle $2\pi\gamma$ along $D$ in $C_{loc}^{\infty}$-topology outside divisor $D$ and globally in $C^{\alpha,\gamma}$-sense for any $\alpha\in(0,\min\{1,\frac{1}{\gamma}-1\})$.
\end{thm}
\begin{pro}\label{201906001}
There exists a constant $\hat{\delta}$, such that for any $t\in[0,\hat{\delta})$, $\varepsilon\in[0,\frac{1}{2}]$ and $z\in M$, $\varphi_{\gamma,\varepsilon}(t)$ is increasing with respect to $\gamma\in[\beta,1]$.
\end{pro}
\begin{proof} We first consider the case $\varepsilon\in(0,\frac{1}{2}]$. By Ko{\l}odziej's results \cite{K000, K}, there exists a H\"older continuous solution $u_{\gamma,\varepsilon}$ to the equation
\begin{equation}(\omega_{0}+\sqrt{-1}\partial\bar{\partial}u_{\gamma,\varepsilon})^{n}=e^{-F_{0}-\mu_\gamma\varphi_{0}+\hat{C}}\frac{\omega_{0}^{n}}
{(\varepsilon^2+|s|_{h}^{2})^{(1-\gamma)}},
\end{equation}
and $u_{\gamma,\varepsilon}$ satisfies 
\begin{equation}\|u_{\gamma,\varepsilon}\|_{L^\infty(M)}\leqslant C,
\end{equation}
where the normalization constant $\hat{C}$ is uniformly bounded independent of $\varepsilon$ and $\gamma$, constant $C$ depends only on $\|\varphi_0\|_{L^\infty(M)}$, $\beta$ and $F_0$. We define function
\begin{equation}\psi_{\gamma,\varepsilon}(t)=(1-te^{\mu_\gamma t})\varphi_{0}+te^{\mu_\gamma t}u_{\gamma,\varepsilon}+h(t)e^{\mu_\gamma t},
\end{equation}
where
\begin{equation}
\begin{split}
h(t)&=-t\|\varphi_{0}\|_{L^\infty(M)}-t\|u_{\gamma,\varepsilon}\|_{L^\infty(M)}+n(t\log t-t)e^{-\mu_\gamma t}\\
&\ +\mu_\gamma n\int_0^te^{-\mu_\gamma s}s\log s ds+\hat{C}\int_0^t e^{-\mu_\gamma s} ds.
\end{split}
\end{equation}
From the arguments in the proof of Proposition $3.5$ in \cite{JWLXZ1}, we know that $\psi_{\gamma,\varepsilon}(t)$ is a subsolution of equation $(\ref{TKRF7})$ and hence
\begin{equation}\label{201906002}
\varphi_{\gamma,\varepsilon}(t)\geqslant \psi_{\gamma,\varepsilon}(t)\geqslant \varphi_0+h_1(t),
\end{equation}
where $h_1(t)$ with $h_1(0)=0$ is a continuous function independent of $\gamma$ and $\varepsilon$. Then there exists $\hat{\delta}$, such that $h_1(t)>-\frac{1}{2}$ for any $t\in[0,\hat{\delta}]$. By the normalization $(\ref{nom0})$, we have
\begin{equation}\label{201906003}
\varphi_{\gamma,\varepsilon}(t)\geqslant 1-\frac{1}{2}>0
\end{equation}
for any $\gamma\in[\beta,1]$, $\varepsilon\in(0,1]$, $t\in(0,\hat{\delta})$ and $z\in M$. For $\beta\leqslant\gamma_1\leqslant\gamma_2\leqslant1$, on $[0,\hat{\delta}]\times M$,
\begin{equation*}\label{201906004}
\begin{split}
\frac{\partial}{\partial t}\Big(\varphi_{\gamma_1,\varepsilon}(t)-\varphi_{\gamma_2,\varepsilon}(t)\Big)&=\log\frac{(\omega_0+\sqrt{-1}\partial\bar{\partial}\varphi_{\gamma_1,\varepsilon}(t))^n}{(\omega_0+\sqrt{-1}\partial\bar{\partial}\varphi_{\gamma_2,\varepsilon}(t))^n}+\mu_{\gamma_1}\Big(\varphi_{\gamma_1,\varepsilon}(t)-\varphi_{\gamma_2,\varepsilon}(t)\Big)\\
&\ \ \ \ -(\mu_{\gamma_2}-\mu_{\gamma_1})\varphi_{\gamma_2,\varepsilon}(t)+(\gamma_2-\gamma_1)\log(\varepsilon^2+|s|_h^2)\\
&\leqslant\log\frac{(\omega_0+\sqrt{-1}\partial\bar{\partial}\varphi_{\gamma_1,\varepsilon}(t))^n}{(\omega_0+\sqrt{-1}\partial\bar{\partial}\varphi_{\gamma_2,\varepsilon}(t))^n}+\mu_{\gamma_1}\Big(\varphi_{\gamma_1,\varepsilon}(t)-\varphi_{\gamma_2,\varepsilon}(t)\Big).
\end{split}
\end{equation*}
Then the maximum principle implies that $\varphi_{\mu_{\gamma_1},\varepsilon}(t)\leqslant\varphi_{\mu_{\gamma_2},\varepsilon}(t)$ for any $\varepsilon\in(0,\frac{1}{2}]$, $t\in(0,\hat{\delta})$ and $z\in M$. Letting $\varepsilon\rightarrow0$, we get the $\varepsilon=0$ case.
\end{proof}
\begin{pro}\label{201908}
Let $\{\varepsilon_i\}\in[0,\frac{1}{2}]$ and $\{\gamma_i\}\in[\beta,1]$ (that is, $\mu_{\gamma_i}\in[0,1]$). We assume that $\varepsilon_i$ and $\gamma_i$ converge to $\varepsilon_\infty$ and $\gamma_\infty$ respectively. For any $[\delta,T]$ $(0<\delta<T<\infty)$, there exists $\alpha\in(0,1)$ such that $\varphi_{\gamma_i,\varepsilon_i}(t)$ converge to $\varphi_{\gamma_\infty,\varepsilon_\infty}(t)$ in $C^{\alpha'}$-sense on $[\delta,T]\times M$ for any $\alpha'\in(0,\alpha)$, and on $(0,\infty)\times (M\setminus D)$ the convergence is in $C^\infty_{loc}$-sense.
\end{pro}
\begin{proof} On $[\delta,T]\times M$ with $0<\delta<T<\infty$, by Lemma \ref{201901} and Lemma \ref{201902}, $\varphi_{\gamma_i,\varepsilon_i}(t)$ and $\dot{\varphi}_{\gamma_i,\varepsilon_i}(t)$ are uniformly bounded. By Ko{\l}odziej's $L^{p}$-estimates \cite{K000, K}, there exists $\alpha\in(0,1)$ such that $\|\varphi_{\gamma_i,\varepsilon_i}(t)\|_{C^\alpha(M, \omega_0)}$ are uniformly bounded. Then there exists a subsequence $\varphi_{\gamma_{i_k},\varepsilon_{i_k}}(t)$ converge to a function $\varphi_\infty(t)$ in $C^{\alpha'}$-sense on $[\delta,T]\times M$ with $\alpha'\in(0,\alpha)$, and in $C^\infty_{loc}$-sense on $(0,\infty)\times (M\setminus D)$ from Lemma \ref{201906}.  Furthermore, on $(0,\infty)\times M$ (or $(0,\infty)\times (M\setminus D)$ if $\varepsilon_\infty=0$), $\varphi_\infty(t)$ satisfies equation
\begin{equation*}
\frac{\partial \varphi_{\gamma_\infty,\varepsilon_\infty}(t)}{\partial t}=\log\frac{(\omega_{0}+\sqrt{-1}\partial\bar{\partial}\varphi_{\gamma_\infty,\varepsilon_\infty}(t))^{n}}{\omega_{0}^{n}}+\mu_{\gamma_\infty}\varphi_{\gamma_\infty,\varepsilon_\infty}(t)+F_{0}+\log(\varepsilon_\infty^{2}+|s|_{h}^{2})^{1-\gamma_\infty},
\end{equation*}

We claim that $\varphi_\infty(t)=\varphi_{\gamma_\infty,\varepsilon_\infty}(t)$. Fix $\gamma$, $t$ and $z\in M$, by Proposition $3.3$ in \cite{JWLXZ1}, $\varphi_{\gamma,\varepsilon}(t)$ is decreasing as $\varepsilon\searrow0$. Combining this with Proposition \ref{201906001} and $(\ref{201906002})$, we have
\begin{equation}\label{201906006}
h_1(t)\leqslant\varphi_{\gamma_{i_k},\varepsilon_{i_k}}(t)-\varphi_0\leqslant\varphi_{\gamma_{i_k},\frac{1}{2}}(t)-\varphi_0\leqslant\varphi_{1,\frac{1}{2}}(t)-\varphi_0
\end{equation}
on $[0,\hat{\delta})\times M$ for any $\gamma_{i_k}$ and $\varepsilon_{i_k}$. Let $k\rightarrow\infty$, we have
\begin{equation}\label{201906007}
h_1(t)\leqslant\varphi_{\infty}(t)-\varphi_0\leqslant\varphi_{1,\frac{1}{2}}(t)-\varphi_0.
\end{equation}\label{201906008}
Since $\varphi_{1,\frac{1}{2}}(t)$ converge to $\varphi_0$ in $L^\infty$-sense as $t\rightarrow0$ and $h_1(0)=0$, $\varphi_{\infty}(t)$ converge to $\varphi_0$ in $L^\infty$-sense as $t\rightarrow0$. By the uniqueness results (see Proposition $2.7$ or Theorem $3.7$ in \cite{JWLXZ1}), we prove the claim. Next, we prove that $\varphi_{\gamma_i,\varepsilon_i}(t)$ converge to $\varphi_{\gamma_\infty,\varepsilon_\infty}(t)$ in $C^{\alpha'}$-sense on $[\delta,T]\times M$ for any $\alpha'\in(0,\alpha)$. If this is not true, then there exists $\alpha_{0}\in(0,\alpha)$, $\epsilon_0>0$ and a sequence $\varphi_{\gamma_{j_k},\varepsilon_{j_k}}(t)$ such that
\begin{equation}\label{3.22.6016}\|\varphi_{\gamma_{j_k},\varepsilon_{j_k}}(t)-\varphi_{\gamma_\infty,\varepsilon_\infty}(t)\|_{C^{\alpha_{0}}([\delta,T]\times M)}\geqslant \epsilon_0.\end{equation}
Since $\|\varphi_{\gamma_{j_k},\varepsilon_{j_k}}(t)\|_{C^{\alpha^{\prime}}([\delta,T]\times M)}$ are uniformly bounded for some $\alpha^{\prime}\in(\alpha_{0},\alpha)$, there exists a subsequence which we also denote it by $\varphi_{\gamma_{j_k},\varepsilon_{j_k}}(t)$ such that $\varphi_{\gamma_{j_k},\varepsilon_{j_k}}(t)$ converge in $C^{\alpha_{0}}$-sense to a function $\tilde{\varphi}_{\infty}(t)$ as $k\rightarrow\infty$ and
\begin{equation}\label{3.22.6116}\|\tilde{\varphi}_{\infty}(t)-\varphi_{\gamma_\infty,\varepsilon_\infty}(t)\|_{C^{\alpha_{0}}([\delta,T]\times M)}\geqslant \epsilon_{0}.\end{equation}
By the similar arguments as above, $\tilde{\varphi}_\infty(t)$ is also a solution of the equation $(\ref{TKRF7})$ with $\mu_{\gamma_\infty}$ and $\varepsilon_\infty$. However, $\tilde{\varphi}_{\infty}(t)\not \equiv\varphi_{\gamma_\infty,\varepsilon_\infty}(t)$ by $(\ref{3.22.6116})$, which is impossible by the uniqueness theorem. Hence we prove the $C^\alpha$-convergence. From the uniform $C^{\infty}_{loc}$-estimates for $\varphi_{\gamma_i,\varepsilon_i}(t)$, we can also prove that $\varphi_{\gamma_i,\varepsilon_i}(t)$ converge to $\varphi_{\gamma_\infty,\varepsilon_\infty}(t)$ in $C^{\infty}_{loc}$-topology on $(0,\infty)\times (M\setminus D)$ by the similar arguments.
\end{proof}
\begin{rem}\label{201908001}
As a consequence of Proposition \ref{201908}, on $[\delta,T]\times M$, the $C^\alpha$-norm of $\varphi_{\gamma,\varepsilon}(t)$ on $[\delta,T]\times M$ (for some $\alpha\in(0,1)$) is continuous with respect to $\gamma\in[\beta,1]$ and $\varepsilon\in[0,\frac{1}{2}]$. Furthermore, If we fix $\varepsilon\in(0,\frac{1}{2}]$, then the $C^\infty$-norm of $\varphi_{\gamma,\varepsilon}(t)$ on $[\delta,T]\times M$ is continuous with respect to $\gamma\in[\beta,1]$.
\end{rem}
Straightforward calculation shows that 
\begin{equation}
\begin{split}
\label{3.22.16}&\ \ \ (\frac{d}{dt}-\Delta_{\omega_{\gamma,\varepsilon}(t)})(\Delta_{\omega_{\gamma,\varepsilon}(t)} u_{\gamma,\varepsilon}(t))\\
&=-|\nabla\overline{\nabla}u_{\gamma,\varepsilon}(t)|_{\omega_{\gamma,\varepsilon}(t)}^{2}+\mu_\gamma\Delta_{\omega_{\gamma,\varepsilon}(t)} u_{\gamma,\varepsilon}(t),\\
&\ \ \ (\frac{d}{dt}-\Delta_{\omega_{\gamma,\varepsilon}(t)})(|\nabla u_{\gamma,\varepsilon}(t)|_{\omega_{\gamma,\varepsilon}(t)}^{2})\\
&=-|\nabla\overline{\nabla}u_{\gamma,\varepsilon}(t)|_{\omega_{\gamma,\varepsilon}(t)}^{2}-|\nabla\nabla u_{\gamma,\varepsilon}(t)|_{\omega_{\gamma,\varepsilon}(t)}^{2}+\mu_\gamma|\nabla u_{\gamma,\varepsilon}(t)|_{\omega_{\gamma,\varepsilon}(t)}^{2}\\
&\ \ \ \ \ -\frac{1-\gamma}{2}\theta_{\varepsilon}(grad\ u_{\gamma,\varepsilon}(t),\mathcal{J}(grad\ u_{\gamma,\varepsilon}(t))),
\end{split}
\end{equation}
where $\mathcal{J}$ is the complex structure on $M$.
\begin{lem}\label{1.8.601}  If $\|u_{\gamma,\varepsilon}(t)\|_{C^0(M)}\leqslant A$ on $[T_1,T_2]$ with ($1<T_1<T_1+1<T_2\leqslant\infty$). Then there exists uniform constant $C$ depending only on $\|\varphi_0\|_{L^\infty(M)}$, $\beta$, $n$, $\omega_{0}$ and $A$ such that on $M\times [T_1+1,T_2]$,
\begin{equation}
\begin{split}
&\ \label{3.22.18}|\nabla u_{\gamma,\varepsilon}(t)|_{\omega_{\gamma,\varepsilon}(t)}^{2}\leqslant C(u_{\gamma,\varepsilon}(t)+C)\leqslant C,\\
&\ R(\omega_{\gamma,\varepsilon}(t))-(1-\gamma)tr_{\omega_{\gamma,\varepsilon}(t)}\theta_{\varepsilon}\leqslant C(u_{\gamma,\varepsilon}(t)+C)\leqslant C.
\end{split}
\end{equation}
\end{lem}
\begin{proof}
We first prove that there exists uniform constant $C$ such that 
\begin{equation}
\begin{split}
\label{1.5.2}|\nabla u_{\gamma,\varepsilon}(T_1+1)|_{\omega_{\gamma,\varepsilon}(T_1+1)}^{2}&\leqslant C,\\
|R(\omega_{\gamma,\varepsilon}(T_1+1))-(1-\gamma)tr_{\omega_{\gamma,\varepsilon}(T_1+1)}\theta_{\varepsilon}|&\leqslant C.
\end{split}
\end{equation}
Calculation shows that $u_{\gamma,\varepsilon}(t)$ can be written as
\begin{equation}\label{201903120015}
u_{\gamma,\varepsilon}(1)=\log\frac{\omega_{\gamma,\varepsilon}^n(1)(\varepsilon^2+|s|_h^2)^{1-\gamma}}{\omega_0^n}+F_0
+\mu_\gamma\varphi_{\gamma,\varepsilon}(1)+C_{\gamma,\varepsilon,1},
\end{equation}
where $C_{\gamma,\varepsilon,1}$ is the normalization constant such that $\frac{1}{V}\int_M e^{-u_{\gamma,\varepsilon}(1)}dV_{\gamma,\varepsilon}(1)=1$. By Lemma \ref{201901} and
Lemma \ref{201902}, $C_{\gamma, \varepsilon,1}$ and $u_{\gamma, \varepsilon}(1)$ are uniformly bounded by a constant $C$ depending only on $\|\varphi_0\|_{L^\infty(M)}$, $\beta$, $n$ and $\omega_{0}$, and hence $A_{\mu_\gamma,\varepsilon}(1)$ are uniformly bounded. Since $A_{\mu_\gamma,\varepsilon}(t)$ is increasing along $(TKRF_{\mu_\gamma,\varepsilon})$ when $\mu_\gamma>0$, $A_{\mu_\gamma,\varepsilon}(t)$ is uniformly bounded from below for any  $\varepsilon>0$, $\gamma\in(\beta,1]$ and $t\geqslant 1$. On the other hand, by Jensen's inequality, $A_{\mu_\gamma,\varepsilon}(t)\leqslant0$. 

Let $H_{\gamma,\varepsilon}(t)=(t-T_1)|\nabla u_{\gamma,\varepsilon}(t)|^{2}_{\omega_{\gamma,\varepsilon}(t)}+\frac{3}{2}u^{2}_{\gamma,\varepsilon}(t)$ and $( t_0,x_0)$ be the maximum point of $H_{\gamma,\varepsilon}(t)$ on $[T_1,T_1+1]\times M$. Combining $(\ref{3.22.16})$ with
\begin{equation*}
\label{3.22.6}
\frac{\partial}{\partial t}u_{\gamma,\varepsilon}(t)=\Delta_{\omega_{\gamma,\varepsilon}(t)} u_{\gamma,\varepsilon}(t)+\mu_\gamma u_{\gamma,\varepsilon}(t)-\mu_\gamma A_{\mu_\gamma,\varepsilon}(t),
\end{equation*}
we obtain
\begin{equation}(\frac{d}{dt}-\Delta_{\omega_{\gamma,\varepsilon}(t)})H_{\gamma,\varepsilon}(t)\leqslant-|\nabla u_{\gamma,\varepsilon}(t)|^{2}_{\omega_{\gamma,\varepsilon}(t)}+C,\end{equation}
where constant $C$ depends only on $A$, $\|\varphi_0\|_{L^\infty(M)}$, $\beta$, $n$ and $\omega_{0}$.

$Case\ 1$, $t_0=T_1$. Then $(t-T_1)|\nabla u_{\gamma,\varepsilon}(t)|^{2}_{\omega_{\gamma,\varepsilon}(t)}\leqslant \frac{3}{2}u^{2}_{\gamma,\varepsilon}(T_1)\leqslant \frac{3}{2}A^2$.

$Case\ 2$, $t_0>T_1$. By the maximum principle, we have $|\nabla u_{\gamma,\varepsilon}(t_0,x_0)|^{2}_{\omega_{\gamma,\varepsilon}(t_0)}\leqslant C$. Hence $(t-T_1)|\nabla u_{\gamma,\varepsilon}(t)|^{2}_{\omega_{\gamma,\varepsilon}(t)}\leqslant C$.

From the above two cases, $(t-T_1)|\nabla u_{\gamma,\varepsilon}(t)|^{2}_{\omega_{\gamma,\varepsilon}(t)}\leqslant C$ on $[T_1,T_1+1] \times M$. Obviously $|\nabla u_{\gamma,\varepsilon}(T_1+1)|^{2}_{\omega_{\gamma,\varepsilon}(T_1+1)}$ are uniformly bounded by a constant $C$ which depends only on $A$, $\|\varphi_0\|_{L^\infty(M)}$, $\beta$, $n$ and $\omega_{0}$.

Since $\Delta_{\omega_{\gamma,\varepsilon}(t)} u_{\gamma,\varepsilon}(t)=-R(\omega_{\gamma,\varepsilon}(t))+\mu_\gamma n+(1-\gamma)tr_{\omega_{\gamma,\varepsilon}(t)}\theta_{\varepsilon}$, we only need to prove the uniform upper bound of $-\Delta_{\omega_{\gamma,\varepsilon}(t)} u_{\gamma,\varepsilon}(t)$. 

We take $G_{\gamma,\varepsilon}(t)=(t-T_1)^2(-\Delta_{\omega_{\gamma,\varepsilon}(t)} u_{\gamma,\varepsilon}(t))+2(t-T_1)^2|\nabla u_{\gamma,\varepsilon}(t)|^{2}_{\omega_{\gamma,\varepsilon}(t)}$. According to (\ref{3.22.16}) and
\begin{equation}
|\nabla \bar{\nabla}u_{\gamma,\varepsilon}(t)|^{2}_{\omega_{\gamma,\varepsilon}(t)}\geqslant \frac{(\Delta_{\omega_{\gamma,\varepsilon}(t)} u_{\gamma,\varepsilon}(t))^2}{n},
\end{equation}
the evolution equation of $G_{\gamma,\varepsilon}(t)$ can be controlled as
\begin{equation*}
\begin{split}
&\ (\frac{d}{dt}-\Delta_{\omega_{\gamma,\varepsilon}(t)})G_{\gamma,\varepsilon}(t)\\
&\leqslant\big(\mu_\gamma (t-T_1)^2+2(t-T_1)\big)\big(-\Delta_{\omega_{\gamma,\varepsilon}(t)} u_{\gamma,\varepsilon}(t)\big)-\frac{(t-T_1)^2}{n}\big(\Delta_{\omega_{\gamma,\varepsilon}(t)} u_{\gamma,\varepsilon}(t)\big)^2+C_0
\end{split}
\end{equation*}
where constant $C_0$ depends only on $A$, $\|\varphi_0\|_{L^\infty(M)}$, $\beta$, $n$ and $\omega_{0}$. Assuming that $(t_0,x_0)$ is the maximum point of $G_{\gamma,\varepsilon}(t)$ on $[T_1,T_1+1] \times M$.

$Case\ 1$, $t_0=T_1$, then $(t-T_1)^2(-\Delta_{\omega_{\gamma,\varepsilon}(t)} u_{\gamma,\varepsilon}(t))\leqslant0$.

$Case\ 2$, $t_0>T_1$. We assume $-\Delta_{\omega_{\gamma,\varepsilon}(t)} u_{\gamma,\varepsilon}(t)> 0$ at $(t_0,x_0)$ without loss of generality. We claim that $(t_{0}-T_1)^{2}(-\Delta_{\omega_{\gamma,\varepsilon}(t_0)} u_{\gamma,\varepsilon}(t_0,x_0))\leqslant n(3+C_0)$. If not, by the maximum principle, we have
\begin{equation*}
\begin{split}0&\leqslant\big(-\Delta_{\omega_{\gamma,\varepsilon}(t_0)} u_{\gamma,\varepsilon}(t_0,x_0)\big)\Big(3-\frac{(t_{0}-T_1)^2}{n}\big(-\Delta_{\omega_{\gamma,\varepsilon}(t_0)} u_{\gamma,\varepsilon}(t_0,x_0)\big)\Big)+C_0\\
&<-nC_0(3+C_0)+C_0<0.
\end{split}
\end{equation*}
We get a contradiction. From these two cases, we conclude that $-(t-T_1)^2\Delta_{\omega_{\gamma,\varepsilon}(t)} u_{\gamma,\varepsilon}(t)\leqslant C$ on $[T_1,T_1+1]\times M$ for some constant $C$ depending only on $A$, $\|\varphi_0\|_{L^\infty(M)}$, $\beta$, $n$ and $\omega_{0}$. Furthermore, $R(\omega_{\gamma,\varepsilon}(T_1+1))-(1-\gamma)tr_{\omega_{\gamma,\varepsilon}(T_1+1)}\theta_{\varepsilon}=-\Delta_{\omega_{\gamma,\varepsilon}(T_1+1)} u_{\gamma,\varepsilon}(T_1+1)+\mu_\gamma n\leqslant C$.

Since $\|u_{\gamma,\varepsilon}(t)\|_{C^0(M)}\leqslant A$ on $[T_1,T_2]$, there exists a uniform constant $B>1$ such that $u_{\gamma,\varepsilon}(t)>-B$ on $[T_1,T_2]$. Define $H_{\gamma,\varepsilon}(t)=\frac{|\nabla u_{\gamma,\varepsilon}(t)|_{\omega_{\gamma,\varepsilon}(t)}^{2}}{u_{\gamma,\varepsilon}(t)+2B}$. By arguments in \cite{JWLXZ},
\begin{equation*}
\begin{split}
&\ \ \ (\frac{d}{dt}-\Delta_{\omega_{\gamma,\varepsilon}(t)}) H_{\gamma,\varepsilon}(t)\\
&\leqslant\frac{-|\nabla\overline{\nabla}u_{\gamma,\varepsilon}(t)|_{\omega_{\gamma,\varepsilon}(t)}^{2}-|\nabla\nabla u_{\gamma,\varepsilon}(t)|_{\omega_{\gamma,\varepsilon}(t)}^2}{u_{\gamma,\varepsilon}(t)+2B}+\frac{|\nabla u_{\gamma,\varepsilon}(t)|_{\omega_{\gamma,\varepsilon}(t)}^{2}(2B\mu_\gamma+\mu_\gamma A_{\mu_\gamma,\varepsilon}(t))}{(u_{\gamma,\varepsilon}(t)+2B)^{2}}\\
&\ +\frac{\delta}{2}\frac{|\nabla u_{\gamma,\varepsilon}(t)|_{\omega_{\gamma,\varepsilon}(t)}^4}{(u_{\gamma,\varepsilon}(t)+2B)^{3}}+\delta\frac{|\nabla\nabla u_{\gamma,\varepsilon}(t)|_{\omega_{\gamma,\varepsilon}(t)}^{2}+|\nabla\overline{\nabla}u_{\gamma,\varepsilon}(t)|_{\omega_{\gamma,\varepsilon}(t)}^{2}}{u_{\gamma,\varepsilon}(t)+2B}
-\delta\frac{|\nabla u_{\gamma,\varepsilon}(t)|_{\omega_{\gamma,\varepsilon}(t)}^4}{(u_{\gamma,\varepsilon}(t)+2B)^{3}}\\
&\ -\frac{1-\gamma}{2}\frac{\theta_{\varepsilon}(grad\ u_{\gamma,\varepsilon}(t),\mathcal{J}(grad\ u_{\gamma,\varepsilon}(t)))}{u_{\gamma,\varepsilon}(t)+2B}+(2-\delta)Re\frac{\overline{\nabla}u_{\gamma,\varepsilon}(t)\cdot\nabla H_{\gamma,\varepsilon}(t)}{u_{\gamma,\varepsilon}(t)+2B}.
\end{split}
\end{equation*}
Taking $\delta<1$, since $\theta_\varepsilon(grad\ u_{\gamma,\varepsilon}(t),\mathcal{J}(grad\ u_{\gamma,\varepsilon}(t)))\geqslant  0$, we have
\begin{equation*}
\begin{split}
&\ \ \ (\frac{d}{dt}-\Delta_{\omega_{\gamma,\varepsilon}(t)}) H_{\gamma,\varepsilon}(t)\\
&\leqslant\frac{|\nabla u_{\gamma,\varepsilon}(t)|_{\omega_{\gamma,\varepsilon}(t)}^{2}(2B+C_{1})}{(u_{\gamma,\varepsilon}(t)+2B)^{2}}+(2-\delta)Re\frac{\overline{\nabla}u_{\gamma,\varepsilon}(t)\cdot\nabla H_{\gamma,\varepsilon}(t)}{u_{\gamma,\varepsilon}(t)+2B}-\frac{\delta}{2}\frac{|\nabla u_{\gamma,\varepsilon}(t)|_{\omega_{\gamma,\varepsilon}(t)}^{4}}{(u_{\gamma,\varepsilon}(t)+2B)^{3}}.
\end{split}
\end{equation*}
From $(\ref{1.5.2})$, we have
\begin{equation}\label{1.5.5}\sup\limits_{M}\frac{|\nabla u_{\gamma,\varepsilon}(T_1+1)|_{\omega_{\gamma,\varepsilon}(T_1+1)}^{2}}{u_{\gamma,\varepsilon}(T_1+1)+2B}\leqslant C_{2},\end{equation}
where $C_{2}$ depends only on $A$, $\|\varphi_0\|_{L^\infty(M)}$, $\beta$, $n$ and $\omega_{0}$. The maximum principle implies that $H_{\gamma,\varepsilon}(t)\leqslant \max(C_{2},\ 2(2B+C_{1})\delta^{-1})$ for $t\in [T_1+1,T_2]$. We get the first inequality in $(\ref{3.22.18})$.

Now we prove the second inequality. Let $G_{\gamma,\varepsilon}(t)=\frac{-\Delta_{\omega_{\gamma,\varepsilon}(t)} u_{\gamma,\varepsilon}(t)}{u_\gamma,{\varepsilon}(t)+2B}+2H_{\gamma,\varepsilon}(t)$.
\begin{equation*}
\begin{split}
&\ \ \ (\frac{d}{dt}-\Delta_{\omega_{\gamma,\varepsilon}(t)}) G_{\gamma,\varepsilon}(t)=\frac{-2|\nabla\nabla u_{\gamma,\varepsilon}(t)|_{\omega_{\gamma,\varepsilon}(t)}^{2}-|\nabla\overline{\nabla}u_{\gamma,\varepsilon}(t)|_{\omega_{\gamma,\varepsilon}(t)}^{2}}{u_{\gamma,\varepsilon}(t)+2B}\\
&+\frac{(-\Delta_{\omega_{\gamma,\varepsilon}(t)} u_{\gamma,\varepsilon}(t)+2|\nabla u_{\gamma,\varepsilon}(t)|_{\omega_{\gamma,\varepsilon}(t)}^{2})(2B\mu_\gamma+\mu_\gamma A_{\mu_\gamma,\varepsilon}(t))}{(u_{\gamma,\varepsilon}(t)+2B)^{2}}
\\
&\ +2Re\frac{\overline{\nabla}u_{\gamma,\varepsilon}(t)\cdot\nabla G_{\gamma,\varepsilon}(t)}{u_{\gamma,\varepsilon}(t)+2B}-\frac{1-\gamma}{2}\frac{\theta_{\varepsilon}(grad\ u_{\gamma,\varepsilon}(t),\mathcal{J}(grad\ u_{\gamma,\varepsilon}(t)))}{u_{\gamma,\varepsilon}(t)+2B}.
\end{split}
\end{equation*}
Since $\theta_{\varepsilon}$ is semi-positive,
\begin{equation*}
\begin{split}
(\frac{d}{dt}-\Delta_{\omega_{\gamma,\varepsilon}(t)}) G_{\gamma,\varepsilon}(t)&\leqslant\frac{-|\nabla\overline{\nabla}u_{\gamma,\varepsilon}(t)|_{\omega_{\gamma,\varepsilon}(t)}^{2}}{u_{\gamma,\varepsilon}(t)+2B}+2Re \frac{\overline{\nabla}u_{\gamma,\varepsilon}(t)\cdot\nabla G_{\gamma,\varepsilon}(t)}{u_{\gamma,\varepsilon}(t)+2B}\\
&+\frac{(-\Delta_{\omega_{\gamma,\varepsilon}(t)} u_{\gamma,\varepsilon}(t)+2|\nabla u_{\gamma,\varepsilon}(t)|_{\omega_{\gamma,\varepsilon}(t)}^{2})(2B\mu_\gamma+\mu_\gamma A_{\mu_\gamma,\varepsilon}(t))}{(u_{\gamma,\varepsilon}(t)+2B)^{2}}.
\end{split}
\end{equation*}
By using inequality 
\begin{equation}\label{3.22.25}(\Delta_{\omega_{\gamma,\varepsilon}(t)} u_{\gamma,\varepsilon}(t))^{2}\leqslant n|\nabla\overline{\nabla}u_{\gamma,\varepsilon}(t)|_{\omega_{\gamma,\varepsilon}(t)}^{2},
\end{equation}
we have
\begin{equation*}
\begin{split}
(\frac{d}{dt}-\Delta_{\omega_{\gamma,\varepsilon}(t)}) G_{\gamma,\varepsilon}(t)
&\leqslant\frac{-\Delta_{\omega_{\gamma,\varepsilon}(t)} u_{\gamma,\varepsilon}(t)}{u_{\gamma,\varepsilon}(t)+2B}(\frac{2B\mu_\gamma+\mu_\gamma A_{\mu_\gamma,\varepsilon}(t)}{u_{\gamma,\varepsilon}(t)+2B}-\frac{-\Delta_{\omega_{\gamma,\varepsilon}(t)} u_{\gamma,\varepsilon}(t)}{n(u_{\gamma,\varepsilon}(t)+2B)})\\
&\ +2\frac{|\nabla u_{\gamma,\varepsilon}(t)|_{\omega_{\gamma,\varepsilon}(t)}^{2}(2B\mu_\gamma+\mu_\gamma A_{\mu_\gamma,\varepsilon}(t))}{(u_{\gamma,\varepsilon}(t)+2B)^{2}}+2Re \frac{\overline{\nabla}u_{\gamma,\varepsilon}(t)\cdot\nabla G_{\gamma,\varepsilon}(t)}{u_{\gamma,\varepsilon}(t)+2B}.
\end{split}
\end{equation*}
Since $\frac{-\Delta_{\omega_{\gamma,\varepsilon}(T_1+1)} u_{\gamma,\varepsilon}(T_1+1)}{u_{\gamma,\varepsilon}(T_1+1)+2B}$ are bounded uniformly, by the maximum principle, there exists a uniform constant $C$ depending only on $A$, $\|\varphi_0\|_{L^\infty(M)}$, $\beta$, $n$ and $\omega_{0}$ such that $G_{\gamma,\varepsilon}(t)\leqslant C$ for any $t\in[T_1+1,T_2]$. Hence we get $-\Delta_{\omega_{\gamma,\varepsilon}(t)}u_{\gamma,\varepsilon}(t)\leqslant C(u_{\gamma,\varepsilon}(t)+2B)$ and hence $R(\omega_{\gamma,\varepsilon}(t))-(1-\gamma)tr_{\omega_{\gamma,\varepsilon}(t)}\theta_{\varepsilon}\leqslant C(u_{\gamma,\varepsilon}(t)+C)$ on $[T_1+1,T_2]\times M$.
\end{proof}
Now, we recall some uniform Sobolev inequalities along the twisted K\"ahler-Ricci flows $(TKRF_{\mu_\gamma,\varepsilon})$. From the proof of Theorem $6.1$ in \cite{JWLXZ} (see also \cite{LW,RGY,QSZ}), when $t\geqslant1$, the Sobolev constants along $(TKRF_{\mu_\gamma,\varepsilon})$ depend only on $n$, $\omega_0$, $\max\limits_{M}(R(\omega_{\gamma,\varepsilon}(1))-(1-\gamma)tr_{\omega_{\gamma,\varepsilon}(1)}\theta_{\varepsilon})^{-}$ and $\mathcal{C}_{S}(M,\omega_{\gamma,\varepsilon}(1))$, where the latter two are uniformly bounded from Theorem \ref{Sobolev}, Lemma \ref{1902} and Lemma \ref{201904} . Hence we have the following Sobolev inequality.
\begin{thm}\label{1.8.5.1} Let  $M$ be Fano manifold of complex dimension $n\geq2$ and $\omega_{\gamma,\varepsilon}(t)$ be a solution of the twisted K\"ahler-Ricci flow $(TKRF_{\mu_\gamma,\varepsilon})$. There exist uniform constants A and B depending only on $\|\varphi_0\|_{L^\infty(M)}$, $\beta$, $n$ and $\omega_0$ such that 
\begin{equation*}
\begin{split}
\label{3.22.31}(\int_{M}v^{\frac{2n}{n-1}}dV_{\gamma,\varepsilon}(t))^{\frac{n-1}{n}}&\leqslant A\Big(\int_{M}4|\nabla v|_{\omega_{\gamma,\varepsilon}(t)}^{2}dV_{\gamma,\varepsilon}(t)\\
&\ \ \ \ \ \ \ \ \ +\int_{M}(R(\omega_{\gamma,\varepsilon}(t))-(1-\gamma)tr_{\omega_{\gamma,\varepsilon}(t)}\theta_{\varepsilon}+B)v^{2}dV_{\gamma,\varepsilon}(t)\Big)
\end{split}
\end{equation*}
for any $v\in W^{1,2}(M,\omega_{\gamma,\varepsilon}(t))$, $\gamma\in(\beta,1]$, $\varepsilon>0$ and $t\geq 1$.
\end{thm}
By using Hus's work (see Theorem $1$ in \cite{SYH}) and the arguments in the proof of Theorem $6.1$ in \cite{JWLXZ} (see also \cite{LW,RGY,QSZ}), we have the following uniform Sobolev inequality for $n=1$.
\begin{thm}\label{1.8.5.10001} Let  $M$ be a Fano manifold of complex dimension $1$ and $\omega_{\gamma,\varepsilon}(t)$ be a solution of the twisted K\"ahler-Ricci flow $(TKRF_{\mu_\gamma,\varepsilon})$. For $n_0>1$, there exist uniform constants A and B depending only on $\|\varphi_0\|_{L^\infty(M)}$, $\beta$, $n_0$ and $\omega_0$ such that
\begin{equation*}
\begin{split}
\label{3.22.310001}(\int_{M}v^{\frac{2n_0}{n_0-1}}dV_{\gamma,\varepsilon}(t))^{\frac{n_0-1}{n_0}}&\leqslant A\Big(\int_{M}4|\nabla v|_{\omega_{\gamma,\varepsilon}(t)}^{2}dV_{\gamma,\varepsilon}(t)\\
&\ \ \ \ \ \ \ \ \ +\int_{M}(R(\omega_{\gamma,\varepsilon}(t))-(1-\gamma)tr_{\omega_{\gamma,\varepsilon}(t)}\theta_{\varepsilon}+B)v^{2}dV_{\gamma,\varepsilon}(t)\Big)
\end{split}
\end{equation*}
for any $v\in W^{1,2}(M,\omega_{\gamma,\varepsilon}(t))$, $\gamma\in(\beta,1]$, $\varepsilon>0$ and $t\geq 1$.
\end{thm}
Combining these uniform Sobolev inequalitities with Lemma \ref{1902}, we have the following Theorem by following Jiang's work (Theorem $1.12$ in \cite{WSJ1}). 
\begin{thm}\label{Jiang}
Let  $M$ be a Fano manifold of complex dimension $n$ and $\omega_{\gamma,\varepsilon}(t)$ be a solution of the twisted K\"ahler-Ricci flow $(TKRF_{\mu_\gamma,\varepsilon})$. Let $f$ be a non-negative Lipschitz continuous function on $[0,\infty)\times M$ satisfying 
\begin{equation}\label{20190306001}
\frac{\partial}{\partial t}f\leqslant\Delta_{\omega_{\gamma,\varepsilon}(t)}f+af
\end{equation}
on $[0,\infty)\times M$ in the weak sense, where $a\geqslant0$. For $p>0$, there exists a constant $C$ depending only on $\|\varphi_0\|_{L^\infty(M)}$, $\beta$, $a$, $p$, $\omega_0$ and $n_0$ ($n_0=n$ if $n\geqslant2$, $n_0>1$ if $n=1$) such that
\begin{equation}\label{20190306002}
\sup\limits_M|f(t,x)|\leqslant\frac{C}{(t-T)^{\frac{n_0+1}{p}}}\Big(\int_T^{T+1}\int_M f^p dV_{\gamma,\varepsilon}(t)dt\Big)^{\frac{1}{p}}
\end{equation}
holds for any $T<t<T+1$ with $T\geqslant1$, $\gamma\in(\beta,1]$ and $\varepsilon>0$.
\end{thm}
Let $\omega_{\varphi_{\beta}}=\omega_0+\sqrt{-1}\partial\bar{\partial}\varphi_\beta$ be the Ricci flat conical K\"ahler-Einstein metric obtained in Theorem \ref{20190002}, then $\varphi_\beta\in C^0(M)\cap C^\infty(M\setminus D)$ satisfies
\begin{equation}\label{201902230101}
(\omega_0+\sqrt{-1}\partial\bar{\partial}\varphi_\beta)^n=e^{-F_0+\hat{C}_\beta}\frac{\omega_0^n}{|s|_h^{2(1-\beta)}},
\end{equation}
where $\hat{C}_\beta$ is the normalization constant such that $\frac{1}{V}\int_Me^{-F_0+\hat{C}_\beta}\frac{dV_0}{|s|_h^{2(1-\beta)}}=1$. By using the normalization of $\frac{1}{V}\int_Me^{-F_0} dV_0=1$,  we conclude that $\hat{C}_\beta=-\log\frac{1}{V}\int_Me^{-F_0}\frac{dV_0}{|s|_h^{2(1-\beta)}}\leqslant0$. By Ko{\l}odziej's $L^p$-estimates  \cite{K000, K}, we have 
\begin{equation}\label{20190223010100}
osc_M \varphi_\beta\leqslant K_\beta.
\end{equation}
At the same time, by Guenancia-P$\breve{a}$un's results \cite{CGP,GP1} (see also Liu-Zhang \cite{JWLCJZ}), there exists constant $M_\beta$ such that 
\begin{equation}\label{20190223010100}
\sup\limits_{M\setminus D} tr_{\omega_\beta}\omega_{\varphi_\beta} \leqslant M_\beta.
\end{equation}
We denote
\begin{equation}\label{20190304} L_\beta=\max(K_\beta+M_\beta, D_\beta),
\end{equation}
where $D_\beta$ is the constant in $(\ref{1904000})$. Next, we prove Lemma \ref{1901}.

\medskip

{\it Proof of Lemma \ref{1901}.} \ \  If this lemma is not true. For $\delta_1<\min(\frac{1}{2},1-\beta)$, there exist $\varepsilon_1\in(0,\delta_1)$, $\gamma'_1\in(\beta,\beta+\delta_1)$ and $t'_1\in[\frac{1}{\delta_1},\infty)$, such that
\begin{equation*}\label{1905}
\|u_{\gamma'_1,\varepsilon_1}(t'_1)\|_{C^0(M)}+osc_{M}\varphi_{\gamma'_1,\varepsilon_1}(t'_1)+\sup\limits_{M} tr_{\omega_{\varepsilon_1}^{\gamma'_1}}\omega_{\gamma'_1,\varepsilon_1}(t'_1) >L_{\beta}+1.
\end{equation*}
Assume $t'_1\in[\frac{1}{\delta_1},T_1]$, then we have
\begin{equation*}\label{1906}
\sup\limits_{t\in[\frac{1}{\delta_1},T_1]}\Big(\|u_{\gamma'_1,\varepsilon_1}(t)\|_{C^0(M)}+osc_{M}\varphi_{\gamma'_1,\varepsilon_1}(t)+\sup\limits_M tr_{\omega_{\varepsilon_1}^{\gamma'_1}}\omega_{\gamma'_1,\varepsilon_1}(t) \Big)>L_{\beta}+1.
\end{equation*}
Theorem \ref{2019021901} and Remark \ref{201908001} imply that there exists $\gamma_1\in(\beta,\gamma'_1)$, such that
\begin{equation*}\label{1907}
\sup\limits_{t\in[\frac{1}{\delta_1},T_1]}\Big(\|u_{\gamma_1,\varepsilon_1}(t)\|_{C^0(M)}+osc_{M}\varphi_{\gamma_1,\varepsilon_1}(t)+\sup\limits_M tr_{\omega_{\varepsilon_1}^{\gamma_1}}\omega_{\gamma_1,\varepsilon_1}(t) \Big)=L_{\beta}+1.
\end{equation*}
Hence there exists $t_1\in[\frac{1}{\delta_1},T_1]$ such that
\begin{equation*}\label{1908}
\|u_{\gamma_1,\varepsilon_1}(t_1)\|_{C^0(M)}+osc_{M}\varphi_{\gamma_1,\varepsilon_1}(t_1)+\sup\limits_M tr_{\omega_{\varepsilon_1}^{\gamma_1}}\omega_{\gamma_1,\varepsilon_1}(t_1) \geqslant L_{\beta}+\frac{8}{9},
\end{equation*}
and then
\begin{equation*}\label{1909}
\|u_{\gamma_1,\varepsilon_1}(t_1)\|_{L^\infty(M\setminus D)}+osc_{M}\varphi_{\gamma_1,\varepsilon_1}(t_1)+\sup\limits_{M\setminus D} tr_{\omega_{\varepsilon_1}^{\gamma_1}}\omega_{\gamma_1,\varepsilon_1}(t_1) \geqslant L_{\beta}+\frac{7}{8}.
\end{equation*}
For $\delta_2=\min(\frac{1}{2}, \frac{1}{T_1+1}, \varepsilon_1, \gamma_1-\beta)$, there exist $\varepsilon_2\in(0,\delta_2)$, $\gamma'_2\in(\beta,\beta+\delta_2)$ and $t'_2\in[\frac{1}{\delta_2},\infty)$, such that
\begin{equation*}\label{1910}
\|u_{\gamma'_2,\varepsilon_2}(t'_2)\|_{C^0(M)}+osc_{M}\varphi_{\gamma'_2,\varepsilon_2}(t'_2)+\sup\limits_M tr_{\omega_{\varepsilon_2}^{\gamma'_2}}\omega_{\gamma'_2,\varepsilon_2}(t'_2) >L_{\beta}+1
\end{equation*}
Assume $t'_2\in[T_1+1,T_2]$, we have
\begin{equation*}\label{1911}
\sup\limits_{t\in[T_1+1,T_2]}\Big(\|u_{\gamma'_2,\varepsilon_2}(t)\|_{C^0(M)}+osc_{M}\varphi_{\gamma'_2,\varepsilon_2}(t)+\sup\limits_M tr_{\omega_{\varepsilon_2}^{\gamma'_2}}\omega_{\gamma'_2,\varepsilon_2}(t)\Big)>L_{\beta}+1.
\end{equation*}
Theorem \ref{2019021901} and Remark \ref{201908001} imply that there exists $\gamma_2\in(\beta,\gamma'_2)$, such that
\begin{equation*}\label{1912}
\sup\limits_{t\in[T_1+1,T_2]}\Big(\|u_{\gamma_2,\varepsilon_2}(t)\|_{C^0(M)}+osc_{M}\varphi_{\gamma_2,\varepsilon_2}(t)+\sup\limits_M tr_{\omega_{\varepsilon_2}^{\gamma_2}}\omega_{\gamma_2,\varepsilon_2}(t)\Big)=L_{\beta}+1.
\end{equation*}
Hence there exists $t_2\in[T_1+1,T_2]$ such that
\begin{equation*}\label{1913}
\|u_{\gamma_2,\varepsilon_2}(t_2)\|_{C^0(M)}+osc_{M}\varphi_{\gamma_2,\varepsilon_2}(t_2)+\sup\limits_M tr_{\omega_{\varepsilon_2}^{\gamma_2}}\omega_{\gamma_2,\varepsilon_2}(t_2)\geqslant L_{\beta}+\frac{8}{9},
\end{equation*}
and then
\begin{equation*}\label{1914}
\|u_{\gamma_2,\varepsilon_2}(t_2)\|_{L^\infty(M\setminus D)}+osc_{M}\varphi_{\gamma_2,\varepsilon_2}(t_2)+\sup\limits_{M\setminus D} tr_{\omega_{\varepsilon_2}^{\gamma_2}}\omega_{\gamma_2,\varepsilon_2}(t_2)\geqslant L_{\beta}+\frac{7}{8}.
\end{equation*}
After repeating above process, we get a subsequence $\varphi_{\gamma_i,\varepsilon_i}(t_i)$ with $\varepsilon_i\searrow0$, $\gamma_i\searrow\beta$, $t_i\nearrow\infty$ and $t_i\in[T_{i-1}+1,T_{i}]$ satisfying
\begin{equation}\label{1915}
\sup\limits_{t\in[T_{i-1}+1,T_{i}]}\Big(\|u_{\gamma_i,\varepsilon_i}(t)\|_{C^0(M)}+osc_{M}\varphi_{\gamma_i,\varepsilon_i}(t)+\sup\limits_M tr_{\omega_{\varepsilon_i}^{\gamma_i}}\omega_{\gamma_i,\varepsilon_i}(t)\Big)=L_{\beta}+1
\end{equation}
and
\begin{equation} \label{2000019}
\|u_{\gamma_i,\varepsilon_i}(t_i)\|_{L^\infty(M\setminus D)}+osc_{M}\varphi_{\gamma_i,\varepsilon_i}(t_i)+\sup\limits_{M\setminus D} tr_{\omega_{\varepsilon_i}^{\gamma_i}}\omega_{\gamma_i,\varepsilon_i}(t_i)\geqslant L_{\beta}+\frac{7}{8}.
\end{equation}
We claim that there exists a uniform constant $A$ depending only on $\|\varphi_0\|_{L^\infty(M)}$, $\beta$, $\omega_0$ and $n$, such that
\begin{equation}\label{1916}
\sup\limits_{t\in[T_{i-1}+\frac{1}{2},T_{i}]}\Big(\|u_{\gamma_i,\varepsilon_i}(t)\|_{C^0(M)}+osc_{M}\varphi_{\gamma_i,\varepsilon_i}(t)+\sup\limits_M tr_{\omega_{\varepsilon_i}^{\gamma_i}}\omega_{\gamma_i,\varepsilon_i}(t)\Big)\leqslant A.
\end{equation}
Combining the arguments below $(\ref{201903120015})$ with Lemma \ref{1902}, we have
\begin{equation*}
\begin{split}\label{3.22.6}
\frac{\partial}{\partial t}u_{\gamma,\varepsilon}(t)&=\Delta_{\omega_{\gamma,\varepsilon}(t)} u_{\gamma,\varepsilon}(t)+\mu_\gamma u_{\gamma,\varepsilon}(t)-\mu_\gamma A_{\mu_\gamma,\varepsilon}(t),\\
&\leqslant \mu_\gamma u_{\gamma,\varepsilon}(t)+C.
\end{split}
\end{equation*}
Then for the above sequence, we have
\begin{equation}\label{1917}
\frac{\partial}{\partial t}(e^{-\mu_{\gamma_i}t}u_{\gamma_i,\varepsilon_i}(t))\leqslant Ce^{-\mu_{\gamma_i}t}.
\end{equation}
Integrating form $t$ to $T_{i-1}+1$ on both sides, where $t\in[T_{i-1},T_{i-1}+1]$, we obtain
\begin{equation*}
\label{1918}
e^{-\mu_{\gamma_i}(T_{i-1}+1)}u_{\gamma_i,\varepsilon_i}(T_{i-1}+1)\leqslant e^{-\mu_{\gamma_i}t}u_{\gamma_i,\varepsilon_i}(t)+C\int^{T_{i-1}+1}_{t}e^{-\mu_{\gamma_i}s}ds.
\end{equation*}
Hence there exists a uniform constant $C$ such that
\begin{equation}\label{191911}
u_{\gamma_i,\varepsilon_i}(t)\geqslant -C
\end{equation}
for any $t\in[T_{i-1}, T_{i-1}+1]$, $\varepsilon_i$ and $\gamma_i$. By using Lemma \ref{1902} again, we have 
\begin{equation}\label{201903060000001}
\begin{split}
&\ \int_M \Big|R(\omega_{\gamma,\varepsilon}(t))-(1-\gamma)tr_{\omega_{\gamma,\varepsilon}(t)}\theta_{\varepsilon}\Big| dV_{\gamma,\varepsilon}(t)\\
&=\int_M  \Big|\Delta_{\omega_{\gamma,\varepsilon}(t)}u_{\gamma,\varepsilon}(t)-\mu_\gamma n+C-C\Big| dV_{\gamma,\varepsilon}(t)\\
&=\int_M  \Big(C-\Delta_{\omega_{\gamma,\varepsilon}(t)}u_{\gamma,\varepsilon}(t)+\mu_\gamma n+C\Big) dV_{\gamma,\varepsilon}(t)\leqslant C
\end{split}
\end{equation}
for any $\gamma\in[\beta,1]$, $\varepsilon>0$ and $t\geqslant1$. Since
\begin{equation}\label{1916001}
\frac{d}{dt}\mathcal{M}_{\mu_\gamma,\varepsilon}(\varphi_{\gamma,\varepsilon}(t))=-\int_M|\nabla u_{\gamma,\varepsilon}(t)|^2_{\omega_{\gamma,\varepsilon}(t)}dV_{\gamma,\varepsilon}(t),
\end{equation}
the twisted Mabuchi energy $\mathcal{M}_{\mu_\gamma,\varepsilon}$ is decreasing along the twisted K\"ahler-Ricci flow $(TKRF_{\mu_\gamma,\varepsilon})$. By Lemma \ref{201901} and Lemma \ref{201902}, we have
\begin{equation*}
\begin{split}
\mathcal{M}_{\mu_\gamma,\varepsilon}(\varphi_{\gamma,\varepsilon}(t))&\leqslant\mathcal{M}_{\mu_\gamma,\varepsilon}(\varphi_{\gamma,\varepsilon}(1))\\
&=-\mu_\gamma\Big(I_{\omega_{0}}(\varphi_{\gamma,\varepsilon}(1))-J_{\omega_{0}}(\varphi_{\gamma,\varepsilon}(1))\Big)+\frac{1}{V}\int_{M}\log\frac{\omega^n_{\gamma,\varepsilon}(1)}{\omega_{0}^{n}}dV_{\gamma,\varepsilon}(1)\\
&\ \ \ \ \ -\frac{1}{V}\int_{M}\Big(F_0+(1-\gamma)\log(\varepsilon^2+|s|^2_h)\Big)(dV_{0}-dV_{\gamma,\varepsilon}(1))\\
&\leqslant C
\end{split}
\end{equation*}
for any $\gamma\in[\beta,1]$, $\varepsilon>0$ and $t\geqslant1$. On the other hand, we write equation $(\ref{TKRF7})$ as 
\begin{equation}\label{1922010}
(\omega_0+\sqrt{-1}\partial\bar{\partial}\varphi_{\gamma,\varepsilon}(t))^n=e^{\dot{\varphi}_{\gamma,\varepsilon}(t)-\mu_{\gamma}\varphi_{\gamma,\varepsilon}(t)-F_0}\frac{\omega_0^n}{(\varepsilon^2+|s|_h^2)^{(1-\gamma)}}.
\end{equation}
Integrating above equation on both sides, there exists uniform constant such that for any $t\geqslant1$, $\gamma\in[\beta,1]$ and $\varepsilon>0$,
\begin{equation}\label{192202}
\sup\limits_M\big(\dot{\varphi}_{\gamma,\varepsilon}(t)-\mu_{\gamma}\varphi_{\gamma,\varepsilon}(t)\big)\geqslant C\ \ \ and\ \ \ \inf\limits_M\big(\dot{\varphi}_{\gamma,\varepsilon}(t)-\mu_{\gamma}\varphi_{\gamma,\varepsilon}(t)\big)\leqslant C.
\end{equation}
By using $(\ref{1915})$, we have
\begin{equation*}\label{190312001}
\begin{split}
osc_M\big(\dot{\varphi}_{\gamma_i,\varepsilon_i}(t)-\mu_{\gamma_i}\varphi_{\gamma_i,\varepsilon_i}(t)\big)&\leqslant osc_M\dot{\varphi}_{\gamma_i,\varepsilon_i}(t)+\mu_{\gamma_i}osc_M\varphi_{\gamma_i,\varepsilon_i}(t)\\
&=osc_M u_{\gamma_i,\varepsilon_i}(t)+\mu_{\gamma_i}osc_M\varphi_{\gamma_i,\varepsilon_i}(t)\leqslant C
\end{split}
\end{equation*}
for any $t\in[T_{i-1}+1,T_{i}]$, $\gamma_i$ and $\varepsilon_i$. Hence $\|\dot{\varphi}_{\gamma_i,\varepsilon_i}(t)-\mu_i\varphi_{\gamma_i,\varepsilon_i}(t)\|_{C^0(M)}$ are uniformly bounded on $[T_{i-1}+1,T_{i}]$. Combining this with $(\ref{1915})$, we have
\begin{equation*}
\begin{split}
\mathcal{M}_{\mu_{\gamma_i},\varepsilon_i}(\varphi_{\gamma_i,\varepsilon_i}(T_{i-1}+1))&=-\mu_{\gamma_i}\Big(I_{\omega_{0}}(\varphi_{\gamma_i,\varepsilon_i}(T_{i-1}+1))-J_{\omega_{0}}(\varphi_{\gamma_i,\varepsilon_i}(T_{i-1}+1))\Big)\\
&\ \ \ \ \ +\frac{1}{V}\int_{M}\log\frac{\omega^n_{\gamma_i,\varepsilon_i}(T_{i-1}+1)}{\omega_{0}^{n}}dV_{\gamma_i,\varepsilon_i}(T_{i-1}+1)\\
&\ \ \ \ \ -\frac{1}{V}\int_{M}\Big(F_0+(1-\gamma)\log(\varepsilon^2+|s|^2_h)\Big)(dV_{0}-dV_{\gamma,\varepsilon}(T_{i-1}+1))\\
&=-\mu_{\gamma_i}\Big(I_{\omega_{0}}(\varphi_{\gamma_i,\varepsilon_i}(T_{i-1}+1))-J_{\omega_{0}}(\varphi_{\gamma_i,\varepsilon_i}(T_{i-1}+1))\Big)\\
&\ \ \ \ \ +\frac{1}{V}\int_{M}\Big(\dot{\varphi}_{\gamma_i,\varepsilon_i}(T_{i-1}+1)-\mu_i\varphi_{\gamma_i,\varepsilon_i}(T_{i-1}+1)\Big)dV_{\gamma_i,\varepsilon_i}(T_{i-1}+1)\\
&\ \ \ \ \ -\frac{1}{V}\int_{M}\Big(F_0+(1-\gamma)\log(\varepsilon^2+|s|^2_h)\Big)dV_{0}\\
&\geqslant -C.
\end{split}
\end{equation*}
Integrating $(\ref{1916001})$ from $T_{i-1}$ to $T_{i-1}+1$ on both sides, we have
\begin{equation}\label{201903060000002}
\begin{split}
&\ \int_{T_{i-1}}^{T_{i-1}+1}\int_M|\nabla u_{\gamma_i,\varepsilon_i}(t)|^2_{\omega_{\gamma_i,\varepsilon_i}(t)}dV_{\gamma_i,\varepsilon_i}(t)dt\\
&=-\mathcal{M}_{\mu_{\gamma_i},\varepsilon_i}(\varphi_{\gamma_i,\varepsilon_i}(T_{i-1}+1))+\mathcal{M}_{\mu_{\gamma_i},\varepsilon_i}(\varphi_{\gamma_i,\varepsilon_i}(T_{i-1}))\\
&\leqslant C,
\end{split}
\end{equation}
for any $\gamma_i$ and $\varepsilon_i$. Straightforward calculation shows that
\begin{equation*}
\begin{split}
&\ \ \ (\frac{\partial}{\partial t}-\Delta_{\omega_{\gamma,\varepsilon}(t)})\Big(R(\omega_{\gamma,\varepsilon}(t))-(1-\gamma)tr_{\omega_{\gamma,\varepsilon}(t)}\theta_{\varepsilon}+4n+|\nabla u_{\gamma,\varepsilon}(t)|_{\omega_{\gamma,\varepsilon}(t)}^{2}\Big)\\
&=\mu_\gamma \Big(R(\omega_{\gamma,\varepsilon}(t))-(1-\gamma)tr_{\omega_{\gamma,\varepsilon}(t)}\theta_{\varepsilon}+|\nabla u_{\gamma,\varepsilon}(t)|_{\omega_{\gamma,\varepsilon}(t)}^{2}\Big)-|\nabla\nabla u_{\gamma,\varepsilon}(t)|_{\omega_{\gamma,\varepsilon}(t)}^{2}\\
&\ \ \ \ -\frac{1-\gamma}{2}\theta_{\varepsilon}(grad\ u_{\gamma,\varepsilon}(t),\mathcal{J}(grad\ u_{\gamma,\varepsilon}(t)))-n\mu_\gamma^2\\
&\leqslant\mu_\gamma \Big(R(\omega_{\gamma,\varepsilon}(t))-(1-\gamma)tr_{\omega_{\gamma,\varepsilon}(t)}\theta_{\varepsilon}+4n+|\nabla u_{\gamma,\varepsilon}(t)|_{\omega_{\gamma,\varepsilon}(t)}^{2}\Big).
\end{split}
\end{equation*}
Applying Theorem \ref{Jiang} to $\varphi_{\gamma_i,\varepsilon_i}(t)$, we have
\begin{equation*}\label{2019030600201}
\begin{split}
&\ \sup\limits_M\Big|R(\omega_{\gamma_i,\varepsilon_i}(t))-(1-\gamma_i)tr_{\omega_{\gamma_i,\varepsilon_i}(t)}\theta_{\varepsilon_i}+4n+|\nabla u_{\gamma_i,\varepsilon_i}(t)|^2_{\omega_{\gamma_i,\varepsilon_i}(t)}\Big|\\
&\leqslant C\frac{\Big(\int_{T_{i-1}}^{T_{i-1}+1}\int_M \Big|R(\omega_{\gamma_i,\varepsilon_i}(t))-(1-\gamma)tr_{\omega_{\gamma_i,\varepsilon_i}(t)}\theta_{\varepsilon_i}+4n+|\nabla u_{\gamma_i,\varepsilon_i}(t)|^2_{\omega_{\gamma_i,\varepsilon_i}(t)}\Big| dV_{\gamma_i,\varepsilon_i}(t)dt\Big)}{(t-T_{i-1})^{n+1}}\\
&\leqslant \frac{C}{(t-T_{i-1})^{n+1}}
\end{split}
\end{equation*}
for any $t\in(T_{i-1},T_{i-1}+1]$, $\gamma_i$ and $\varepsilon_i$ by using $(\ref{201903060000001})$ and $(\ref{201903060000002})$. Hence there exists uniform constant $C$ such that for any $\gamma_i$, $\varepsilon_i$ and $t\in[T_{i-1}+\frac{1}{4} ,T_{i-1}+1]$,
\begin{equation}\label{2019030701}
R(\omega_{\gamma_i,\varepsilon_i}(t))-(1-\gamma)tr_{\omega_{\gamma_i,\varepsilon_i}(t)}\theta_{\varepsilon_i}\leqslant C\ \ \ and\ \ \ |\nabla u_{\gamma_i,\varepsilon_i}(t)|^2_{\omega_{\gamma_i,\varepsilon_i}(t)}\leqslant C.
\end{equation}
By Jensen's inequality and normalization $\frac{1}{V}\int_M e^{-u_{\gamma,\varepsilon}(t)}dV_{\gamma,\varepsilon}(t)=1$, we have
\begin{equation}\label{20190308001}
\begin{split}
A_{\mu_\gamma,\varepsilon}(t)&=\frac{1}{V}\int_M u_{\gamma,\varepsilon}(t)e^{-u_{\gamma,\varepsilon}(t)}dV_{\gamma,\varepsilon}(t)\\
&\leqslant\log\Big(\frac{1}{V}\int_M e^{u_{\gamma,\varepsilon}(t)}e^{-u_{\gamma,\varepsilon}(t)}dV_{\gamma,\varepsilon}(t)\Big)=0.
\end{split}
\end{equation}
When $\mu_\gamma\geqslant 0$, combining this with $(\ref{2019030701})$, we control the evolution of $u_{\gamma_i,\varepsilon_i}(t)$ as 
\begin{equation*}
\begin{split}\label{190312002}
\frac{\partial}{\partial t}u_{\gamma_i,\varepsilon_i}(t)&=\Delta_{\omega_{\gamma_i,\varepsilon_i}(t)} u_{\gamma_i,\varepsilon_i}(t)+\mu_{\gamma_i} u_{\gamma_i,\varepsilon_i}(t)-\mu_{\gamma_i} A_{\mu_{\gamma_i},\varepsilon_i}(t),\\
&\geqslant\mu_{\gamma_i} u_{\gamma_i,\varepsilon_i}(t)-C.
\end{split}
\end{equation*}
Then we have
\begin{equation}\label{190312003}
\frac{\partial}{\partial t}(e^{-\mu_{\gamma_i}t}u_{\gamma_i,\varepsilon_i}(t))\geqslant -Ce^{-\mu_{\gamma_i}t}.
\end{equation}
Integrating form $t$ to $T_{i-1}+1$ on both sides, where $t\in[T_{i-1}+\frac{1}{4} ,T_{i-1}+1]$, we obtain
\begin{equation*}\label{190312004}
e^{-\mu_{\gamma_i}(T_{i-1}+1)}u_{\gamma_i,\varepsilon_i}(T_{i-1}+1)\geqslant e^{-\mu_{\gamma_i}t}u_{\gamma_i,\varepsilon_i}(t)-C\int^{T_{i-1}+1}_{t}e^{-\mu_{\gamma_i}s}ds.
\end{equation*}
Hence there exists a uniform constant $C$ such that
\begin{equation}\label{190312005}
u_{\gamma_i,\varepsilon_i}(t)\leqslant C
\end{equation}
for any $t\in[T_{i-1}+\frac{1}{4} , T_{i-1}+1]$, $\varepsilon_i$ and $\gamma_i$. Since $u_{\gamma,\varepsilon}(t)=\dot{\varphi}_{\gamma,\varepsilon}(t)+\hat{c}_{\gamma,\varepsilon}(t)$, where $\hat{c}_{\gamma,\varepsilon}(t)$ is only a function of time $t$, from $(\ref{191911})$ and $(\ref{190312005})$, we have 
\begin{equation}\label{190312006}
\|u_{\gamma_i,\varepsilon_i}(t)\|_{C^0(M)}\leqslant C\ \ \ and\ \ \ osc_M \dot{\varphi}_{\gamma_i,\varepsilon_i}(t)\leqslant C
\end{equation}
for any $t\in[T_{i-1}+\frac{1}{4} , T_{i-1}+1]$, $\varepsilon_i$ and $\gamma_i$. By using $(\ref{2019030701})$ again, we have
\begin{equation*}
\begin{split}\label{19200}
\frac{\partial}{\partial t}\dot{\varphi}_{\gamma_i,\varepsilon_i}(t)&=\Delta_{\omega_{\gamma_i,\varepsilon_i}(t)} \dot{\varphi}_{\gamma_i,\varepsilon_i}(t)+\mu_{\gamma_i} \dot{\varphi}_{\gamma_i,\varepsilon_i}(t)\\
&\geqslant \mu_{\gamma_i} \dot{\varphi}_{\gamma_i,\varepsilon_i}(t)-C.
\end{split}
\end{equation*}
Hence
\begin{equation}\label{191701}
\frac{\partial}{\partial t}(e^{-\mu_{\gamma_i}t}\dot{\varphi}_{\gamma_i,\varepsilon_i}(t))\geqslant - Ce^{-\mu_{\gamma_i}t}.
\end{equation}
Integrating form $t$ to $T_{i-1}+1$ on both sides, where $t\in[T_{i-1}+\frac{1}{4} ,T_{i-1}+1]$, we obtain
\begin{equation}
\label{1903120010}
e^{-\mu_{\gamma_i}(T_{i-1}+1)}\dot{\varphi}_{\gamma_i,\varepsilon_i}(T_{i-1}+1)\geqslant e^{-\mu_{\gamma_i}t}\dot{\varphi}_{\gamma_i,\varepsilon_i}(t)-C\int^{T_{i-1}+1}_{t}e^{-\mu_{\gamma_i}s}ds,
\end{equation}
and then
\begin{equation}\label{1903120011}
e^{-\mu_{\gamma_i}(T_{i-1}+1-t)}\dot{\varphi}_{\gamma_i,\varepsilon_i}(T_{i-1}+1)\geqslant \dot{\varphi}_{\gamma_i,\varepsilon_i}(t)-C\int^{T_{i-1}+1}_{t}e^{-\mu_{\gamma_i}(s-t)}ds.
\end{equation}
By using Lemma \ref{1902}, we have
\begin{equation*}
\begin{split}\label{190312007}
\frac{\partial}{\partial t}\dot{\varphi}_{\gamma_i,\varepsilon_i}(t)&=\Delta_{\omega_{\gamma_i,\varepsilon_i}(t)} \dot{\varphi}_{\gamma_i,\varepsilon_i}(t)+\mu_{\gamma_i} \dot{\varphi}_{\gamma_i,\varepsilon_i}(t)\\
&\leqslant \mu_{\gamma_i} \dot{\varphi}_{\gamma_i,\varepsilon_i}(t)+C.
\end{split}
\end{equation*}
Hence
\begin{equation}\label{190312008}
\frac{\partial}{\partial t}(e^{-\mu_{\gamma_i}t}\dot{\varphi}_{\gamma_i,\varepsilon_i}(t))\leqslant Ce^{-\mu_{\gamma_i}t}.
\end{equation}
Integrating form $t$ to $T_{i-1}+1$ on both sides, where $t\in[T_{i-1}+\frac{1}{4} ,T_{i-1}+1]$, we obtain
\begin{equation}
\label{190312009}
e^{-\mu_{\gamma_i}(T_{i-1}+1)}\dot{\varphi}_{\gamma_i,\varepsilon_i}(T_{i-1}+1)\leqslant e^{-\mu_{\gamma_i}t}\dot{\varphi}_{\gamma_i,\varepsilon_i}(t)+C\int^{T_{i-1}+1}_{t}e^{-\mu_{\gamma_i}s}ds.
\end{equation}
Hence 
\begin{equation}\label{1903120012}
e^{-\mu_{\gamma_i}(T_{i-1}+1-t)}\dot{\varphi}_{\gamma_i,\varepsilon_i}(T_{i-1}+1)\leqslant \dot{\varphi}_{\gamma_i,\varepsilon_i}(t)+C\int^{T_{i-1}+1}_{t}e^{-\mu_{\gamma_i}(s-t)}ds.
\end{equation}
Integrating form $t$ to $T_{i-1}+1$ on both sides of $(\ref{1903120011})$ and $(\ref{1903120012})$, where $t\in[T_{i-1}+\frac{1}{4} ,T_{i-1}+1]$,
\begin{equation*}
\begin{split}\label{1903120013}
\varphi_{\gamma_i,\varepsilon_i}(t)&\geqslant \varphi_{\gamma_i,\varepsilon_i}(T_{i-1}+1)-\dot{\varphi}_{\gamma_i,\varepsilon_i}(T_{i-1}+1)\int_t^{T_{i-1}+1}e^{-\mu_{\gamma_i}(T_{i-1}+1-s)}ds-C,\\
\varphi_{\gamma_i,\varepsilon_i}(t)&\leqslant\varphi_{\gamma_i,\varepsilon_i}(T_{i-1}+1)-\dot{\varphi}_{\gamma_i,\varepsilon_i}(T_{i-1}+1)\int_t^{T_{i-1}+1}e^{-\mu_{\gamma_i}(T_{i-1}+1-s)}ds+C.
\end{split}
\end{equation*}
From above arguments, for any $t\in[T_{i-1}+\frac{1}{4} , T_{i-1}+1]$, $\varepsilon_i$ and $\gamma_i$, we get
\begin{equation*}
\begin{split}\label{1903120014}
osc_M\varphi_{\gamma_i,\varepsilon_i}(t)&\leqslant osc_M\varphi_{\gamma_i,\varepsilon_i}(T_{i-1}+1)+osc_M\dot{\varphi}_{\gamma_i,\varepsilon_i}(T_{i-1}+1)\int_t^{T_{i-1}+1}e^{-\mu_{\gamma_i}(T_{i-1}+1-s)}ds+C\\
&\leqslant osc_M\varphi_{\gamma_i,\varepsilon_i}(T_{i-1}+1)+osc_M\dot{\varphi}_{\gamma_i,\varepsilon_i}(T_{i-1}+1)+C\\
&\leqslant C.
\end{split}
\end{equation*}
Let $\psi_{\gamma_i,\varepsilon_i}(t)=\varphi_{\gamma_i,\varepsilon_i}(t)+\tilde{C}_{\gamma_i,\varepsilon_i}e^{\mu_{\gamma_i}t}$ and denote 
\begin{equation*}
\tilde{C}_{\gamma_i,\varepsilon_i}=e^{-\mu_{\gamma_i}}\frac{1}{\mu_{\gamma_i}}\Big(\int_{1}^{T_{i-1}+1}e^{-\mu_{\gamma_i} (t-1)}\|\nabla u_{\gamma_i,\varepsilon_i}(t)\|^{2}_{L^{2}}dt-\frac{1}{V}\int_{M}\big(F_{\gamma_i,\varepsilon_i}(1)+\mu_{\gamma_i}\varphi_{\gamma_i,\varepsilon_i}(1)\big)dV_{\gamma_i,\varepsilon_i}(1)\Big),
\end{equation*}
where $F_{\gamma_i,\varepsilon_i}(1)=F_{0}+\log\Big(\frac{\omega_{\gamma_i,\varepsilon_i}^n(1)(\varepsilon^{2}+|s|_{h}^{2})^{1-\gamma_i}}{\omega_{0}^{n}}\Big)$. We know that $\tilde{C}_{\gamma_i,\varepsilon_i}$ is well-defined from the uniform Perelman's estimates (here the estimates depend on $\mu_{\gamma_i}$) and that $\psi_{\gamma_i,\varepsilon_i}(t)$ is a solution of equation $(\ref{TKRF7})$ with initial data $\varphi_0+\tilde{C}_{\gamma_i,\varepsilon_i}$. From $(\ref{201903120015})$, $u_{\gamma_i,\varepsilon_i}(t)$ can be written as
\begin{equation}\label{201903120016}
u_{\gamma_i,\varepsilon_i}(1)=\log\frac{\omega_{\gamma_i,\varepsilon_i}^n(1)(\varepsilon_i^2+|s|_h^2)^{1-\gamma_i}}{\omega_0^n}+F_0
+\mu_{\gamma_i}\varphi_{\gamma_i,\varepsilon_i}(1)+C_{\gamma_i,\varepsilon_i,1},
\end{equation}
where $C_{\gamma_i,\varepsilon_i,1}$ is the normalization constant such that $\frac{1}{V}\int_M e^{-u_{\gamma_i,\varepsilon_i}(1)}dV_{\gamma_i,\varepsilon_i}(1)=1$. By Lemma \ref{201901} and
Lemma \ref{201902}, $C_{\gamma_i, \varepsilon_i,1}$ and $u_{\gamma_i, \varepsilon_i}(1)$ are uniformly bounded independent of $\varepsilon_i$ and $\gamma_i$. Let $u_{\gamma_i,\varepsilon_i}(t)= \dot{\psi}_{\gamma_i,\varepsilon_i}(t)+c_{\gamma_i,\varepsilon_i}(t)$. Then
\begin{equation}\label{333}
c_{\gamma_i,\varepsilon_i}(1)=C_{\gamma_i,\varepsilon_i,1}-\mu_{\gamma_i} e^{\mu_{\gamma_i}}\tilde{C}_{\gamma_i,\varepsilon_i}.
\end{equation}

Next, we prove that there exists a uniform constant $C$ such that for any $\varepsilon_i$ and $\gamma_i$,
\begin{equation}\label{201903120017}\|\dot{\psi}_{\gamma_i,\varepsilon_i}(t)\|_{C^{0}(M)}\leqslant C\ \ \ on\ \ \ [T_{i-1}+\frac{1}{4} , T_{i-1}+1].
\end{equation}
We let
\begin{equation*}\label{3.22.32}\alpha_{\gamma_i,\varepsilon_i}(t)=\frac{1}{V}\int_{M}\dot{\psi}_{\gamma_i,\varepsilon_i}(t)dV_{\gamma_i,\varepsilon_i}(t)
=\frac{1}{V}\int_{M}u_{\gamma_i,\varepsilon_i}(t)dV_{\gamma_i,\varepsilon_i}(t)-c_{\gamma_i,\varepsilon_i}(t).
\end{equation*}
Through computing, we have
\begin{equation}\frac{d}{dt}\alpha_{\gamma_i,\varepsilon_i}(t)=\mu_{\gamma_i}\alpha_{\gamma_i,\varepsilon_i}(t)-\|\nabla u_{\gamma_i,\varepsilon_i}(t)\|^{2}_{L^{2}},
\end{equation}
Integrating this equality form $1$ to $t$ on both sides, we have
\begin{equation*}
\begin{split}
\label{3.22.33}e^{-\mu_{\gamma_i} (t-1)}\alpha_{\gamma_i,\varepsilon_i}(t)&=\alpha_{\gamma_i,\varepsilon_i}(1)-\int_{1}^{t}e^{-\mu_{\gamma_i} (s-1)}\|\nabla u_{\gamma_i,\varepsilon_i}(s)\|^{2}_{L^{2}}ds\\
&=\frac{1}{V}\int_{M}u_{\gamma_i,\varepsilon_i}(1)dV_{\gamma_i,\varepsilon_i}(1)-c_{\gamma_i,\varepsilon_i}(1)-\int_{1}^{t}e^{-\mu_{\gamma_i} (s-1)}\|\nabla u_{\gamma_i,\varepsilon_i}(s)\|^{2}_{L^{2}}ds.
\end{split}
\end{equation*}
Putting $(\ref{201903120016})$ and $(\ref{333})$ into above equality, we have
\begin{equation}e^{-\mu_{\gamma_i} (t-1)}\alpha_{\gamma_i,\varepsilon_i}(t)=\int_{t}^{T_{i-1}+1}e^{-\mu_{\gamma_i} (s-1)}\|\nabla u_{\gamma_i,\varepsilon_i}(s)\|^{2}_{L^{2}}ds.
\end{equation}
Therefore, by using $(\ref{2019030701})$, there exists uniform constant $C$ such that for any $\gamma_i$, $\varepsilon_i$ and $t\in[T_{i-1}+\frac{1}{4} ,T_{i-1}+1]$,
\begin{equation}\label{3.22.34}0\leq\alpha_{\gamma_i,\varepsilon_i}(t)=\int_{t}^{T_{i-1}+1}e^{-\mu_{\gamma_i} (s-t)}\|\nabla u_{\gamma_i,\varepsilon_i}(s)\|^{2}_{L^{2}}ds\leq C.
\end{equation}
Then we conclude that $\|\dot{\psi}_{\gamma_i,\varepsilon_i}(t)\|_{C^0(M)}$ is uniformly bounded on $[T_{i-1}+\frac{1}{4} ,T_{i-1}+1]$ by using $(\ref{190312006})$. Since $osc_{M}\psi_{\gamma_i,\varepsilon_i}(t)=osc_{M}\varphi_{\gamma_i,\varepsilon_i}(t)$, $osc_{M}\psi_{\gamma_i,\varepsilon_i}(t)$ are uniform bounded on $[T_{i-1}+\frac{1}{4} ,T_{i-1}+1]$. We have proved that there exists constant $C$ depending only on $L_\beta$, $n$, $\|\varphi_0\|_{L^\infty(M)}$, $\beta$ and $\omega_{0}$ such that 
\begin{equation}\label{1915001}
\sup\limits_{t\in[T_{i-1}+\frac{1}{4},T_{i-1}+1]}\Big(\|u_{\gamma_i,\varepsilon_i}(t)\|_{C^0(M)}+\|\dot{\psi}_{\gamma_i,\varepsilon_i}(t)\|_{C^0(M)}+osc_{M}\psi_{\gamma_i,\varepsilon_i}(t)\Big)\leqslant C.
\end{equation}
From above estimates, we also have
\begin{equation}\label{190312019}
\sup\limits_M(\mu_{\gamma_i}\psi_{\gamma_i,\varepsilon_i}(t))\geqslant C\ \ \ and\ \ \ \inf\limits_M(\mu_{\gamma_i}\psi_{\gamma_i,\varepsilon_i}(t))\leqslant C,
\end{equation}
which implies that $\|\mu_{\gamma_i}\psi_{\gamma_i,\varepsilon_i}(t)\|_{C^0(M)}$ are uniformly bounded on $[T_{i-1}+\frac{1}{4} ,T_{i-1}+1]$. From the arguments in Lemma $2.3$ of \cite{JWLXZ1}, there exist uniform constant $B$ and $C$ such that, on $(T_{i-1}+\frac{1}{4},T_{i-1}+1]\times M$, we have
\begin{equation*}
\begin{split}
tr_{\omega_{\varepsilon_i}^{\gamma_i}}\omega_{\gamma_i,\varepsilon_i}(t)&\leqslant exp\frac{\Big(C+B(\psi_{\gamma_i,\varepsilon_i}(t)-k\chi_{\gamma_i})-B\inf\limits_{[T_{i-1}+\frac{1}{4},T_{i-1}+1]\times M}(\psi_{\gamma_i,\varepsilon_i}(t)-k\chi_{\gamma_i})\Big)}{t-(T_{i-1}+\frac{1}{4})}\\
&\leqslant exp\frac{\Big(C+B\psi_{\gamma_i,\varepsilon_i}(t)-B\psi_{\gamma_i,\varepsilon_i}(t_1)+B\psi_{\gamma_i,\varepsilon_i}(t_1)-B\psi_{\gamma_i,\varepsilon_i}(t_1,z_1)\Big)}{t-(T_{i-1}+\frac{1}{4})}\\
&\leqslant exp\frac{\Big(C+\frac{3B}{4}\|\dot{\psi}_{\gamma_i,\varepsilon_i}(t)\|_{C^0(M)}+B osc_M \psi_{\gamma_i,\varepsilon_i}(t_1)\Big)}{t-(T_{i-1}+\frac{1}{4})}\\
&\leqslant exp\frac{C}{t-(T_{i-1}+\frac{1}{4})},
\end{split}
\end{equation*}
where $(t_1,z_1)$ is the minimum point of $ \psi_{\gamma_i,\varepsilon_i}(t)$ on $[T_{i-1}+\frac{1}{4},T_{i-1}+1]\times M$, and we use that $\chi_{\gamma}$ with $\gamma\in[\beta,1)$ can be uniformly bounded by a constant $C$ depending only on $\beta$ in the second inequality. So on  $t\in[T_{i-1}+\frac{1}{2},T_{i-1}+1]\times M$, we have
\begin{equation}\label{201903040003}
tr_{\omega_{\varepsilon_i}^{\gamma_i}}\omega_{\gamma_i,\varepsilon_i}(t)\leqslant C.
\end{equation}
Hence we prove the claim $(\ref{1916})$.

From $(\ref{1916})$, for $t\in[T_{i-1}+\frac{1}{2},T_{i}]$, $\gamma_i$ and $\varepsilon_i$, we have
\begin{equation*}\label{190312001}
\begin{split}
osc_M\big(\dot{\varphi}_{\gamma_i,\varepsilon_i}(t)-\mu_{\gamma_i}\varphi_{\gamma_i,\varepsilon_i}(t)\big)&\leqslant osc_M\dot{\varphi}_{\gamma_i,\varepsilon_i}(t)+\mu_{\gamma_i}osc_M\varphi_{\gamma_i,\varepsilon_i}(t)\\
&=osc_M u_{\gamma_i,\varepsilon_i}(t)+\mu_{\gamma_i}osc_M\varphi_{\gamma_i,\varepsilon_i}(t)\leqslant C.
\end{split}
\end{equation*}
Combining this with $(\ref{192202})$, we conclude that the $C^0$-norms of $\dot{\varphi}_{\gamma_i,\varepsilon_i}(t)-\mu_{\gamma_i}\varphi_{\gamma_i,\varepsilon_i}(t)$ and then $\dot{\phi}_{\gamma_i,\varepsilon_i}(t)-\mu_{\gamma_i}\phi_{\gamma_i,\varepsilon_i}(t)$ are uniformly bounded on $[T_{i-1}+\frac{1}{2},T_{i}]$. So we obtain
\begin{equation*}
\begin{split}
tr_{\omega_{\gamma_i,\varepsilon_i}(t)}\omega_{\varepsilon_i}^{\gamma_i}&\leqslant\frac{1}{(n-1)!}(tr_{\omega_{\varepsilon_i}^{\gamma_i}}\omega_{\gamma_i,\varepsilon_i}(t))^{n-1}\frac{(\omega^{\gamma_i}_{\varepsilon_i})^n}{\omega_{\gamma_i,\varepsilon_i}^n(t)}\\
&\leqslant C\exp\Big(-\dot{\phi}_{\gamma_i,\varepsilon_i}(t)+\mu_{\gamma_i}\phi_{\gamma_i,\varepsilon_i}(t)+F_{\gamma_i,\varepsilon_i}\Big)\leqslant C,
\end{split}
\end{equation*}
where we use that $F_{\gamma,\varepsilon}$ with $\gamma\in[\beta,1)$ can be uniformly bounded by a constant $C$ depending only on $\beta$, $n$ and $\omega_0$. Hence for any $t\in[T_{i-1}+\frac{1}{2},T_i]$, $\gamma_i$ and $\varepsilon_i$, we have
\begin{equation}\label{192400010101}
C^{-1}\omega^{\gamma_i}_{\varepsilon_i}\leqslant\omega_{\gamma_i,\varepsilon_i}(t)\leqslant C\omega^{\gamma_i}_{\varepsilon_i}\ \ \ on\ \ M.
\end{equation}
By Evans-Krylov-Safonov's estimates, for any $B_r(p)\subset\subset M\setminus D$, there is uniform constant $C_{k,p,r}$ such that
\begin{equation*}\label{1924}
\|\varphi_{\gamma_i,\varepsilon_i}(t)\|_{C^{k}\big(B_r(p)\big)}\leqslant C_{k,p,r}
\end{equation*}
for any $t\in[T_{i-1}+\frac{3}{4},T_i]$, $\gamma_i$ and $\varepsilon_i$. In particular, for any $K\subset\subset M\setminus D$, we have
\begin{equation}\label{1925}
 \|\varphi_{\gamma_i,\varepsilon_i}(t_i)\|_{C^k(K)}\leqslant C_{k,K}
\end{equation}
for any $i$ and $k\in\mathbb{N}^+$. Then $\varphi_{\gamma_i,\varepsilon_i}(t_i)$ (by taking a subsequence if necessary) converge to a function $\varphi_{\beta}(\infty)\in C^\infty(M\setminus D)$ in $C^\infty_{loc}$-topology in $M\setminus D$. By estimates $(\ref{192400010101})$, the Lebesgue Dominated Convergence theorem implies that $\int_M (\omega_0+\sqrt{-1}\partial\bar{\partial}\varphi_{\beta}(\infty))^n=\int_M\omega_0^n$.

Let $u_{\beta}(\infty)$ be the twisted Ricci potential of $\omega_{\varphi_{\beta}(\infty)}=\omega_0+\sqrt{-1}\partial\bar{\partial}\varphi_{\beta}(\infty)$. For any $K\subset\subset M\setminus D$, since $\omega_{\gamma_i,\varepsilon_i}(t_i)$ are uniformly equivalent to $\omega_0$ on $K$ by $(\ref{192400010101})$ and $\mu_{\gamma_i}$ tend to $0$, we conclude that $u_\beta(\infty)$ satisfies
\begin{equation}
-Ric(\omega_{\varphi_{\beta}(\infty)})+(1-\beta)[D]=\sqrt{-1}\partial\bar{\partial}u_{\beta}(\infty)\ \ \ \ on\ \ \ M\setminus D.
\end{equation}
By the Lebesgue Dominated Convergence theorem again, we get that $\frac{1}{V}\int_M e^{-u_{\beta}(\infty)}dV_\beta(\infty)=1$. Since $A_{\mu_\gamma,\varepsilon}(t)$ is increasing along $(TKRF_{\mu_\gamma,\varepsilon})$ with $\mu_\gamma>0$, for any $j<i$,
\begin{equation}\label{1926}
A_{\mu_{\gamma_i},\varepsilon_i}(t_j)\leqslant A_{\mu_{\gamma_i},\varepsilon_i}(t_i)\leqslant0,
\end{equation}
the second inequality due to Jensen's inequality. Let $i\rightarrow\infty$, we have
\begin{equation}\label{1927}
A_{\mu_\beta}(t_j)\leqslant \frac{1}{V}\int_M u_{\beta}(\infty)e^{-u_{\beta}(\infty)}dV_\beta(\infty)\leqslant0.
\end{equation}
Since $\mu_\beta=0$, from Theorem \ref{20190002}, the conical K\"ahler-Ricci flow $(CKRF_{\mu_\beta})$ converge to a conical K\"ahler-Einstein metric. By using the normalization of the twisted Ricci potential $u_{\beta}(t)$, we conclude that $u_{\beta}(t)$ converge to $0$ in $C^\infty_{loc}$-topology on $M\setminus D$, and hence $A_{\mu_\beta}(t)$ converge to $0$. This implies that
\begin{equation}\label{1928}
\frac{1}{V}\int_M u_{\beta}(\infty)e^{-u_{\beta}(\infty)}dV_\beta(\infty)=0
\end{equation}
after letting $j\rightarrow\infty$. Since $\log$ is strictly concave, we conclude $u_{\beta}(\infty)=0$ by using Jensen's inequality and its normalization. This means that $\omega_{\varphi_\beta(\infty)}$ is a Ricci flat conical K\"ahler-Einstein metric with cone angle $2\pi\beta$ along $D$, hence $\varphi_\beta(\infty)$ satisfies equation $(\ref{201902230101})$. Then Dinew's uniqueness theorem (Theorem $1.2$ in \cite{SDINEW}) implies that $\varphi_\beta(\infty)=\varphi_\beta+C$. Letting $i\rightarrow\infty$ in $(\ref{2000019})$, we get
\begin{equation}\label{1929}
K_\beta+M_\beta\geqslant osc_{M}\varphi_{\beta}+\sup\limits_{M\setminus D} tr_{\omega_\beta}\omega_{\varphi_\beta}\geqslant L_\beta+\frac{7}{8}\geqslant K_\beta+M_\beta+\frac{7}{8}.
\end{equation}
This leads to a contradiction. Thus Lemma \ref{1901} is proved.

\medskip

We denote constant
\begin{equation*}
\tilde{C}_{\gamma,\varepsilon}=e^{-\mu_\gamma}\frac{1}{\mu_\gamma}\Big(\int_{1}^{+\infty}e^{-\mu_\gamma (t-1)}\|\nabla u_{\gamma,\varepsilon}(t)\|^{2}_{L^{2}}dt-\frac{1}{V}\int_{M}\big(F_{\gamma,\varepsilon}(1)+\mu_\gamma\varphi_{\gamma,\varepsilon}(1)\big)dV_{\gamma,\varepsilon}(1)\Big),
\end{equation*}
where $F_{\gamma,\varepsilon}(1)=F_{0}+\log\Big(\frac{\omega_{\gamma,\varepsilon}^n(1)(\varepsilon^{2}+|s|_{h}^{2})^{1-\gamma}}{\omega_{0}^{n}}\Big)$. Since $\mu_\gamma>0$, $\tilde{C}_{\gamma,\varepsilon}$ are well-defined and uniformly bounded by the uniform Perelman's estimates (independent of $\varepsilon$ but depend on $\mu_\gamma$). We denote $\psi_{\gamma,\varepsilon}(t)=\varphi_{\gamma,\varepsilon}(t)+\tilde{C}_{\gamma,\varepsilon}e^{\gamma t}$, which is a solution to the equation
\begin{equation}\label{TKRF6}
\left \{\begin{split}
&\  \frac{\partial \psi_{\gamma,\varepsilon}(t)}{\partial t}=\log\frac{(\omega_{0}+\sqrt{-1}\partial\bar{\partial}\psi_{\gamma,\varepsilon}(t))^{n}}{\omega_{0}^{n}}+\mu_\gamma\psi_{\gamma,\varepsilon}(t)\\
&\ \ \ \ \ \ \ \ \ \ \ \ \ \ \ \ \ \ \ \ \ +F_{0}+\log(\varepsilon^{2}+|s|_{h}^{2})^{1-\gamma}.  \\
&\  \psi_{\gamma,\varepsilon}(0)=\varphi_{0}+\tilde{C}_{\gamma,\varepsilon}
\end{split}
\right.
\end{equation}
From the Proposition $4.5$ in \cite{JWLXZ1}, we know that $\|\dot{\psi}_{\gamma,\varepsilon}(t)\|_{C^0(M)}$ are uniform bounded for any $\varepsilon>0$ and $t\geqslant1$. For any $\gamma\in(\beta,\beta+\delta(\lambda))$ obtained in Lemma \ref{1901}, we can get the uniform $C^0$-estimates and Laplacian $C^2$-estimates of $\psi_{\gamma,\varepsilon}(t)$ for any $\varepsilon\in(0,\delta(\lambda))$ and $t\geqslant1$. Hence the twisted Mabuchi energy $\mathcal{M}_{\mu_\gamma,\varepsilon}$ are uniformly bounded from below along the twisted K\"ahler-Ricci flow $(TKRF_{\mu_\gamma,\varepsilon})$, that is, there exists uniform constant $C$ such that for $\varepsilon\in(0,\delta(\lambda))$ and $t\geqslant 1$,
\begin{equation*}
\begin{split}
\mathcal{M}_{\mu_{\gamma},\varepsilon}(\psi_{\gamma,\varepsilon}(t))&=-\mu_{\gamma}\Big(I_{\omega_{0}}(\psi_{\gamma,\varepsilon}(t))-J_{\omega_{0}}(\psi_{\gamma,\varepsilon}(t))\Big)+\frac{1}{V}\int_{M}\log\frac{\omega^n_{\gamma,\varepsilon}(t)}{\omega_{0}^{n}}dV_{\gamma,\varepsilon}(t)\\
&\ \ \ \ \ -\frac{1}{V}\int_{M}\Big(F_0+(1-\gamma)\log(\varepsilon^2+|s|^2_h)\Big)(dV_{0}-dV_{\gamma,\varepsilon}(t))\\
&=-\mu_{\gamma}\Big(I_{\omega_{0}}(\psi_{\gamma,\varepsilon}(t))-J_{\omega_{0}}(\psi_{\gamma,\varepsilon}(t))\Big)+\frac{1}{V}\int_{M}\Big(\dot{\psi}_{\gamma,\varepsilon}(t)-\mu_{\gamma}\psi_{\gamma,\varepsilon}(t)\Big)dV_{\gamma,\varepsilon}(t)\\
&\ \ \ \ \ -\frac{1}{V}\int_{M}\Big(F_0+(1-\gamma)\log(\varepsilon^2+|s|^2_h)\Big)dV_{0}\\
&\geqslant -C.
\end{split}
\end{equation*}
Then by using the arguments in \cite{JWLXZ, JWLXZ1}, we deduce the convergence of the conical K\"ahler-Ricci flows $(CKRF_{\mu_\gamma})$ with $\gamma\in(\beta,\beta+\delta(\lambda))$.

\medskip

{\it Proof of Theorem \ref{20190223001}.}\ \ Combining Lemma \ref{1901}, Lemma \ref{1902} with Lemma \ref{1.8.601}, we obtain that $|R(\omega_{\gamma,\varepsilon}(t))-(1-\gamma)tr_{\omega_{\gamma,\varepsilon}(t)}\theta_{\varepsilon}|$ and $\|u_{\gamma,\varepsilon}(t)\|_{C^{1}(\omega_{\gamma,\varepsilon}(t))}$ are uniform bounded for any $\varepsilon\in(0,\delta(\lambda))$, $\gamma\in(\beta,\beta+\delta(\lambda))$ and $t\in[\frac{1}{\delta(\lambda)},+\infty)$. From Lemma \ref{1901}, we also know that there exists uniform constant $C$ such that
\begin{equation}
\label{19220201001}
osc_{M}(\dot{\varphi}_{\gamma,\varepsilon}(t)-\mu_\gamma\varphi_{\gamma,\varepsilon}(t))\leqslant C\ \ and \ \ \sup\limits_Mtr_{\omega_\varepsilon^\gamma}\omega_{\gamma,\varepsilon}(t)\leqslant C
\end{equation}
hold for any $\varepsilon\in(0,\delta(\lambda))$, $\gamma\in(\beta,\beta+\delta(\lambda))$ and $t\in[\frac{1}{\delta(\lambda)},+\infty)$. Combining this with $(\ref{192202})$, we conclude that $\|\dot{\varphi}_{\gamma,\varepsilon}(t)-\mu_\gamma\varphi_{\gamma,\varepsilon}(t)\|_{C^0(M)}$ and then $\|\dot{\phi}_{\gamma,\varepsilon}(t)-\mu_\gamma\phi_{\gamma,\varepsilon}(t)\|_{C^0(M)}$ are uniform bounded. So
\begin{equation*}
\begin{split}
tr_{\omega_{\gamma,\varepsilon}(t)}\omega_{\varepsilon}^{\gamma}&\leqslant\frac{1}{(n-1)!}(tr_{\omega_{\varepsilon}^{\gamma}}\omega_{\gamma,\varepsilon}(t))^{n-1}\frac{(\omega^{\gamma}_{\varepsilon})^n}{\omega_{\gamma,\varepsilon}^n(t)}\\
&\leqslant C\exp\Big(-\dot{\phi}_{\gamma,\varepsilon}(t)+\mu_{\gamma}\phi_{\gamma,\varepsilon}(t)+F_{\gamma,\varepsilon}\Big)\leqslant C.
\end{split}
\end{equation*}
Hence for any $\varepsilon\in(0,\delta(\lambda))$, $\gamma\in(\beta,\beta+\delta(\lambda))$ and $t\in[\frac{1}{\delta(\lambda)},+\infty)$, we have
\begin{equation}\label{1924000101111}
C^{-1}\omega^{\gamma}_{\varepsilon}\leqslant\omega_{\gamma,\varepsilon}(t)\leqslant C\omega^{\gamma}_{\varepsilon}\leqslant C\omega_\beta\ \ \ on\ \ M.
\end{equation}
Let $(M, d_\beta)$ be the metric completion of $(M,\omega_\beta)$, which is of finite diameter. Therefore,
\begin{equation}diam(M, \omega_{\gamma,\varepsilon}(t)) \leqslant C diam(M, d_\beta)
\end{equation}
are uniformly bounded from above (see also Proposition $2.4$ in \cite{CDS1}). So we prove Theorem \ref{20190223001} for $\gamma\in(\beta,\beta+\delta(\lambda))$. Fix $\gamma_0\in(\beta,\beta+\delta(\lambda))$, we can get Theorem \ref{20190223001} for any $\varepsilon\in(0,1)$, $\gamma\in(\gamma_0,1)$ and $t\in[1,+\infty)$ by following the arguments in \cite{JWLXZ}. Combining these two parts, we complete the proof of Theorem \ref{20190223001}.

\section{Proof of the case $\mu_\beta>0$}\label{sec:4}

In this section, we always assume that $\mu_\beta$ is positive and there is a conical K\"ahler-Einstein metric $\omega_{\varphi_\beta}$ $(0<\beta<1)$ with cone angle $2\pi\beta$ along $D$. We first deduce some uniform regularities along the twisted K\"ahler-Ricci flows $(TKRF^\beta_{\mu_\gamma,\varepsilon})$, and then prove Lemma \ref{0080} and Lemma \ref{009}. For the sake of brevity, we only consider the case $\lambda=1$, that is, $\mu_{\gamma}=\gamma$. Our arguments are also valid for any $\lambda>0$, only if the coefficient $\gamma$ before $\omega_{\gamma}(t)$ in the case of $\lambda=1$ is replaced by $\mu_{\gamma}=1-(1-\gamma)\lambda$.

Denote $\omega_{\varphi_\beta}=\omega_0+\sqrt{-1}\partial\bar{\partial}\varphi_\beta$ with $\varphi_\beta\in L^\infty(M)\bigcap PSH(M,\omega_0)$. Assume that $\sup\limits_M\varphi_\beta\leqslant-1$ and satisfies
\begin{equation}\label{2019022301}
(\omega_0+\sqrt{-1}\partial\bar{\partial}\varphi_\beta)^n=e^{-\beta\varphi_\beta-F_0+\xi_\beta}\frac{\omega_0^n}{|s|_h^{2(1-\beta)}},
\end{equation}
where $\xi_\beta$ is the constant such that $\frac{1}{V}\int_M e^{-\beta\varphi_\beta-F_0+\xi_\beta}\frac{\omega_0^n}{|s|_h^{2(1-\beta)}}=1$. From Ko{\l}odziej's $L^p$-estimates \cite{K000,K}, $\varphi_\beta$ is H\"older continuous. By Demailly's regularization result \cite{JPD}, we approximate $\varphi_\beta$ with a decreasing sequence of smooth $\omega_0$-psh functions $\varphi_\varepsilon$. Then Dini's theorem implies that $\varphi_\varepsilon$ converge to $\varphi_\beta$ in $L^\infty$-sense on $M$. Without loss of generality, we assume that
\begin{equation}\label{04}
-C\leqslant\varphi_\varepsilon\leqslant 0\ \ \ on\ \ \ M
\end{equation}
for $\varepsilon\in(0,1]$. In the case $\mu_\beta>0$, we normalize $\varphi_0$ by
\begin{equation}\label{nom}
\inf\limits_M\varphi_0\geqslant \sup\limits_M\varphi_{1}+1,
\end{equation}
where $\varphi_{1}$ comes from the approaching sequence $\varphi_\varepsilon$ with $\varepsilon=1$. Since the twisted K\"ahler-Ricci flow preserves the K\"ahler class, we write $(TKRF^\beta_{\gamma,\varepsilon})$ as the parabolic Monge-Amp\`ere equation on potentials,
\begin{equation}\label{TKRF2}
\left \{\begin{split}
&\  \frac{\partial \varphi^\beta_{\gamma,\varepsilon}(t)}{\partial t}=\log\frac{(\omega_{0}+\sqrt{-1}\partial\bar{\partial}\varphi^\beta_{\gamma,\varepsilon}(t))^{n}}{\omega_{0}^{n}}+\gamma\varphi^\beta_{\gamma,\varepsilon}(t)+(\beta-\gamma)\varphi_\varepsilon\\
&\ \ \ \ \ \ \ \ \ \ \ \ \ \ \ \ \ \ \ \ +F_{0}+\log(\varepsilon^{2}+|s|_{h}^{2})^{1-\beta}\\
&\  \varphi^\beta_{\gamma,\varepsilon}(0)=\varphi_{0}.
\end{split}
\right.
\end{equation}
We rewrite equation (\ref{TKRF2}) as
\begin{equation}\label{TKRF4}
\left \{\begin{split}
&\  \frac{\partial \phi^\beta_{\gamma,\varepsilon}(t)}{\partial t}=\log\frac{(\omega^\beta_{\varepsilon}+\sqrt{-1}\partial\bar{\partial}\phi^\beta_{\gamma,\varepsilon}(t))^{n}}{(\omega^\beta_{\varepsilon})^{n}}+F^\beta_{\gamma,\varepsilon}+\gamma\Big(\phi^\beta_{\gamma,\varepsilon}(t)+k\chi_\beta\Big),\\
&\  \\
&\  \phi^\beta_{\gamma,\varepsilon}(0)=\varphi_{0}-k\chi_\beta:=\phi^\beta_{\varepsilon}
\end{split}
\right.
\end{equation}
where $F^\beta_{\gamma,\varepsilon}=F_{0}+\log(\frac{(\omega^\beta_{\varepsilon})^{n}}{\omega_{0}^{n}}\cdot(\varepsilon^{2}+|s|_{h}^{2})^{1-\beta})+(\beta-\gamma)\varphi_\varepsilon$ and $\phi^\beta_{\gamma,\varepsilon}(t)=\varphi^\beta_{\gamma,\varepsilon}(t)-k\chi_\beta$. By the same arguments as in Lemma \ref{1902} and \ref{201901}, we have
\begin{pro}\label{2019022401}
$R(\omega^\beta_{\gamma,\varepsilon}(t))-(\beta-\gamma)tr_{\omega^\beta_{\gamma,\varepsilon}(t)}\omega_{\varphi_\varepsilon}-(1-\beta)tr_{\omega^\beta_{\gamma,\varepsilon}(t)}\theta_{\varepsilon}$ are uniformly bounded from below by $-4n$ along the twisted K\"ahler-Ricci flows $(TKRF^\beta_{\mu_\gamma,\varepsilon})$, that is, for any $\gamma\in[0,1]$, $\beta\in[0,1]$, $\varepsilon>0$ and $t\geqslant 1$,
\begin{equation}\label{3.22.9}R(\omega^\beta_{\gamma,\varepsilon}(t))-(\beta-\gamma)tr_{\omega^\beta_{\gamma,\varepsilon}(t)}\omega_{\varphi_\varepsilon}-(1-\beta)tr_{\omega^\beta_{\gamma,\varepsilon}(t)}\theta_{\varepsilon}\geqslant -4n.
\end{equation}
\end{pro}
\begin{lem}\label{200} For any $T>0$, there exists uniform constant $C$, such that for any $t\in[0,T]$, $\varepsilon>0$ and $\gamma\in[0,1]$,
\begin{equation*}\label{20}
\|\phi^\beta_{\gamma,\varepsilon}(t)\|_{L^\infty(M)}\leqslant C,
\end{equation*}
where constant $C$ depends only on $\|\varphi_0\|_{L^\infty(M)}$, $\beta$, $n$, $\omega_{0}$ and $T$.
\end{lem}
By using the arguments in section $2$ of \cite{JWLXZ1}, we get the following uniform Laplacian $C^2$-estimates on $M$ along the twisted K\"ahler-Ricci flow $(TKRF^\beta_{\gamma,\varepsilon})$.
\begin{lem}\label{21} For any $T>0$, there exists constant $C$ depending only on $\|\varphi_0\|_{L^\infty(M)}$, $n$, $\beta$, $\omega_{0}$ and $T$, such that for any $t\in(0,T]$, $\varepsilon>0$ and $\gamma\in[0,1]$,
\begin{equation*}\label{21}
\frac{t^n}{C}\leqslant\frac{(\omega^\beta_{\varepsilon}+\sqrt{-1}\partial\bar{\partial}\phi^\beta_{\gamma,\varepsilon}(t))^{n}}{(\omega^\beta_{\varepsilon})^n}\leqslant e^{\frac{C}{t}}.
\end{equation*}
\end{lem}
\begin{lem}\label{22} For any $T>0$, there exists constant $C$ depending only on $\|\varphi_0\|_{L^\infty(M)}$, $n$, $\beta$, $\omega_{0}$ and $T$, such that for any $t\in(0,T]$, $\varepsilon>0$ and $\gamma\in[0,1]$,
\begin{equation*}\label{23}
e^{-\frac{C}{t}}\omega^\beta_{\varepsilon}\leqslant\omega^\beta_{\gamma,\varepsilon}(t)\leqslant e^{\frac{C}{t}}\omega^\beta_{\varepsilon}.
\end{equation*}
\end{lem}
\begin{proof} Since there exists term $(\beta-\gamma)\varphi_\varepsilon$ in equation $(\ref{TKRF4})$, from the proofs of  Proposition $3.1$ in \cite{JWLXZ} and Lemma $2.3$ in \cite{JWLXZ1}, we need only deal with the term
\begin{equation}\label{24}
t(\beta-\gamma)\frac{\Delta_{\omega^\beta_{\varepsilon}}\varphi_\varepsilon}{tr_{\omega^\beta_{\varepsilon}}\omega^\beta_{\gamma,\varepsilon}(t)}.
\end{equation}
From \cite{CGP}, there exists a uniform constant $A$ depending on $\beta$ and $\omega_0$ such that
\begin{equation}\label{25}
0\leqslant\omega_{\varphi_\varepsilon}=\omega_0+\sqrt{-1}\partial\bar{\partial}\varphi_\varepsilon\leqslant A \omega^\beta_{\varepsilon}+\sqrt{-1}\partial\bar{\partial}\varphi_\varepsilon.
\end{equation}
When $\gamma>\beta$, by $(\ref{25})$ and $n\leqslant tr_{\omega^\beta_{\gamma,\varepsilon}(t)}\omega_\varepsilon^\beta\cdot tr_{\omega_\varepsilon^\beta}\omega^\beta_{\gamma,\varepsilon}(t)$, we have
\begin{equation}\label{250001}
t(\beta-\gamma)\frac{\Delta_{\omega^\beta_{\varepsilon}}\varphi_\varepsilon}{tr_{\omega^\beta_{\varepsilon}}\omega^\beta_{\gamma,\varepsilon}(t)}\leqslant\frac{AnT(\gamma-\beta)}{tr_{\omega^\beta_{\varepsilon}}\omega^\beta_{\gamma,\varepsilon}(t)}\leqslant AT(\gamma-\beta)tr_{\omega^\beta_{\gamma,\varepsilon}(t)}\omega_\varepsilon^\beta.
\end{equation}
When $\gamma\leqslant\beta$, we have
\begin{equation}\label{26}
An+\Delta_{\omega^\beta_{\varepsilon}}\varphi_\varepsilon\leqslant tr_{\omega^\beta_{\gamma,\varepsilon}(t)} (A \omega^\beta_{\varepsilon}+\sqrt{-1}\partial\bar{\partial}\varphi_\varepsilon)\cdot tr_{\omega^\beta_\varepsilon}\omega^\beta_{\gamma,\varepsilon}(t).
\end{equation}
Hence we have
\begin{equation}\label{27}
t(\beta-\gamma)\frac{\Delta_{\omega^\beta_{\varepsilon}}\varphi_\varepsilon}{tr_{\omega^\beta_{\varepsilon}}\omega^\beta_{\gamma,\varepsilon}(t)} \leqslant AT(\beta-\gamma)tr_{\omega^\beta_{\gamma,\varepsilon}(t)}\omega^\beta_\varepsilon+t(\beta-\gamma) \Delta_{\omega^\beta_{\gamma,\varepsilon}(t)}\varphi_\varepsilon.
\end{equation}
Let
\begin{equation}\label{28}
\Psi^\rho_{\varepsilon}=B\frac{1}{\rho}\int_{0}^{|s|_h^2}\frac{(\varepsilon^{2}+r)^{\rho}-
\varepsilon^{2\rho}}{r}dr
\end{equation}
be the uniformly bounded function introduced by Guenancia-P$\breve{a}$un in \cite{GP1}. Denote functions $H_1=t\log tr_{\omega^\beta_{\varepsilon}}\omega^\beta_{\gamma,\varepsilon}(t)+t\Psi^\rho_{\varepsilon}$ when $\gamma>\beta$ and $H_2=t\log tr_{\omega^\beta_{\varepsilon}}\omega^\beta_{\gamma,\varepsilon}(t)+t\Psi^\rho_{\varepsilon}+t(\beta-\gamma)\varphi_\varepsilon$ when $\gamma\leqslant\beta$. By choosing suitable $B$ and $\rho$, and following the arguments in \cite{JWLXZ1}, we have
\begin{equation}\label{29}
(\frac{\partial}{\partial t}-\Delta_{\omega^\beta_{\gamma,\varepsilon}(t)})H_i\leqslant\log tr_{\omega^\beta_{\varepsilon}}\omega^\beta_{\gamma,\varepsilon}(t)+Ctr_{\omega^\beta_{\gamma,\varepsilon}(t)}\omega^\beta_{\varepsilon}+C,\ \ \ \ \ i=1,\ 2,
\end{equation}
where constant $C$ depends only on $\|\varphi_0\|_{L^\infty(M)}$, $n$, $\beta$, $\omega_{0}$ and $T$. Then by the arguments as that in the proof of Lemma $2.3$ in \cite{JWLXZ1}, we get the uniform Laplacian $C^2$-estimates for $\omega^\beta_{\gamma,\varepsilon}(t)$.
\end{proof}

Hence away from time $0$, on any compact subset in $M\setminus D$, $\omega^\beta_{\gamma,\varepsilon}$ are uniformly equivalent to $\omega_0$. Then Evans-Krylov-Safonov's estimates (see \cite{NVK}) implies the following proposition.
\begin{pro}\label{30} For any $0<\delta<T<\infty$, $k\in\mathbb{N}^{+}$ and $B_r(p)\subset\subset M\setminus D$, there exists constant $C_{\beta,\delta,T,k,p,r}$, such that for any $\varepsilon>0$ and $\gamma\in[0,1]$,
\begin{equation*}\label{31}
\begin{split}
\|\varphi^\beta_{\gamma,\varepsilon}(t)\|_{C^{k}\big([\delta,T]\times B_r(p)\big)}\leqslant C_{\beta,\delta,T,k,p,r},\end{split}
\end{equation*}
where constant $C_{\beta,\delta,T,k,p,r}$ depends on $\|\varphi_0\|_{L^\infty(M)}$, $n$, $\beta$, $\delta$, $k$, $T$, $\omega_0$ and $dist_{\omega_{0}}(B_r(p),D)$.
\end{pro}
Straightforward calculation shows that the twisted Ricci potential $u^\beta_{\gamma,\varepsilon}(t)$ with respect to $\omega^\beta_{\gamma,\varepsilon}(t)$ at $t=\frac{1}{2}$ can be written as
\begin{equation*}
u^\beta_{\gamma,\varepsilon}(\frac{1}{2})=\log\frac{(\omega^\beta_{\gamma,\varepsilon}(\frac{1}{2}))^n(\varepsilon^2+|s|_h^2)^{1-\beta}}{\omega_0^n}+F_0
+\gamma\varphi^\beta_{\gamma,\varepsilon}(\frac{1}{2})+(\beta-\gamma)\varphi_\varepsilon+C^\beta_{\gamma,\varepsilon}(\frac{1}{2}),
\end{equation*}
where $C^\beta_{\gamma,\varepsilon}(\frac{1}{2})$ is a normalization constant such that $\frac{1}{V}\int_M e^{-u^\beta_{\gamma,\varepsilon}(\frac{1}{2})}dV^\beta_{\gamma,\varepsilon}(\frac{1}{2})=1$. By $(\ref{04})$, Lemma \ref{200} and Lemma \ref{21}, $C^\beta_{\gamma,\varepsilon}(\frac{1}{2})$ and $u^\beta_{\gamma,\varepsilon}(\frac{1}{2})$ are uniformly bounded. Hence $A^\beta_{\gamma,\varepsilon}(\frac{1}{2})$ are uniformly bounded from below for any $\gamma\in(0,1)$ and $\varepsilon\in(0,1)$. Since $A^\beta_{\gamma,\varepsilon}(t)$ is increasing along $(TKRF^\beta_{\gamma,\varepsilon})$ when $\gamma\in(0,\beta]$, combining this with Jensen's inequality, we have
\begin{lem}\label{1.8.4} There exists uniform constant $C$, such that
\begin{equation}\label{3.22.8}|A^\beta_{\gamma,\varepsilon}(t)|\leqslant C\end{equation}
for any $t\geqslant \frac{1}{2}$, $\gamma\in(0,\beta]$ and $\varepsilon\in(0,1]$.
\end{lem}
We consider the twisted K\"ahler-Ricci flow $(TKRF^\beta_{\gamma,\varepsilon})$ starting at $t=\frac{1}{2}$. Following the arguments in section $4$ of \cite{JWLXZ}, we have the following uniform Perelman's estimates.
\begin{thm}\label{20190314001} Let $\omega^\beta_{\gamma,\varepsilon}(t)$ be a solution of the twisted K\"ahler Ricci flow $(TKRF^\beta_{\gamma,\varepsilon})$. Then for any $0<\beta_1<\beta_2\leqslant\beta$, there exists a uniform constant $C$, such that
\begin{equation}\begin{split}\label{p0}
&\ |R(\omega^\beta_{\gamma,\varepsilon}(t))-(\beta-\gamma)tr_{\omega^\beta_{\gamma,\varepsilon}(t)}\omega_{\varphi_\varepsilon}-(1-\beta)tr_{\omega^\beta_{\gamma,\varepsilon}(t)}\theta_{\varepsilon}|\leq C,\\
&\ \|u^\beta_{\gamma,\varepsilon}(t)\|_{C^{1}(\omega^\beta_{\gamma,\varepsilon}(t))}\leq C,\\
&\ diam(M,\omega^\beta_{\gamma,\varepsilon}(t))\leq C
\end{split}
\end{equation}
hold for any $\gamma\in[\beta_1,\beta_2]$, $t\geq1$ and $\varepsilon>0$.\end{thm}
By using these uniform Perelman's estimates, we can get the following estimates and convergence for $A^\beta_{\gamma}(t)$.
\begin{thm}\label{46}
For any $\gamma\in(0,\beta]$, there exists uniform constant $C$, such that
\begin{equation}
0\leqslant-A^\beta_{\gamma}(t)\leqslant C\|\nabla u^\beta_{\gamma}(t)\|^{\frac{1}{n+1}}_{L^2(M)}\|\nabla u^\beta_{\gamma}(t)\|_{C^{0}(M)}^{\frac{n}{n+1}}.
\end{equation}
Furthermore, if $\mathcal{M}^\beta_{\gamma}$ is bounded form below, then $A^\beta_{\gamma}(t)$ converge to $0$ as $t\rightarrow+\infty$.
\end{thm}
We denote constant
\begin{equation*}
\tilde{C}^\beta_{\gamma,\varepsilon}=e^{-\gamma}\frac{1}{\gamma}\Big(\int_{1}^{+\infty}e^{-\gamma (t-1)}\|\nabla u^\beta_{\gamma,\varepsilon}(t)\|^{2}_{L^{2}}dt-\frac{1}{V}\int_{M}\big(F^\beta_{\gamma,\varepsilon}(1)+\gamma\varphi^\beta_{\gamma,\varepsilon}(1)\big)dV^\beta_{\gamma,\varepsilon}(1)\Big),
\end{equation*}
where $F^\beta_{\gamma,\varepsilon}(1)=F_{0}+(\beta-\gamma)\varphi_\varepsilon+\log\Big(\frac{(\omega^\beta_{\gamma,\varepsilon}(1))^{n}(\varepsilon^{2}+|s|_{h}^{2})^{1-\beta}}{\omega_{0}^{n}}\Big)$. From previous discussions, $\tilde{C}^\beta_{\gamma,\varepsilon}$ is well-defined and uniformly bounded. Next, we consider the solution $\psi^\beta_{\gamma,\varepsilon}(t)=\varphi^\beta_{\gamma,\varepsilon}(t)+\tilde{C}^\beta_{\gamma,\varepsilon}e^{\gamma t}$ to the equation
\begin{equation}\label{TKRF5}
\left \{\begin{split}
&\  \frac{\partial \psi^\beta_{\gamma,\varepsilon}(t)}{\partial t}=\log\frac{(\omega_{0}+\sqrt{-1}\partial\bar{\partial}\psi^\beta_{\gamma,\varepsilon}(t))^{n}}{\omega_{0}^{n}}+\gamma\psi^\beta_{\gamma,\varepsilon}(t)+(\beta-\gamma)\varphi_\varepsilon\\
&\ \ \ \ \ \ \ \ \ \ \ \ \ \ \ \ \ \ \ \ +F_{0}+\log(\varepsilon^{2}+|s|_{h}^{2})^{1-\beta}.\\
&\  \psi^\beta_{\gamma,\varepsilon}(0)=\varphi_{0}+\tilde{C}^\beta_{\gamma,\varepsilon}
\end{split}
\right.
\end{equation}
By using the same arguments as in the proof of Proposition $4.5$ in \cite{JWLXZ1}, we have the following two propositions.
\begin{pro}\label{1.8.5.3} For any $\beta_0\in(0,\beta)$, there exists uniform constant $\hat{C}_{\beta_0}$ such that
\begin{equation}\|\dot{\psi}^\beta_{\gamma, \varepsilon}(t)\|_{C^{0}(M)}\leqslant \hat{C}_{\beta_0}
\end{equation}
for any $\varepsilon>0$, $\gamma\in[\beta_0,\beta]$ and $t\geqslant 1$.
\end{pro}
\begin{pro}\label{1.8.5.3.11} For any $\beta_0\in(0,1)$, there exists uniform constant $C$ such that
$$\|\dot{\psi}_{\gamma, \varepsilon}(t)\|_{C^{0}(M)}\leqslant C$$
for any $\varepsilon>0$, $\gamma\in[\beta_0,1)$ and $t\geqslant 1$, where $\psi_{\gamma, \varepsilon}(t)$ is the solution of equation $(\ref{TKRF6})$.
\end{pro}
Next, we prove the continuity of $\varphi^\beta_{\gamma,\varepsilon}(t)$ with respect to variables $\gamma$ and $\varepsilon$.
\begin{pro}\label{36}
There exists a constant $\hat{\delta}$, such that for any $t\in(0,\hat{\delta})$, $\varepsilon\in[0,1)$ and $z\in M$, $\varphi^\beta_{\gamma,\varepsilon}(t)$ is increasing with respect to $\gamma\in[0,1]$.
\end{pro}
\begin{proof} We first consider the case $\varepsilon\in(0,1)$. By Ko{\l}odziej's results \cite{K000, K}, there exists a H\"older continuous solution $u^\beta_{\gamma,\varepsilon}$ to the equation
\begin{equation}(\omega_{0}+\sqrt{-1}\partial\bar{\partial}u^\beta_{\gamma,\varepsilon})^{n}=e^{-F_{0}-\gamma\varphi_{0}-(\beta-\gamma)\varphi_\varepsilon+\hat{C}}\frac{\omega_{0}^{n}}
{(\varepsilon^2+|s|_{h}^{2})^{(1-\beta)}},
\end{equation}
and $u^\beta_{\gamma,\varepsilon}$ satisfies
\begin{equation}\|u^\beta_{\gamma,\varepsilon}\|_{L^\infty(M)}\leqslant C,
\end{equation}
where the normalization constant $\hat{C}$ is uniformly bounded independent of $\varepsilon$ and $\gamma$, constant $C$ depends only on $\|\varphi_0\|_{L^\infty(M)}$, $\beta$ and $F_0$. We define function
\begin{equation}\psi^\beta_{\gamma,\varepsilon}(t)=(1-te^{\gamma t})\varphi_{0}+te^{\gamma t}u^\beta_{\gamma,\varepsilon}+h(t)e^{\gamma t},
\end{equation}
where
\begin{equation}\begin{split}
h(t)&=-t\|\varphi_{0}\|_{L^\infty(M)}-t\|u^\beta_{\gamma,\varepsilon}\|_{L^\infty(M)}+n(t\log t-t)e^{-\gamma t}\\
&\ +\gamma n\int_0^te^{-\gamma s}s\log s ds+\hat{C}\int_0^t e^{-\gamma s} ds.
\end{split}
\end{equation}
From similar arguments in the proof of Proposition $3.5$ in \cite{JWLXZ1}, we know that $\psi^\beta_{\gamma,\varepsilon}(t)$ is a subsolution of equation $(\ref{TKRF2})$ and hence
\begin{equation}\label{37}
\varphi^\beta_{\gamma,\varepsilon}(t)\geqslant \psi^\beta_{\gamma,\varepsilon}(t)\geqslant \varphi_0(z)+h_1(t)
\end{equation}
for any $t\in(0,1)$, where $h_1(t)$ is a continuous function independent of $\gamma$ and $\varepsilon$ such that $h(0)=0$. Hence there exists $\hat{\delta}$, such that $h_1(t)>-\frac{1}{2}$ for any $t\in[0,\hat{\delta}]$. By the normalization $(\ref{nom})$, we have
\begin{equation}\label{38}
\varphi^\beta_{\gamma,\varepsilon}(t)\geqslant \sup\limits_M\varphi_{1}+1-\frac{1}{2}>\varphi_{1}
\end{equation}
for any $\gamma\in[0,1]$, $\varepsilon\in(0,1)$, $t\in(0,\hat{\delta})$ and $z\in M$. For $\gamma_1\leqslant\gamma_2$, on $[0,\hat{\delta}]\times M$,
\begin{equation*}\label{39}
\begin{split}
\frac{\partial}{\partial t}(\varphi^\beta_{\gamma_1,\varepsilon}(t)-\varphi^\beta_{\gamma_2,\varepsilon}(t))&=\log\frac{(\omega_0+\sqrt{-1}\partial\bar{\partial}\varphi^\beta_{\gamma_1,\varepsilon}(t))^n}{(\omega_0+\sqrt{-1}\partial\bar{\partial}\varphi^\beta_{\gamma_2,\varepsilon}(t))^n}+\gamma_1\Big(\varphi^\beta_{\gamma_1,\varepsilon}(t)-\varphi^\beta_{\gamma_2,\varepsilon}(t)\Big)\\
&\ \ \ \ -(\gamma_2-\gamma_1)\varphi^\beta_{\gamma_2,\varepsilon}(t)+(\gamma_2-\gamma_1)\varphi_\varepsilon\\
&\leqslant\log\frac{(\omega_0+\sqrt{-1}\partial\bar{\partial}\varphi^\beta_{\gamma_1,\varepsilon}(t))^n}{(\omega_0+\sqrt{-1}\partial\bar{\partial}\varphi^\beta_{\gamma_2,\varepsilon}(t))^n}+\gamma_1\Big(\varphi^\beta_{\gamma_1,\varepsilon}(t)-\varphi^\beta_{\gamma_2,\varepsilon}(t)\Big)\\
&\ \ \ \ -(\gamma_2-\gamma_1)\varphi^\beta_{\gamma_2,\varepsilon}(t)+(\gamma_2-\gamma_1)\varphi_{1}\\
&\leqslant\log\frac{(\omega_0+\sqrt{-1}\partial\bar{\partial}\varphi^\beta_{\gamma_1,\varepsilon}(t))^n}{(\omega_0+\sqrt{-1}\partial\bar{\partial}\varphi^\beta_{\gamma_2,\varepsilon}(t))^n}+\gamma_1\Big(\varphi^\beta_{\gamma_1,\varepsilon}(t)-\varphi^\beta_{\gamma_2,\varepsilon}(t)\Big).
\end{split}
\end{equation*}
Then by the maximum principle, we have $\varphi^\beta_{\gamma_1,\varepsilon}(t)\leqslant \varphi^\beta_{\gamma_2,\varepsilon}(t)$ for any $\varepsilon\in(0,1)$, $t\in(0,\hat{\delta})$ and $z\in M$. Then letting $\varepsilon\rightarrow0$, we get the $\varepsilon=0$ case.
\end{proof}

Since $\varphi_\varepsilon$ decrease to $\varphi_\beta$ as $\varepsilon\searrow0$, then for fix $\gamma\in[0,\beta]$, $t\in[0,\infty)$ and $z\in M$, $\varphi^\beta_{\gamma,\varepsilon}(t)$ is decreasing as $\varepsilon\searrow0$ (see Proposition $3.3$ in \cite{JWLXZ1}). Then by using the same arguments as in section $3$, we have the following results.
\begin{pro}\label{35}
Let $\{\varepsilon_i\}\in[0,1)$ and $\{\gamma_i\}\in[0,\beta]$. We assume that $\varepsilon_i$ and $\gamma_i$ converge to $\varepsilon_\infty$ and $\gamma_\infty$ respectively. Then for any  $[\delta,T]$ with $0<\delta<T<\infty$, there exists $\alpha\in(0,1)$ such that $\varphi^\beta_{\gamma_i,\varepsilon_i}(t)$ converge to $\varphi^\beta_{\gamma_\infty,\varepsilon_\infty}(t)$ on $[\delta,T]\times M$ in $C^{\alpha'}$-sense for any $\alpha'\in(0,\alpha)$, and on $(0,\infty)\times (M\setminus D)$ the convergence is in $C^\infty_{loc}$-sense.
\end{pro}
\begin{rem}\label{43} From Proposition \ref{35}, we know that for any $[\delta,T]$ with $0<\delta<T<\infty$, the $C^\alpha$-norm of $\varphi^\beta_{\gamma,\varepsilon}(t)$ on $[\delta,T]\times M$ for some $\alpha\in(0,1)$ is continuous with respect to $\gamma\in[0,\beta]$ and $\varepsilon\in[0,1)$. If we fix $\varepsilon\in(0,1)$, then the $C^\infty$-norm of $\varphi^\beta_{\gamma,\varepsilon}(t)$ on $[\delta,T]\times M$ is continuous with respect to $\gamma\in[0,\beta]$.
\end{rem}
By following the arguments as that in Theorem \ref{2019021901} and Theorem \ref{20190002}, we have the following two theorems.
\begin{thm}\label{20190221001}
There exists uniform constant $D^\beta_0$ such that
\begin{equation}\label{19040001090}
\|u^\beta_{0,\varepsilon}(t)\|_{C^0(M)}+osc_{M}\varphi^\beta_{0,\varepsilon}(t)+tr_{\omega_\varepsilon^\beta}\omega^\beta_{0,\varepsilon} \leqslant D^\beta_{0}
\end{equation}
for any $\varepsilon\in(0,\frac{1}{2})$ and $t\geqslant 1$.
\end{thm}
\begin{proof} From the proof of Theorem \ref{2019021901}, we need only prove that $\frac{1}{V}\int_M\phi^\beta_{0,\varepsilon}(t) e^{-F_0+C_{\beta,\varepsilon}}\frac{dV_0}{(\varepsilon^2+|s|_h^2)^{1-\beta}}$ is decreasing with respect to $t$. By Jensen's inequality and $(\ref{04})$, we have
\begin{equation*}
\begin{split}
&\ \frac{d}{dt}\frac{1}{V}\int_M\phi^\beta_{0,\varepsilon}(t) e^{-F_0+C_{\beta,\varepsilon}}\frac{dV_0}{(\varepsilon^2+|s|_h^2)^{1-\beta}}\\
&=\frac{1}{V}\int_M\Big(\log\frac{(\omega^\beta_{0,\varepsilon}(t))^n}{\omega^n_0} +F_0+\beta\varphi_\varepsilon+(1-\beta)\log(\varepsilon^2+|s|_h^2)\Big)e^{-F_0+C_{\beta,\varepsilon}}\frac{dV_0}{(\varepsilon^2+|s|_h^2)^{1-\beta}}\\
&\leqslant \log \Big(\frac{1}{V}\int_M e^{\beta\varphi_\varepsilon+C_{\beta,\varepsilon}}dV^\beta_{0,\varepsilon}(t)\Big)\leqslant C_{\beta,\varepsilon}\leqslant0.
\end{split}
\end{equation*}
Then we can get this theorem by following the arguments in Theorem \ref{2019021901}.
\end{proof}
\begin{thm}\label{20190221002}
The conical K\"ahler-Ricci flow $(CKRF^\beta_{0})$ converges to the conical K\"ahler-Einstein metric $\omega_{\varphi_\beta}$ in $C_{loc}^{\infty}$-topology outside divisor $D$ and globally in $C^{\alpha,\beta}$-sense for any $\alpha\in(0,\min\{1,\frac{1}{\beta}-1\})$.
\end{thm}
From the arguments in \cite{LW,JWLXZ,RGY,QSZ}, we have the following Sobolev inequalities along $(TKRF^\beta_{\gamma,\varepsilon})$ by using Theorem \ref{Sobolev}, Theorem \ref{Sobolev1} and Proposition \ref{2019022401}.
\begin{thm}\label{1.8.5.1kk} Let  $M$ be a Fano manifold with complex dimension $n\geq2$ and $\omega^\beta_{\gamma,\varepsilon}(t)$ be a solution of the twisted K\"ahler-Ricci flow $(TKRF^\beta_{\gamma,\varepsilon})$. There exist uniform constants A and B depending only on $\|\varphi_0\|_{L^\infty(M)}$, $n$, $\beta$ and $\omega_0$ such that
\begin{equation*}
\begin{split}
\label{3.22.31kk}(\int_{M}v^{\frac{2n}{n-1}}dV^\beta_{\gamma,\varepsilon}(t))^{\frac{n-1}{n}}&\leqslant A\Big(\int_{M}4|\nabla v|_{\omega^\beta_{\gamma,\varepsilon}(t)}^{2}dV^\beta_{\gamma,\varepsilon}(t)+\int_{M}\big(R(\omega^\beta_{\gamma,\varepsilon}(t))-(\beta-\gamma)tr_{\omega^\beta_{\gamma,\varepsilon}(t)}\omega_{\varphi_\varepsilon}\\
&\ \ \ \ \ -(1-\beta)tr_{\omega^\beta_{\gamma,\varepsilon}(t)}\theta_{\varepsilon}+B\big)v^{2}dV^\beta_{\gamma,\varepsilon}(t)\Big)
\end{split}
\end{equation*}
for any $v\in W^{1,2}(M,\omega^\beta_{\gamma,\varepsilon}(t))$, $\gamma\in(0,1]$, $\varepsilon>0$ and $t\geq 1$.
\end{thm}
\begin{thm}\label{1.8.5.10001k} Let  $M$ be a Fano manifold with complex dimension $1$ and $\omega^\beta_{\gamma,\varepsilon}(t)$ be a solution of the twisted K\"ahler-Ricci flow $(TKRF^\beta_{\mu_\gamma,\varepsilon})$. For any $n_0>1$, there exist uniform constants A and B depending only on $\|\varphi_0\|_{L^\infty(M)}$, $n_0$, $\beta$ and $\omega_0$ such that
\begin{equation*}
\begin{split}
\label{3.22.310001k}(\int_{M}v^{\frac{2n_0}{n_0-1}}dV^\beta_{\gamma,\varepsilon}(t))^{\frac{n_0-1}{n_0}}&\leqslant A\Big(\int_{M}4|\nabla v|_{\omega^\beta_{\gamma,\varepsilon}(t)}^{2}dV^\beta_{\gamma,\varepsilon}(t)+\int_{M}\big(R(\omega^\beta_{\gamma,\varepsilon}(t))\\
&\ \ \ \ \ \ \ \ \ -(\beta-\gamma)tr_{\omega^\beta_{\gamma,\varepsilon}(t)}\omega_{\varphi_\varepsilon}-(1-\beta)tr_{\omega^\beta_{\gamma,\varepsilon}(t)}\theta_{\varepsilon}+B\big)v^{2}dV^\beta_{\gamma,\varepsilon}(t)\Big)
\end{split}
\end{equation*}
for any $v\in W^{1,2}(M,\omega^\beta_{\gamma,\varepsilon}(t))$, $\gamma\in(0,1]$, $\varepsilon>0$ and $t\geq 1$.
\end{thm}
Combining these uniform Sobolev inequalities with Proposition \ref{2019022401}, we have the following Theorem by following Jiang's work (Theorem $1.12$ in \cite{WSJ1}). 
\begin{thm}\label{Jiang1}
Let  $M$ be a Fano manifold of complex dimension $n$ and $\omega^\beta_{\gamma,\varepsilon}(t)$ be a solution of the twisted K\"ahler-Ricci flow $(TKRF^\beta_{\gamma,\varepsilon})$. Let $f$ be a non-negative Lipschitz continuous function on $[0,\infty)\times M$ satisfying 
\begin{equation}\label{20190306001k}
\frac{\partial}{\partial t}f\leqslant\Delta_{\omega^\beta_{\gamma,\varepsilon}(t)}f+af
\end{equation}
on $[0,\infty)\times M$ in the weak sense, where $a\geqslant0$. For any $p>0$, there exists a constant $C$ depending only on $\|\varphi_0\|_{L^\infty(M)}$, $\beta$, $a$, $p$, $\omega_0$ and $n_0$ ($n_0=n$ if $n\geqslant2$, $n_0>1$ if $n=1$) such that
\begin{equation}\label{20190306002l}
\sup\limits_M|f(t,x)|\leqslant\frac{C}{(t-T)^{\frac{n_0+1}{p}}}\Big(\int_T^{T+1}\int_M f^p dV^\beta_{\gamma,\varepsilon}(t)dt\Big)^{\frac{1}{p}}
\end{equation}
holds for any $T<t<T+1$ with $T\geqslant1$, $\gamma\in(0,1]$ and $\varepsilon>0$.
\end{thm}
Denote $L^\beta_0=\max(\sup\limits_M tr_{\omega_\beta}\omega_{\varphi_\beta}+osc_M\varphi_\beta, D^\beta_0)$. By using the arguments as that in the proof of Lemma \ref{1901}, we have
\begin{lem}\label{20190221003}
There exists a constant $\tilde{\delta}>0$ such that
\begin{equation}\label{1903}
\|u^\beta_{\gamma,\varepsilon}(t)\|_{C^0(M)}+osc_{M}\varphi^\beta_{\gamma,\varepsilon}(t)+\sup\limits_M tr_{\omega_\varepsilon^\beta}\omega^\beta_{\gamma,\varepsilon} \leqslant L^\beta_{0}+1
\end{equation}
for any $\varepsilon\in(0,\tilde{\delta})$, $\gamma\in(0,\tilde{\delta})$ and $t\in[\frac{1}{\tilde{\delta}},+\infty)$.
\end{lem}
Following the arguments in the proof of Theorem \ref{20190223001}, we improve the uniform Perelman's estimates in Theorem \ref{20190314001} independent of $\gamma\in(0,\beta]$ along $(TKRF^\beta_{\gamma,\varepsilon})$.
\begin{thm}\label{201902230010} Let $\omega^\beta_{\gamma,\varepsilon}(t)$ be a solution of the twisted K\"ahler Ricci flow $(TKRF^\beta_{\gamma,\varepsilon})$. There exists a uniform constant $C$, such that
\begin{equation}\begin{split}\label{p0}
&\ |R(\omega^\beta_{\gamma,\varepsilon}(t))-(\beta-\gamma)tr_{\omega^\beta_{\gamma,\varepsilon}(t)}\omega_{\varphi_\varepsilon}-(1-\beta)tr_{\omega^\beta_{\gamma,\varepsilon}(t)}\theta_{\varepsilon}|\leqslant C,\\
&\ \|u^\beta_{\gamma,\varepsilon}(t)\|_{C^{1}(\omega^\beta_{\gamma,\varepsilon}(t))}\leqslant C,\\
&\ diam(M,\omega^\beta_{\gamma,\varepsilon}(t))\leqslant C
\end{split}
\end{equation}
hold for any $\gamma\in(0,\beta]$, $t\geqslant 1$ and $\varepsilon\in(0,\tilde{\delta})$, where $\tilde{\delta}$ is the constant in Lemma \ref{20190221003}.
\end{thm}
From the definitions $(\ref{9090990})$ and $(\ref{9090991})$, for $\gamma\in(0,\beta]$, we have the following inequalities.
\begin{equation}\label{20190213133}
\begin{split}
\mathcal{M}^\beta_\gamma(\phi)&=\mathcal{M}_\beta(\phi)+(\beta-\gamma)(I_{\omega_0}-J_{\omega_0})(\phi)-(\beta-\gamma)\frac{1}{V}\int_M \varphi_\beta (dV_0-dV_\phi)\\
&\geqslant \mathcal{M}_\beta(\varphi_\beta)+\frac{(\beta-\gamma)}{n}J_{\omega_0}(\phi)-(\beta-\gamma)\frac{1}{V}\int_M \varphi_\beta (dV_0-dV_\phi)\\
&\geqslant -C,
\end{split}
\end{equation}
where constant $C$ independent of $\gamma$.
\begin{thm}\label{51} Fix a $\beta_0\in(0,\tilde{\delta})$ obtained in Lemma \ref{20190221003}, there exists a constant $C_{\beta_0}$ such that
\begin{equation}\label{52}
\|\psi^\beta_{\beta_0,\varepsilon}(t)\|_{C^0(M)}\leqslant C_{\beta_0}
\end{equation}
for any $\varepsilon\in(0,\tilde{\delta})$ and $t\in[0,\infty)$.
\end{thm}
{\it Proof of Lemma \ref{0080}.}\ \ We denote $A=\max(\|\varphi_\beta\|_{C^0(M)}+\frac{\hat{C}_{\beta_0}+|\xi_\beta|}{\beta_0},C_{\beta_0})$. If the lemma is not true. For $\delta_1=\min(1,\tilde{\delta})$, there are $\varepsilon_1\in(0,\delta_1)$, $\gamma'_1\in(\beta_0,\beta)$ and $t'_1\in[\frac{1}{\delta_1},\infty)$, such that
\begin{equation}\label{53}
\|\psi^{\beta}_{\gamma'_1,\varepsilon_1}(t'_1)\|_{C^0(M)}>A+1.
\end{equation}
Assume that $t'_1\in[1,T_1]$, we have
\begin{equation}\label{55}
\sup\limits_{t\in[1,T_1]}\|\psi^{\beta}_{\gamma'_1,\varepsilon_1}(t)\|_{C^0(M)}>A+1.
\end{equation}
Combining Remark \ref{43} with $(\ref{52})$, there exists $\gamma_1\in(\beta_0,\gamma'_1)$, such that
\begin{equation}\label{56}
\sup\limits_{t\in[1,T_1]}\|\psi^{\beta}_{\gamma_1,\varepsilon_1}(t)\|_{C^0(M)}=A+1.
\end{equation}
Then there exists $t_1\in[1,T_1]$ such that
\begin{equation}\label{57}
\|\psi^{\beta}_{\gamma_1,\varepsilon_1}(t_1)\|_{C^0(M)}\geqslant A+\frac{7}{8}.
\end{equation}
For $\delta_2=\min(\frac{1}{2}, \frac{1}{T_1+1}, \varepsilon_1)$, there exist $\varepsilon_2\in(0,\delta_2)$, $\gamma'_2\in(\beta_0,\beta)$ and $t'_2\in[\frac{1}{\delta_2},\infty)$, such that
\begin{equation}\label{53}
\|\psi^{\beta}_{\gamma'_2,\varepsilon_2}(t'_2)\|_{C^0(M)}>A+1.
\end{equation}
Assume $t'_2\in[T_1+1,T_2]$, then we have
\begin{equation}\label{55}
\sup\limits_{t\in[T_1+1,T_2]}\|\psi^{\beta}_{\gamma'_2,\varepsilon_2}(t)\|_{C^0(M)}>A+1.
\end{equation}
Combining Remark \ref{43} with $(\ref{52})$ again, there exists $\gamma_2\in(\beta_0,\gamma'_2)$, such that
\begin{equation}\label{56}
\sup\limits_{t\in[T_1+1,T_2]}\|\psi^{\beta}_{\gamma_2,\varepsilon_2}(t)\|_{C^0(M)}=A+1,
\end{equation}
and then there exists $t_2\in[T_1+1,T_2]$ such that
\begin{equation}\label{57}
\|\psi^{\beta}_{\gamma_2,\varepsilon_2}(t_2)\|_{C^0(M)}\geqslant A+\frac{7}{8}.
\end{equation}
After repeating above process, we get a subsequence $\psi^{\beta}_{\gamma_i,\varepsilon_i}(t_i)$ with $\varepsilon_i\searrow0$, $\gamma_i\in(\beta_0,\beta)$, $t_i\in[T_{i-1}+1,T_{i}]$ and $t_i\nearrow\infty$ satisfying
\begin{equation}\label{58}
\sup\limits_{t\in[T_{i-1}+1,T_{i}]}\|\psi^{\beta}_{\gamma_i,\varepsilon_i}(t)\|_{C^0(M)}=A+1\ \ and \ \ \|\psi^{\beta}_{\gamma_i,\varepsilon_i}(t_i)\|_{C^0(M)}\geqslant A+\frac{7}{8}.
\end{equation}
For any $t\in[T_{i-1},T_{i-1}+1]$, we assume the maximum (or minimun) point of $\psi^{\beta}_{\gamma_i,\varepsilon_i}(t)$ on $M$ is $z_1$ (or $z_2$), then by Proposition \ref{1.8.5.3} and $(\ref{58})$,
\begin{equation}\label{59}
\begin{split}
osc_M\psi^{\beta}_{\gamma_i,\varepsilon_i}(t)&=\sup\limits_M \psi^{\beta}_{\gamma_i,\varepsilon_i}(t)-\inf\limits_M\psi^{\beta}_{\gamma_i,\varepsilon_i}(t)=\psi^{\beta}_{\gamma_i,\varepsilon_i}(t,z_1)-\psi^{\beta}_{\gamma_i,\varepsilon_i}(t,z_2)\\
&=\psi^{\beta}_{\gamma_i,\varepsilon_i}(t,z_1)-\psi^{\beta}_{\gamma_i,\varepsilon_i}(T_{i-1}+1,z_1)+\psi^{\beta}_{\gamma_i,\varepsilon_i}(T_{i-1}+1,z_1)\\
&\ \ \ -\psi^{\beta}_{\gamma_i,\varepsilon_i}(T_{i-1}+1,z_2)+\psi^{\beta}_{\gamma_i,\varepsilon_i}(T_{i-1}+1,z_2)-\psi^{\beta}_{\gamma_i,\varepsilon_i}(t,z_2)\\
&\leqslant 2\|\dot{\psi}^{\beta}_{\gamma_i,\varepsilon_i}(t)\|_{C^0(M)}+2\|\psi^{\beta}_{\gamma_i,\varepsilon_i}(T_{i-1}+1)\|_{C^0(M)}\leqslant C,
\end{split}
\end{equation}
where constant $C$ independent of $\gamma_i$, $\varepsilon_i$ and $t\in[T_{i-1},T_{i-1}+1]$. On the other hand, we can write equation $(\ref{TKRF5})$ as
\begin{equation}\label{60}
(\omega_0+\sqrt{-1}\partial\bar{\partial}\psi^{\beta}_{\gamma,\varepsilon}(t))^n=e^{\dot{\psi}^{\beta}_{\gamma,\varepsilon}(t)-\gamma\psi^{\beta}_{\gamma,\varepsilon}(t)-(\beta-\gamma)\varphi_{\varepsilon}-F_0}\frac{\omega_0^n}{(\varepsilon^2+|s|_h)^{2(1-\beta)}}.
\end{equation}
We integrate above equality on both sides, there exists a uniform constant $C$ such that
\begin{equation}\label{61}
\inf\limits_M\psi^{\beta}_{\gamma,\varepsilon}(t)\leqslant C\ \ \ and\ \ \ \sup\limits_M\psi^{\beta}_{\gamma,\varepsilon}(t)\geqslant -C
\end{equation}
for any $\gamma\in[\beta_0,\beta]$, $\varepsilon\in(0,1)$ and $t\geqslant 1$. Then by $(\ref{59})$, we have
\begin{equation}
\begin{split}
\|\psi^{\beta}_{\gamma_i,\varepsilon_i}(t)\|_{C^0(M)}\leqslant C
\end{split}
\end{equation}
for any $t\in[T_{i-1},T_{i-1}+1]$, $\gamma_i$ and $\varepsilon_i$. Hence there exists uniform constant $C$ such that
\begin{equation}\label{62}
\sup\limits_{t\in[T_{i-1},T_{i}]}\|\psi^{\beta}_{\gamma_i,\varepsilon_i}(t)\|_{C^0(M)}\leqslant C
\end{equation}
for any $\gamma_i$ and $\varepsilon_i$. By Lemma \ref{22} and Proposition \ref{1.8.5.3}, there exists constant $C$ such that
\begin{equation}\label{19240001}
e^{-\frac{C}{t-T_{i-1}}}\omega^{\beta}_{\varepsilon_i}\leqslant\omega^\beta_{\gamma_i,\varepsilon_i}(t)\leqslant e^{\frac{C}{t-T_{i-1}}}\omega^{\beta}_{\varepsilon_i}\ \ \ on\ \ (T_{i-1},T_{i-1}+\frac{1}{4}]\times M.
\end{equation}
Then for any $\gamma_i$ and $\varepsilon_i$, at $t=T_{i-1}+\frac{1}{4}$,
\begin{equation}C^{-1}\omega^\beta_{\varepsilon_i}\leqslant\omega^\beta_{\gamma_i,\varepsilon_i}(T_{i-1}+\frac{1}{4})\leqslant C\omega^\beta_{\varepsilon_i}.
\end{equation}
Then by Proposition $3.1$ in \cite{JWLXZ}, $(\ref{62})$ and Proposition \ref{1.8.5.3}, on $[T_{i-1}+\frac{1}{4},T_i]\times M$, we have
\begin{equation}
e^{-C}\omega^\beta_{\varepsilon_i}\leqslant\omega^\beta_{\gamma_i,\varepsilon_i}(t)\leqslant e^{C}\omega^\beta_{\varepsilon_i}.
\end{equation}
By Evans-Krylov-Safonov's estimates, for any $B_r(p)\subset\subset M\setminus D$, there is uniform constant $C_{k,p,r}$ such that
\begin{equation*}\label{1924}
\|\psi^\beta_{\gamma_i,\varepsilon_i}(t)\|_{C^{k}\big(B_r(p)\big)}\leqslant C_{k,p,r}
\end{equation*}
for any $t\in[T_{i-1}+\frac{1}{2},T_i]$, $\gamma_i$ and $\varepsilon_i$. Hence, for any $K\subset\subset M\setminus D$, we have
\begin{equation}\label{64}
A+\frac{7}{8}\leqslant\|\psi^{\beta}_{\gamma_i,\varepsilon_i}(t_i)\|_{C^0(M)}\leqslant A+1 \ \ and\ \ \|\psi^{\beta}_{\gamma_i,\varepsilon_i}(t_i)\|_{C^k(K)}\leqslant C_{k,K}
\end{equation}
for any $i$ and $k\in\mathbb{N}^+$. Then $\psi^{\beta}_{\gamma_i,\varepsilon_i}(t_i)$ (by taking a subsequence if necessary) converge to a function $\psi^{\beta}_{\gamma_\infty}\in C^0(M)\bigcap C^\infty(M\setminus D)$.

Let $u^{\beta}_{\gamma_\infty}$ be the twisted Ricci potential of $\omega^{\beta}_{\gamma_\infty}=\omega_0+\sqrt{-1}\partial\bar{\partial}\psi^{\beta}_{\gamma_\infty}$. The Lebesgue Dominated Convergence theorem implies that $\frac{1}{V}\int_M e^{-u^{\beta}_{\gamma_\infty}}dV^\beta_{\gamma_\infty}=1$. Since $A^\beta_{\gamma,\varepsilon}(t)$ is increasing along the flow $(TKRF^\beta_{\gamma,\varepsilon})$, for any $j<i$,
\begin{equation}\label{65}
A^\beta_{\gamma_i,\varepsilon_i}(t_j)\leqslant A^\beta_{\gamma_i,\varepsilon_i}(t_i)\leqslant0,
\end{equation}
the second inequality due to Jensen's inequality. Let $i\rightarrow\infty$, we have
\begin{equation}\label{66}
A^\beta_{\gamma_\infty}(t_j)\leqslant \frac{1}{V}\int_M u^{\beta}_{\gamma_\infty}e^{-u^{\beta}_{\gamma_\infty}}dV^\beta_{\gamma_\infty}\leqslant0.
\end{equation}
Since $\mathcal{M}^\beta_{\gamma_\infty}$ is bounded from below (see $(\ref{20190213133})$), by Theorem \ref{46}, we obtain
\begin{equation}\label{67}
\frac{1}{V}\int_M u^{\beta}_{\gamma_\infty}e^{-u^{\beta}_{\gamma_\infty}}dV^\beta_{\gamma_\infty}=0
\end{equation}
after letting $j\rightarrow\infty$. Since function $\log$ is strictly concave, we conclude $u^{\beta}_{\gamma_\infty}=0$ by using Jensen's inequality and the normalization of $u^{\beta}_{\gamma_\infty}$. This means that $\omega^{\beta}_{\gamma_\infty}$ is a twisted conical K\"ahler-Einstein metric with cone angle $2\pi\beta$ along $D$ and satisfies
\begin{equation}\label{68}
Ric(\omega^\beta_{\gamma_\infty})=\gamma_\infty\omega^\beta_{\gamma_\infty}+(\beta-\gamma_\infty)\omega_{\varphi_\beta}+(1-\beta)[D].
\end{equation}
At the same time, we also have
\begin{equation}\label{69}
Ric(\omega_{\varphi_\beta})=\gamma_\infty\omega_{\varphi_\beta}+(\beta-\gamma_\infty)\omega_{\varphi_\beta}+(1-\beta)[D].
\end{equation}
Both $\psi^\beta_{\gamma_\infty}$ and $\varphi_\beta$ are bounded (in fact, they are H\"older continuous). Berndtsson's uniqueness theorem \cite{BBERN} implies that
\begin{equation}\label{70}
\begin{split}
F^*\omega^\beta_{\gamma_\infty}&=\omega_{\varphi_\beta}\\
F^*\Big((\beta-\gamma_\infty)\omega_{\varphi_\beta}+(1-\beta)[D]\Big)&=(\beta-\gamma_\infty)\omega_{\varphi_\beta}+(1-\beta)[D],
\end{split}
\end{equation}
where $F\in Aut(M)$ is generated by a holomorphic vector $X$ on $M$. Since $(\beta-\gamma_\infty)\omega_{\varphi_\beta}$ has zero Lelong number, then by Siu's decomposition for a positive closed current, we have
\begin{equation}\label{71}
F^*[D]=[D]
\end{equation}
which implies that $X$ is tangential to $D$. Then $f=id$ because we assume that there is no nontrivial holomorphic vector fields tangent to $D$. Hence $\omega^\beta_{\gamma_\infty}=\omega_{\varphi_\beta}$, and then $\psi^\beta_{\gamma_\infty}=\varphi_\beta+C_1$. On the other hand, we have
\begin{equation*}\label{1922001}
(\omega_0+\sqrt{-1}\partial\bar{\partial}\psi^\beta_{\gamma_i,\varepsilon_i}(t_i))^n=e^{\dot{\psi}^\beta_{\gamma_i,\varepsilon_i}(t_i)-\gamma_i\psi^\beta_{\gamma_i,\varepsilon_i}(t_i)-(\beta-\gamma_i)\varphi_{\varepsilon_i}-F_0}\frac{\omega_0^n}{(\varepsilon_i^2+|s|_h)^{2(1-\beta)}}.
\end{equation*}
Let $\dot{\psi}^\beta_{\gamma_i,\varepsilon_i}(t_i)=u^\beta_{\gamma_i,\varepsilon_i}(t_i)+c^\beta_{\gamma_i,\varepsilon_i}(t_i)$, where constants $c^\beta_{\gamma_i,\varepsilon_i}(t_i)$ are uniformly bounded and hence converge to a constant $C_2$ (by taking a subsequence if necessary). Since $u^\beta_{\gamma_i,\varepsilon_i}(t_i)$ converge to $0$ in $C^\infty_{loc}$-topology in $M\setminus D$, $\dot{\psi}^\beta_{\gamma_i,\varepsilon_i}(t_i)$ converge to $C_2$ in $C^\infty_{loc}$-topology in $M\setminus D$. Letting $i\rightarrow\infty$ in above equation, we have
\begin{equation}\label{1922002}
(\omega_0+\sqrt{-1}\partial\bar{\partial}\varphi_{\beta})^n=e^{C_2-\gamma_\infty C_1-\beta\varphi_{\beta}-F_0}\frac{\omega_0^n}{|s|_h^{2(1-\beta)}}.
\end{equation}
Equation $(\ref{2019022301})$ implies that $C_1=\frac{C_2-\xi_\beta}{\gamma_\infty}$. From Proposition \ref{1.8.5.3}, we conclude that $|C_1|\leqslant\frac{\hat{C}_{\beta_0}+|\xi_\beta|}{\beta_0}$. Letting $i\rightarrow\infty$ in $(\ref{58})$, we have
\begin{equation*}\label{72}
\frac{\hat{C}_{\beta_0}+|\xi_\beta|}{\beta_0}+\|\varphi_\beta\|_{C^0(M)}\geqslant \|\psi^{\beta}_{\gamma_\infty}\|_{C^0(M)}\geqslant A+\frac{7}{8}\geqslant \|\varphi_\beta\|_{C^0(M)}+\frac{\hat{C}_{\beta_0}+|\xi_\beta|}{\beta_0}+\frac{7}{8}.
\end{equation*}
This leads to a contradiction. Thus Lemma \ref{0080} is proved.

\medskip

By using Lemma \ref{0080} and Remark \ref{43}, we get the uniform $C^0$-estimates for $\psi^\beta_{\beta,\varepsilon}(t)=\psi_{\beta,\varepsilon}(t)$.
\begin{pro} There exists a uniform constant $B_\beta$ such that
\begin{equation}\label{731}
\|\psi_{\beta,\varepsilon}(t)\|_{C^0(M)}\leqslant B_\beta
\end{equation}
for any $\varepsilon\in(0,\delta_{\beta_0})$ and $t\in[0,+\infty)$.
\end{pro}
At last, we prove Lemma \ref{009}.

\medskip

{\it Proof of Lemma \ref{009}.}\ \   We only prove that there exists $\delta>0$ such that
\begin{equation}\label{73}
\|\psi_{\gamma,\varepsilon}(t)\|_{C^0(M)}\leqslant\max(\|\varphi_\beta\|_{C^0(M)}+\frac{C_\beta+|\xi_\beta|}{\beta},B_{\beta})+1
\end{equation}
for any $\varepsilon\in(0,\delta)$, $\gamma\in[\beta,\beta+\delta)$ and $t\in[\frac{1}{\delta},+\infty)$, where $\xi_\beta$ is the constant in $(\ref{2019022301})$, $B_{\beta}$ is the constant in $(\ref{731})$ and $C_\beta$ comes from Proposition \ref{1.8.5.3.11}. The other direction is similar.

Denote $L=\max(\|\varphi_\beta\|_{C^0(M)}+\frac{C_\beta+|\xi_\beta|}{\beta},B_{\beta})$. If this lemma is not true. For $\delta_1=\min(1, \delta_{\beta_0})$, there exist $\varepsilon_1\in(0,\delta_1)$, $\gamma'_1\in(\beta,\beta+\delta_1)$ and $t'_1\in[\frac{1}{\delta_1},\infty)$, such that
\begin{equation}\label{74}
\|\psi_{\gamma'_1,\varepsilon_1}(t'_1)\|_{C^0(M)}>L+1.
\end{equation}
Assume $t'_1\in[1,T_1]$, then we have
\begin{equation}\label{75}
\sup\limits_{t\in[1,T_1]}\|\psi_{\gamma'_1,\varepsilon_1}(t)\|_{C^0(M)}>L+1.
\end{equation}
Combining Remark \ref{201908001} with $(\ref{731})$, there exists $\gamma_1\in(\beta,\gamma'_1)$, such that
\begin{equation}\label{76}
\sup\limits_{t\in[1,T_1]}\|\psi_{\gamma_1,\varepsilon_1}(t)\|_{C^0(M)}=L+1.
\end{equation}
Then there exists $t_1\in[1,T_1]$ such that
\begin{equation}\label{77}
\|\psi_{\gamma_1,\varepsilon_1}(t_1)\|_{C^0(M)}\geqslant L+\frac{7}{8}.
\end{equation}
For $\delta_2=\min(\frac{1}{2}, \frac{1}{T_1+1}, \varepsilon_1, \gamma_1-\beta)$, there exist $\varepsilon_2\in(0,\delta_2)$, $\gamma'_2\in(\beta,\beta+\delta_2)$ and $t'_2\in[\frac{1}{\delta_2},\infty)$, such that
\begin{equation}\label{78}
\|\psi_{\gamma'_2,\varepsilon_2}(t'_2)\|_{C^0(M)}>L+1.
\end{equation}
Assume $t'_2\in[T_1+1,T_2]$, then we have
\begin{equation}\label{79}
\sup\limits_{t\in[T_1+1,T_2]}\|\psi_{\gamma'_2,\varepsilon_2}(t)\|_{C^0(M)}>L+1.
\end{equation}
Combining Remark \ref{201908001} with $(\ref{731})$ again, there exists $\gamma_2\in(\beta,\gamma'_2)$, such that
\begin{equation}\label{80}
\sup\limits_{t\in[T_1+1,T_2]}\|\psi_{\gamma_2,\varepsilon_2}(t)\|_{C^0(M)}=L+1,
\end{equation}
and then there exists $t_2\in[T_1+1,T_2]$ such that
\begin{equation}\label{81}
\|\psi_{\gamma_2,\varepsilon_2}(t_2)\|_{C^0(M)}\geqslant L+\frac{7}{8}.
\end{equation}
After repeating above process, we get a subsequence $\psi_{\gamma_i,\varepsilon_i}(t_i)$ with $\varepsilon_i\searrow0$, $\gamma_i\searrow\beta$, $t_i\in[T_{i-1}+1,T_{i}]$ and $t_i\nearrow\infty$ satisfying
\begin{equation}\label{82}
\sup\limits_{t\in[T_{i-1}+1,T_{i}]}\|\psi_{\gamma_i,\varepsilon_i}(t)\|_{C^0(M)}=L+1\ \ and \ \ \|\psi_{\gamma_i,\varepsilon_i}(t_i)\|_{C^0(M)}\geqslant L+\frac{7}{8}.
\end{equation}
By the similar arguments as in the proof of Lemma \ref{0080}, we conclude that there exists uniform constants $C$ independent of $\varepsilon_i$ and $\gamma_i$, such that
\begin{equation*}\label{87}
e^{-C}\omega^{\gamma_i}_{\varepsilon_i}\leqslant\omega_{\gamma_i,\varepsilon_i}(t)\leqslant e^{C}\omega^{\gamma_i}_{\varepsilon_i}\ \ \ on\ \ \ [T_{i-1}+\frac{1}{2},T_i]\times M.
\end{equation*}
The Evans-Krylov-Safonov's estimates imply that for any $K\subset\subset M\setminus D$ and $k\in\mathbb{N}^+$, there exists uniform constant $C_{k,K}$ such that for any $\gamma_i$, $\varepsilon_i$ and $t\in[T_{i-1}+\frac{3}{4},T_i]$,
\begin{equation*}\label{87001}
\|\psi_{\gamma_i,\varepsilon_i}(t)\|_{C^k(K)}\leqslant C_{k,K}.
\end{equation*}
Then we can conclude that $\psi_{\gamma_i,\varepsilon_i}(t_i)$ (by taking a subsequence if necessary) converge to a function $\psi_{\beta}\in C^0(M)\bigcap C^\infty(M\setminus D)$.

Let $u_{\beta}$ be the twisted Ricci potential of $\omega_{\psi_\beta}=\omega_0+\sqrt{-1}\partial\bar{\partial}\psi_{\beta}$, then $u_\beta$ satisfies $\frac{1}{V}\int_M e^{-u_{\beta}}dV_{\psi_{\beta}}=1$. Since $A_{\gamma,\varepsilon}(t)$ is increasing along the flow $(TKRF_{\gamma,\varepsilon})$, for any $j<i$,
\begin{equation}\label{89}
A_{\gamma_i,\varepsilon_i}(t_j)\leqslant A_{\gamma_i,\varepsilon_i}(t_i)\leqslant0,
\end{equation}
the second inequality due to Jensen's inequality. Let $i\rightarrow\infty$, we have
\begin{equation}\label{90}
A_\beta(t_j)\leqslant \frac{1}{V}\int_M u_{\beta}e^{-u_{\beta}}dV_\beta\leqslant0.
\end{equation}
Since the Log Mabuchi energy $\mathcal{M}_\beta$ is bounded from below, by Theorem \ref{17}, we obtain
\begin{equation}\label{91}
\frac{1}{V}\int_M u_{\beta}e^{-u_{\beta}}dV_\beta=0
\end{equation}
after letting $j\rightarrow\infty$. Since $\log$ is strictly concave, we conclude $u_{\beta}=0$ by using Jensen's inequality and its normalization. This means that $\omega_{\psi_\beta}$ is a conical K\"ahler-Einstein metric with cone angle $2\pi\beta$ along $D$. Since both $\psi_\beta$ and $\varphi_\beta$ are bounded (in fact, they are H\"older continuous), Berndtsson's uniqueness theorem \cite{BBERN} implies that
$\omega_{\psi_\beta}=\omega_{\varphi_\beta}$, and then $\psi_\beta=\varphi_\beta+C_1$. On the other hand, we have
\begin{equation}\label{1922001111}
(\omega_0+\sqrt{-1}\partial\bar{\partial}\psi_{\gamma_i,\varepsilon_i}(t_i))^n=e^{\dot{\psi}_{\gamma_i,\varepsilon_i}(t_i)-\gamma_i\psi_{\gamma_i,\varepsilon_i}-F_0}\frac{\omega_0^n}{(\varepsilon_i^2+|s|_h)^{2(1-\gamma_i)}}.
\end{equation}
Let $\dot{\psi}_{\gamma_i,\varepsilon_i}(t_i)=u_{\gamma_i,\varepsilon_i}(t_i)+c_{\gamma_i,\varepsilon_i}(t_i)$, where constants $c_{\gamma_i,\varepsilon_i}(t_i)$ are uniformly bounded and hence converge to a constant $C_2$ (by taking a subsequence if necessary). Since $u_{\gamma_i,\varepsilon_i}(t_i)$ converge to $0$ in $C^\infty_{loc}$-topology in $M\setminus D$, $\dot{\psi}_{\gamma_i,\varepsilon_i}(t_i)$ converge to $C_2$ in $C^\infty_{loc}$-topology in $M\setminus D$. Letting $i\rightarrow\infty$ in $(\ref{1922001111})$, we have
\begin{equation}\label{192200211}
(\omega_0+\sqrt{-1}\partial\bar{\partial}\varphi_{\beta})^n=e^{C_2-\beta C_1-\beta\varphi_{\beta}-F_0}\frac{\omega_0^n}{|s|_h^{2(1-\beta)}}.
\end{equation}
Equation $(\ref{2019022301})$ implies that $C_1=\frac{C_2-\xi_\beta}{\beta}$. From Proposition \ref{1.8.5.3.11}, we conclude that $|C_1|\leqslant\frac{C_{\beta}+|\xi_\beta|}{\beta}$. Hence the following inequalities are deduced after we let $i\rightarrow\infty$ in $(\ref{82})$.
\begin{equation}\label{96}
\frac{C_{\beta}+|\xi_\beta|}{\beta}+\|\varphi_\beta\|_{C^0(M)}\geqslant \|\psi_{\beta}\|_{C^0(M)}\geqslant L+\frac{7}{8}\geqslant \frac{C_{\beta}+|\xi_\beta|}{\beta}+\|\varphi_\beta\|_{C^0(M)}+\frac{7}{8}.
\end{equation}
This leads to a contradiction. Thus Lemma \ref{009} is proved.

\medskip

Fix a $\gamma\in(\beta-\delta,\beta+\delta)$ obtained in Lemma \ref{009}, combining Proposition \ref{1.8.5.3.11} with Proposition $3.1$ in \cite{JWLXZ}, we know that there exists uniform constant $C$ depending only on $\|\varphi_0\|_{L^\infty(M)}$, $n$, $\beta$ and $\omega_0$ such that
\begin{equation*}
C^{-1}\omega^{\gamma}_{\varepsilon}\leqslant\omega_{\gamma,\varepsilon}(t)\leqslant C\omega^{\gamma}_{\varepsilon}\ \ \ on\ \ \ [1,\infty)\times M
\end{equation*}
for any $\varepsilon\in(0,\delta)$. Then for $k\in \mathbb{N}^{+} $ and $K\subset\subset M\setminus D$, there exists constant $C_{k,K}$ depending only on $\|\varphi_0\|_{L^\infty(M)}$, $n$, $\beta$, $k$, $\omega_0$ and $dist_{\omega_{0}}(K,D)$, such that for $\varepsilon\in(0,\delta)$ and $t\geqslant1$, we have
\begin{eqnarray}\|\psi_{\gamma,\varepsilon}(t)\|_{C^{k}(K)}\leq C_{k,K}.
\end{eqnarray}
At the same time, the twisted Mabuchi energy $\mathcal{M}_{\gamma,\varepsilon}$ are uniformly bounded from below along the twisted K\"ahler-Ricci flow $(TKRF_{\gamma,\varepsilon})$, that is, there exists uniform constant $C$ such that
\begin{equation*}
\begin{split}
\mathcal{M}_{\gamma,\varepsilon}(\psi_{\gamma,\varepsilon}(t))&=-\gamma\Big(I_{\omega_{0}}(\psi_{\gamma,\varepsilon}(t))-J_{\omega_{0}}(\psi_{\gamma,\varepsilon}(t))\Big)+\frac{1}{V}\int_{M}\log\frac{\omega^n_{\gamma,\varepsilon}(t)}{\omega_{0}^{n}}dV_{\gamma,\varepsilon}(t)\\
&\ \ \ \ \ -\frac{1}{V}\int_{M}\Big(F_0+(1-\gamma)\log(\varepsilon^2+|s|^2_h)\Big)(dV_{0}-dV_{\gamma,\varepsilon}(t))\\
&=-\gamma\Big(I_{\omega_{0}}(\psi_{\gamma,\varepsilon}(t))-J_{\omega_{0}}(\psi_{\gamma,\varepsilon}(t))\Big)+\frac{1}{V}\int_{M}\Big(\dot{\psi}_{\gamma,\varepsilon}(t)-\gamma\psi_{\gamma,\varepsilon}(t)\Big)dV_{\gamma,\varepsilon}(t)\\
&\ \ \ \ \ -\frac{1}{V}\int_{M}\Big(F_0+(1-\gamma)\log(\varepsilon^2+|s|^2_h)\Big)dV_{0}\\
&\geqslant -C
\end{split}
\end{equation*}
for any $\varepsilon\in(0,\delta)$ and $t\geqslant1$. Then by using the arguments in \cite{JWLXZ, JWLXZ1}, we can prove that the conical K\"ahler-Ricci flow $(CKRF_{\gamma})$ converges to a conical K\"ahler-Einstein metric with cone angle $2\pi\gamma$ along $D$ in $C_{loc}^{\infty}$-topology outside $D$ and globally in $C^{\alpha,\gamma}$-sense for any $\alpha\in(0,\min\{1,\frac{1}{\gamma}-1\})$.


\begin{thebibliography}{9}

\bibitem{RB}
 \textit{R.~Berman},
 A thermodynamical formalism for Monge-Ampere equations, Moser-Trudinger inequalities and K\"ahler-Einstein metrics, Advances in Mathematics, \textbf{248} (2013), 1254--1297.

\bibitem{BBERN}
 \textit{B.~Berndtsson},
 A Brunn-Minkowski type inequality for Fano manifolds and the Bando-Mabuchi uniqueness theorem, Inventiones mathematicae, \textbf{200} (2014), 149--200.

\bibitem{SB}
\textit{S.~Brendle},
Ricci flat K\"ahler metrics with edge singularities, International Mathematics Research Notices, \textbf{24}(2013), 5727--5766.

\bibitem{CGP}
\textit{F.~Campana}, \textit{H.~Guenancia} and \textit{M.~P$\breve{a}$un},
Metrics with cone singularities along normal crossing divisors and holomorphic tensor fields, Annales scientifiques de l$^{'}\acute{E}$. NS, \textbf{46} (2013), 879--916.

\bibitem{HDC}
\textit{H.~D.~Cao},
Deformation of K\"ahler metrics to K\"ahler-Einstein metrics on compact K\"ahler manifolds, Inventiones Mathematicae, \textbf{81} (1985), 359--372.

\bibitem{CDS1}
\textit{X.~X.~Chen}, \textit{S.~Donaldson} and \textit{S.~Sun},
K\"ahler-Einstein metric on Fano manifolds, I: approximation of metrics with cone singularities, Journal of the American Mathematical Society,  \textbf{28} (2015), 183--197.

\bibitem{CDS2}
\textit{X.~X.~Chen}, \textit{S.~Donaldson} and \textit{S~Sun},
 K\"ahler-Einstein metric on Fano man- ifolds, II: limits with cone angle less than $2\pi$, Journal of the American Mathematical Society,  \textbf{28} (2015), 199--234.

\bibitem{CDS3}
\textit{X.~X.~Chen}, \textit{S.~Donaldson} and \textit{S.~Sun},
K\"ahler-Einstein metric on Fano man- ifolds, III: limits with cone angle approaches $2\pi$ and completion of the main proof, Journal of the American Mathematical Society,  \textbf{28} (2015), 235--278.

\bibitem{CW}
\textit{X.~X.~Chen} and \textit{Y.~Q.~Wang},
Bessel functions, Heat kernel and the Conical K\"ahler-Ricci flow, Journal of Functional Analysis, \textbf{269} (2015), 551--632.

\bibitem{CW1}
\textit{X.~X.~Chen} and \textit{Y.~Q.~Wang},
 On the long-time behaviour of the Conical K\"ahler- Ricci flows, Journal f\"ur die reine und angewandte Mathematik, \textbf{744} (2018), 165--199.

\bibitem{TC}
\textit{T.~Collins} and \textit{G.~Sz$\acute{e}$kelyhidi},
 The twisted K\"ahler-Ricci flow, Journal f\"ur die reine und angewandte Mathematik, \textbf{716} (2016), 179--205.

\bibitem{DWL1}
\textit{X.~Z.~Dai} and \textit{C.~L.~Wang},
Perelman's $W$-functional on manifolds with conical singularities, arXiv:1711.08443.

\bibitem{DWL2}
\textit{X.~Z.~Dai} and \textit{C.~L.~Wang},
Perelman's $\lambda$-Functional on Manifolds with Conical Singularities, The Journal of Geometric Analysis (2017), 1--33.

\bibitem{TDYR}
\textit{T,~Darvas} and \textit{Y,~ Rubinstein},
Tian's properness conjectures and Finsler geometry of the space of K\"ahler metrics, Journal of the American Mathematical Society, \textbf{30} (2017), 347--387.

\bibitem{JPD}
\textit{J.~P.~Demailly},
Regularization of closed positive currents and intersection theory, Journal of Algebraic Geometry, \textbf{1} (1992), 361--409.

\bibitem{SDINEW}
\textit{S.~Dinew},
Uniqueness in $\mathcal{E} (X,\omega)$, Journal of Functional Analysis, \textbf{256} (2009), 2113--2122.

\bibitem{SD2}
\textit{S.~K.~Donaldson},
K\"ahler metrics with cone singularities along a divisor. Essays in mathematics and its applications, Springer Berlin Heidelberg, (2012), 49--79.

\bibitem{SD2009}
\textit{S.~K.~Donaldson},
Discussion of the K\"ahler-Einstein problem, 2009, preprint. Available at http://www2.imperial.ac.uk/~skdona/KENOTES.PDF.

\bibitem{GEDWA}
\textit{G.~Edwards},
 A scalar curvature bound along the conical K\"ahler-Ricci flow, The Journal of Geometric Analysis, \textbf{28} (2018), 225--252.

 \bibitem{GEDWA1}
\textit{G.~Edwards},
Metric contraction of the cone divisor by the conical K\"ahler-Ricci Flow, Mathematische Annalen, (2017), 1--33.

\bibitem{GP1}
\textit{H.~Guenancia} and \textit{M.~P$\breve{a}$un},
Conic singularities metrics with perscribed Ricci curvature: the case of general cone angles along normal crossing divisors, Journal of Differential Geometry, \textbf{103} (2016), 15--57.

\bibitem{GuoSong1}
\textit{B.~Guo} and \textit{J.~Song},
 Schauder estimates for equations with cone metrics I,  arXiv:1612.00075.

\bibitem{GuoSong2}
\textit{B.~Guo} and \textit{J.~Song},
 Schauder estimates for equations with cone metrics II,  arXiv:1809.03116.

\bibitem{SYH}
\textit{S.~Y.~Hsu},
Uniform Sobolev inequalities for manifolds evolving by Ricci flow, arXiv: 0708.0893.

\bibitem{TJEF}
\textit{T.~Jeffres},
 Uniqueness of K\"ahler-Einstein cone metrics, Publ. Mat. \textbf{44} (2000), 437--448.

\bibitem{JMR}
\textit{T.~Jeffres}, \textit{R.~Mazzeo} and \textit{Y.~Rubinstein},
 K\"ahler-Einstein metrics with edge singularities, Annals of Mathematics, \text{183} (2016), 95-176.

\bibitem{WSJ1}
\textit{W.~S.~Jiang},
Bergman kernel along the K\"ahler-Ricci flow and Tian's conjecture, Journal f\"ur die reine und angewandte Mathematik, \textbf{717} (2016), 195--226.

\bibitem{K000}
\textit{S.~Ko{\l}odziej},
 The complex Monge-Amp\`ere equation, Acta mathematica,  \textbf{180} (1998), 69--117.

\bibitem{K}
\textit{S.~Ko{\l}odziej},
 H\"older continuity of solutions to the complex Monge- Amp\`ere equation with the right-hand side in Lipschitz: the case of compact K\"ahler manifolds, Mathematische Annalen, (2008), 379--386.

\bibitem{KKBV}
\textit{K.~Kr\"oncke} and \textit{B.~Vertman},
Perelman's Entropies for Manifolds with conical Singularities, arXiv:1902.02097.

\bibitem{NVK}
\textit{N.~V.~Krylov},
Lectures on elliptic and parabolic equations in H\"older spaces, American Mathematical Society, (1996).

\bibitem{LS}
\textit{C.~Li} and \textit{S.~Sun},
Conical K\"ahler-Einstein metric revisited, Communications in Mathematical Physics,  \textbf{331} (2014), 927--973.

\bibitem{JWL}
\textit{J.~W.~Liu},
 The generalized K\"ahler Ricci flow, Journal of Mathematical Analysis and Applications,  \textbf{408} (2013), 751--761.

\bibitem{JWLCJZ}
\textit{J.~W.~Liu} and \textit{C.~J.~Zhang},
The conical complex Monge-Amp\`ere equations on K\"ahler manifolds, Calculus of Variations and Partial Differential Equations, \textbf{57} (2018), 44.

\bibitem{LW}
\textit{J.~W.~Liu} and \textit{Y.~Wang},
The convergence of the generalized K\"ahler-Ricci flow, Communications in Mathematics and Statistics, \textbf{3} (2015), 239--261.


\bibitem{JWLXZ}
\textit{J.~W.~Liu} and \textit{X.~Zhang},
Conical K\" ahler-Ricci flow on Fano manifolds, Advances in Mathematics,  \textbf{307} (2017), 1324--1371.

\bibitem{JWLXZ1}
\textit{J.~W.~Liu} and \textit{X.~Zhang},
The conical K\"ahler-Ricci flow with weak initial data on Fano manifold, International Mathematics Research Notices, \textbf{17} (2017), 5343--5384.

\bibitem{MAZZ}
\textit{R.~Mazzeo},
K\"ahler-Einstein metrics singular along a smooth divisor, Journ\'ees \'Equations aux d\'eriv\'ees partielles, (1999), 1--10.

\bibitem{Nomura}
\textit{R.~Nomura},
Blow-up behavior of the scalar curvature along the conical K\"ahler-Ricci flow with finite time singularities, Differential Geometry and its Applications, \textbf{58} (2018), 1--16.

\bibitem{Ozuch}
\textit{T.~Ozuch},
Perelman's Functionals on Cones, The Journal of Geometric Analysis (2019), 1--53.

\bibitem{PSST}
\textit{D.~Phong}, \textit{N.~Sesum} and \textit{J.~Sturm},
Multiplier ideal sheaves and the K\"ahler-Ricci flow, Communications in Analysis and Geometry,  \textbf{15} (2007), 613--632.

\bibitem{PSSW}
\textit{D.~Phong}, \textit{J.~Song} \textit{J.~Sturm} and \textit{B.~Weinkove},
The K$\ddot{a}$hler-Ricci flow and the $\bar{\partial}$ operator on vector fields, Journal of Differential Geometry, \textbf{81} (2009), 631--647.

\bibitem{PS}
\textit{D.~Phong} and \textit{J.~Sturm},
On stability and the convergence of the K$\ddot{a}$hler-Ricci flow, Journal of Differential Geometry, \textbf{72} (2006), 149--168.

\bibitem{RUBIN}
\textit{Y.~Rubinstein},
Smooth and singular K\"ahler-Einstein metrics, Contemporary Mathematics, \textbf{630} (2014), 45--138.

\bibitem{LMSH1}
\textit{L.~M.~Shen},
Unnormalize conical K\"ahler-Ricci flow, Journal f\"ur die reine und angewandte Mathematik (2018), DOI: https://doi.org/10.1515/crelle-2018-0007.

\bibitem{LMSH2}
\textit{L.~M.~Shen},
$C^{2,\alpha}$-estimate for conical K\"ahler-Ricci flow, Calculus of Variations and Partial Differential Equations, \textbf{57} (2018), 33.

\bibitem{JSGT}
\textit{J.~Song} and \textit{G.~Tian},
The K\"ahler-Ricci flow through singularities, Inventiones mathematicae, \textbf{207} (2017), 519--595.

\bibitem{SW}
\textit{J.~Song} and \textit{X.~W.~Wang},
The greatest Ricci lower bound, conical Einstein metrics and the Chern number inequality, Geometry $\&$ Topology, \textbf{20} (2016), 49--102.

\bibitem{JSBW}
\textit{J.~Song} and \textit{B.~Weinkove},
Lecture notes on the K\"ahler-Ricci flow, arXiv:1212.3653.

\bibitem{GT94}
\textit{G.~Tian},
K\"ahler-Einstein metrics on algebraic manifolds, Transcendental methods in algebraic geometry (Cetraro 1994), Lecture Notes in Math, 143--185.

\bibitem{T1}
\textit{G.~Tian},
K-stability and K\"ahler-Einstein metrics, Communications on Pure and Applied Mathematics, \textbf{68} (2015), 1085--1156.

\bibitem{TianWang}
\textit{G.~Tian} and \textit{F.~Wang},
Cheeger-Colding-Tian theory for conic K\"ahler-Einstein metrics, arXiv:1807.07209.

\bibitem{TZZZ}
\textit{G.~Tian}, \textit{S.~J.~Zhang}, \textit{Z.~L.~Zhang} and \textit{X.~H.~Zhu},
Perelman's entropy and K\"ahler-Ricci flow on a Fano manifold, Transactions of the American Mathematical Society, \textbf{365} (2013), 6669--6695.

\bibitem{GTXHZ07}
\textit{G.~Tian} and \textit{X.~H.~Zhu},
Convergence of K$\ddot{a}$hler-Ricci flow, Journal of the American Mathematical Society, \textbf{20} (2007), 675--699.

\bibitem{GTXHZ13}
\textit{G.~Tian} and \textit{X.~H.~Zhu},
Convergence of the K\"ahler-Ricci flow on Fano manifolds, Journal f\"ur die reine und angewandte Mathematik, \textbf{678}  (2013), 223--245.

\bibitem{GTXHZ05}
\textit{G.~Tian} and \textit{X.~H.~Zhu},
Properness of Log $\mathcal{F}$-Functionals, arXiv: 1504.03197.

\bibitem{YQW}
\textit{Y.~Q.~Wang},
Smooth approximations of the conical K\"ahler-Ricci flows, Mathematische Annalen, \textbf{365} (2016), 835--856.

\bibitem{Yao1}
\textit{C.~J.~Yao},
Existence of weak conical K\"ahler-Einstein metrics along smooth hypersurfaces, Mathematische Annalen,  \textbf{362} (2015), 1287-1304.

\bibitem{Yao2}
 \textit{C.~J.~Yao},
The continuity method to deform cone angle, The Journal of Geometric Analysis, \textbf{26} (2016), 1155-1172.

\bibitem{RGY} 
\textit{R.~G.~Ye}, 
The logarithmic Sobolev inequality along the Ricci flow, Communications in Mathematics and Statistics, \textbf{2} (2014), 363--368.

\bibitem{QSZ} 
\textit{Q.~S.~Zhang},
A uniform Sobolev inequality under Ricci flow. International Mathematics Research Notices, (2007).

\bibitem{YSZ}
\textit{Y.~S.~Zhang},
A note on conical K\"ahler-Ricci flow on minimal elliptic K\"ahler surfaces, Acta Mathematica Scientia, \textbf{38} (2018), 169--176.

\bibitem{YSZ1}
\textit{Y.~S.~Zhang},
On a twisted conical K\"ahler-Ricci flow, Annals of Global Analysis and Geometry, (2017), 1--30.

\end{thebibliography}
\end{document}